\documentclass[11pt,letterpaper]{article}
\usepackage[margin=1in]{geometry}
\usepackage[T1]{fontenc}
\usepackage{lmodern,microtype}
\usepackage{amsmath,amssymb,amsthm,mathtools}
\usepackage{booktabs,array,graphicx}
\usepackage{float}
\usepackage{enumitem}
\usepackage{needspace}
\usepackage{aliascnt}
\usepackage[numbers,sort&compress]{natbib}
\usepackage[hidelinks]{hyperref}
\usepackage[nameinlink,capitalise,noabbrev]{cleveref}
\newcommand{\E}{\mathbb E}
\newcommand{\Prob}{\mathbb P}
\newcommand{\R}{\mathbb R}
\newcommand{\C}{\mathbb C}
\newcommand{\N}{\mathbb N}
\newcommand{\eps}{\varepsilon}

\newcommand{\one}{\mathbf 1}
\DeclareMathOperator{\tr}{tr}
\DeclareMathOperator{\rank}{rank}
\DeclareMathOperator{\diag}{diag}

\DeclareMathOperator{\Var}{Var}

\DeclareMathOperator{\supp}{supp}
\DeclareMathOperator{\id}{id}

\newcommand{\norm}[1]{\left\lVert#1\right\rVert}

\newcommand{\Lpnorm}[2]{\left\lVert#1\right\rVert_{L^{#2}}}
\newcommand{\fall}[2]{(#1)_{#2}}

\newtheorem{theorem}{Theorem}[section]
\newaliascnt{lemma}{theorem}
\newtheorem{lemma}[lemma]{Lemma}
\aliascntresetthe{lemma}
\newaliascnt{proposition}{theorem}
\newtheorem{proposition}[proposition]{Proposition}
\aliascntresetthe{proposition}
\newaliascnt{corollary}{theorem}
\newtheorem{corollary}[corollary]{Corollary}
\aliascntresetthe{corollary}
\theoremstyle{definition}
\newaliascnt{definition}{theorem}
\newtheorem{definition}[definition]{Definition}
\aliascntresetthe{definition}
\newaliascnt{example}{theorem}
\newtheorem{example}[example]{Example}
\aliascntresetthe{example}
\theoremstyle{remark}
\newaliascnt{remark}{theorem}

\aliascntresetthe{remark}
\setlist{itemsep=2pt,topsep=4pt}
\crefname{lemma}{Lemma}{Lemmas}
\Crefname{lemma}{Lemma}{Lemmas}
\crefname{proposition}{Proposition}{Propositions}
\Crefname{proposition}{Proposition}{Propositions}
\crefname{corollary}{Corollary}{Corollaries}
\Crefname{corollary}{Corollary}{Corollaries}
\crefname{definition}{Definition}{Definitions}
\Crefname{definition}{Definition}{Definitions}
\crefname{example}{Example}{Examples}
\Crefname{example}{Example}{Examples}
\crefname{remark}{Remark}{Remarks}
\Crefname{remark}{Remark}{Remarks}
\crefname{appendix}{Appendix}{Appendices}
\Crefname{appendix}{Appendix}{Appendices}

\hypersetup{
  pdftitle={Sharp Norms from Finite Structure: Graph Matrices and Structured Chaoses},
  pdfauthor={Huibo Xu, Shi Fu, Youming Qiao, Dacheng Tao}
}
\title{Sharp Norms from Finite Structure:\\Graph Matrices and Structured Chaoses}
\author{Huibo Xu$^{1}$ \quad Shi Fu$^{1}$ \quad
Youming Qiao$^{2}$ \quad Dacheng Tao$^{1}$\\[0.5em]
\small $^{1}$Nanyang Technological University, Singapore\\
\small $^{2}$University of Technology Sydney, Sydney, Australia}
\date{}
\begin{document}
\hypersetup{pageanchor=false}
\maketitle
\thispagestyle{empty}
\vspace{-2em}
\begin{abstract}
Graph matrices encode dependencies in random matrices built from shared
random variables. Their norm estimates are central to spectral algorithms,
sum-of-squares (SoS), and high-dimensional statistics: precision can
determine whether a signal dominates noise or a candidate moment matrix
remains positive semidefinite. We develop a structural norm theory that
connects a uniform sharp formula to extensions across models and to
algebraic feasibility and tensor-network spectra.
For every fixed simple graph shape in the dense Rademacher model,
including overlapping or empty matrix boundaries, we prove
\[
 \mathbb E\|M_\alpha\|
 =\Theta_\alpha\!\left(n^{(v+h-s)/2}(\log n)^{a_*/2}\right).
\]
Here $v$ counts vertices, $h$ counts isolated summation vertices, $s$ is
the minimum boundary-separator size, and $a_*$ maximizes, over minimum
separators, the number of components adjacent to the separator and meeting
neither surviving boundary. Thus, despite high-order dependencies, two
finite cut optimizations determine both growth exponents. The formula
resolves the sharp-norm question in this model, replacing the
polylogarithmic gap in separator estimates with an exact logarithmic
exponent and matching lower bounds. It reveals that local fluctuations
add at one cut and compete across cuts. An infinite family with identical
coarse graph parameters but different sharp norm scales exhibits the
structural information this rule detects.
The proof converts label losses in high-moment expansions into layerwise
separator excess, uniformly controlling every defect layer at growing
moment order. Conditional flattening and synchronized fluctuations give
the matching lower bound. Retaining the layers lets us track factor
incidence, scalar tails, and label dimensions separately, yielding sharp
estimates for specified independent-factor chaoses, local weights, and
unequal dimensions. Separate distributional comparisons cover bounded
asymmetric noise, Gaussian inputs, and fixed-degree Hermite inputs,
retaining growing-moment control.

Two applications use this precision together with additional structure.
For degree-four clique SoS in $G(n,1/2)$, critical log-free estimates and
an exact positive square give, with high probability, pseudoexpectations
satisfying all constraints simultaneously for $9\le k\le c\sqrt n$,
for an absolute $c>0$. This sharpens the Hopkins--Kothari--Potechin
quartic-correction construction to the square-root scale without an
asymptotic logarithmic loss. For fixed independent Gaussian tensor networks
with equal edge dimension, we obtain necessary and sufficient
deviation-convergence thresholds, sharp expected deviation scales, and
entropy estimates with bounded additive error. For connected, loopless
complex-Gaussian networks with nonempty input and output boundaries,
an additional flow-quotient comparison and auxiliary-dimension transfer,
combined with the tensor-GUE strong-convergence theorem of Chen,
Garza-Vargas, and van Handel, show that the appropriately rescaled largest
output-state eigenvalue converges in probability to the exact right edge
of the known limiting law, determining the constant correction to
min-entropy. The results connect finite structure to sharp growth scales,
and additional algebraic and spectral structure to full feasibility and
exact limiting constants.

\end{abstract}
\clearpage
\hypersetup{pageanchor=true}
\pagenumbering{arabic}
\section{Introduction}\label{sec:introduction}

Many proofs about computation on random inputs turn on a spectral
comparison: a random error must be smaller than a useful signal or a
positive matrix that supports a relaxation. The matrices in these proofs
often have polynomial entries built from the same random variables, so
their entries are strongly dependent. Graph matrices encode these
dependencies by finite graph templates. They are a central tool in
sum-of-squares (SoS) lower bounds for planted clique, the
Sherrington--Kirkpatrick model, tensor PCA, and the Wishart model of
sparse PCA \citep{BarakEtAl2016,GhoshEtAl2020,PotechinRajendran2023},
and more recently for non-Gaussian component analysis, with consequences
for robust statistics and mixture learning \citep{DiakonikolasEtAl2024}.
Their norm estimates help prove that candidate moment matrices are
positive semidefinite, and hence that a relaxation remains consistent
with an apparent solution.

Spectral algorithms for planted sparse vectors and random overcomplete
tensor decomposition also require bounds on dependent polynomial matrices
\citep{HopkinsSchrammShiSteurer2015}; graph-matrix techniques give
systematic proofs of key estimates in these analyses
\citep[Section~9]{AhnMedarametlaPotechin2016}. Here the norm controls the
random interference against which a signal must be detected. In both
settings, the precision of the estimate can determine the parameter range
in which the computational argument succeeds. This raises a basic
structural question: \emph{which features of a graph template determine
the exact growth of its matrix norm?} We determine that growth for every
fixed simple graph shape in the dense sign model, including the
logarithmic factors that distinguish a removable loss from a real
spectral fluctuation.

To see the object behind these applications, start with independent random
signs: for every unordered pair $\{i,j\}\subseteq[n]=\{1,\ldots,n\}$,
choose $\xi_{\{i,j\}}=+1$ or $-1$ with equal probability. Equivalently,
encode the edges and nonedges of $G(n,1/2)$ by positive and negative signs.
We call this the \emph{dense Rademacher model}. The word ``dense'' refers
to this ambient random input, not to the number of edges in a template.
The signs are independent across unordered pairs; using a pair in the
opposite order does not produce a new random variable.

A graph matrix uses a small graph as a recipe for multiplying and summing
these signs. For example, the path $a-x-b$ defines the matrix
\begin{equation}\label{eq:intro-path-recipe}
 P_{ij}=\sum_{k\notin\{i,j\}}
          \xi_{\{i,k\}}\xi_{\{k,j\}}\quad(i\ne j),
 \qquad P_{ii}=0.
\end{equation}
The endpoint labels $i,j$ specify the row and column; the internal label
$k$ is summed out. Each summand follows the two edges of the path.
Although the underlying signs are independent, different entries of $P$
reuse them and are dependent. More generally, a finite graph, called a
\emph{shape}, specifies which signs to multiply. Its designated row and
column vertices specify the matrix indices; the labels of all other
vertices are summed over, with distinct vertices assigned distinct labels.
The graph is a recipe for a dependent random matrix, not the large random
graph on $n$ vertices itself.
For polynomial matrix-valued functions with the appropriate relabeling
symmetry, grouping terms by shape gives a matrix analogue of a Fourier
expansion. One can then analyze the graph-indexed components and the
patterns produced by multiplying them
\citep{AhnMedarametlaPotechin2016,PotechinRajendran2023}.

Norm estimates turn this representation into a quantitative proof.
Consider, for illustration, a Hermitian block $B+E$, where $B$ is positive
definite and $E$ is a random error. The implication
\[
 \norm{B^{-1/2}EB^{-1/2}}<1
 \quad\Longrightarrow\quad B+E\succ0
\]
expresses the basic role of spectral control: the error must be small
relative to the available positive mass. Actual SoS arguments require
substantially more structure, including suitable factorizations,
preservation of positive terms, and treatment of kernels forced by the
constraints. Graph-matrix norm bounds control the fluctuations and
intersection errors that arise within these constructions. Their precision
can determine whether the resulting comparison closes at the desired
parameter scale \citep{PotechinRajendran2023,GhoshEtAl2020}.

Even a logarithmic loss can matter in such a comparison. In ellipsoid
fitting, refined spectral estimates for particular correlated matrices
enabled a construction to fit a constant multiple of $d^2$ independent
Gaussian points in $\mathbb R^d$, improving a guarantee with a
polylogarithmic loss \citep{HsiehEtAl2023}. Sharpening a matrix estimate
thus changed the scale of a geometric feasibility result.

But not every logarithm is a loss in the proof. A single-edge shape
already illustrates the distinction. If its two endpoints index the row
and column, it gives the symmetric sign matrix $\Xi$, with
$\Xi_{ij}=\xi_{\{i,j\}}$ for $i\ne j$ and $\Xi_{ii}=0$.
If instead one endpoint indexes \emph{both} row and column and the other
is summed out, the same edge gives a diagonal matrix $D$:
\begin{equation}\label{eq:intro-diagonal-recipe}
 D_{ii}=\sum_{j\ne i}\xi_{\{i,j\}},
 \qquad D_{ij}=0\quad(i\ne j).
\end{equation}
These two uses of the same one-edge graph have different norm scales:
\begin{equation}\label{eq:intro-log-contrast}
 \E\norm{\Xi}\asymp\sqrt n,
 \qquad
 \E\norm{D}
 =\E\max_i\left|\sum_{j\ne i}\xi_{\{i,j\}}\right|
 \asymp\sqrt{n\log n}.
\end{equation}
For the diagonal matrix, the operator norm selects the largest row sum:
an extreme local fluctuation produces the extra logarithm. Thus even the
number of vertices and edges does not determine the sharp scale. The
placement of the matrix indices matters. For a general shape, which
fluctuations can the norm select, and how much amplification do they cause?

A substantial theory already addresses these matrices. Separator-based
bounds identify their polynomial norm scale up to polylogarithmic factors
\citep{AhnMedarametlaPotechin2016}. More general matrix-chaos inequalities
express spectral bounds through coefficient flattenings, provide
mechanically computable parameters for combinatorial models, and remove
unnecessary dimensional factors in important examples
\citep{BandeiraEtAl2025}. Decoupling and linearization offer another
systematic approach through lower-degree polynomial matrices
\citep{TulsianiWu2025}. These developments make a finer structural question
natural: for each fixed graph shape, what is the exact logarithmic exponent,
and which feature of the shape forces it? This is the sharp classification
question recorded as Open Problem~31 \citep{NizicNikolac2025Problem31}.

We answer this question for every fixed finite simple graph shape in the
sign model just described, allowing overlapping or empty row and column
boundaries. A terminating finite combinatorial rule determines both
exponents in
\[
 \E\norm{M_\alpha}
 =\Theta_\alpha\!\left(n^{f(\alpha)}(\log n)^{g(\alpha)}\right).
\]
The answer is governed by a competition between cuts. The smallest
vertex cut between the two matrix boundaries determines the power of $n$,
with a correction for isolated summation vertices. Among those smallest
cuts, count the pieces that touch the cut but neither remaining boundary.
The largest such count determines the power of $\log n$. Thus the
classification identifies both the fluctuations the norm can select and
the ones it cannot amplify simultaneously. Matching lower bounds show
that the surviving logarithms are necessary. A finite graph calculation
therefore determines when a log-free estimate is possible and how much
amplification is unavoidable otherwise.

The proof develops a structural norm theory beyond graph edges. Its
layerwise counting theorem separates the effects of factor incidence,
scalar tails, and label-space dimensions. Together with matching
lower-bound arguments and distributional comparisons, it yields sharp
results for specified independent-factor, locally weighted, Gaussian,
and fixed-degree Hermite models. These extensions explain how the cut
rule changes when the random input or its geometry changes.

The applications carry this precision into two different conclusions.
For degree-four SoS, critical log-free estimates and an exact positive
square give full clique feasibility at $c\sqrt n$, removing the
logarithmic loss in the cited construction's scale. For Gaussian tensor
networks, uniform norm estimates give sharp deviation thresholds and
entropy bounds. A further comparison, through a flow quotient and an
auxiliary dimension, proves convergence of the largest eigenvalue to
the exact edge of the known limiting law. The paper thus connects
structural norm classification to algebraic feasibility and to spectral
endpoints; reaching an exact endpoint requires the additional mechanism
that preserves the leading constant.

\subsection{A finite rule for the sharp norm}

We now give the general definition behind \cref{eq:intro-path-recipe}.
\begin{definition}[Admissible shape and graph matrix]\label{def:graph-matrix}
Let $\alpha=(W,E,U,V)$ be a fixed finite simple graph with ordered row and
column boundaries $U,V\subseteq W$. Each boundary has distinct vertices,
but the two boundaries may overlap or be empty. These are precisely the
\emph{admissible shapes} in the sign theorem; no additional connectivity
condition is imposed. Write $v=|W|$, and use the independent uniform signs
defined above. Rows are indexed by injective labels of $U$, and columns
by injective labels of $V$. Thus $\mathbf i$ and $\mathbf j$ record the
labels assigned to the designated row and column vertices. The graph matrix is
\begin{equation}\label{eq:intro-graph-matrix}
 M_\alpha[\mathbf i,\mathbf j]
 =
 \sum_{\substack{\phi:W\hookrightarrow[n]\\
                 \phi|_U=\mathbf i,\ \phi|_V=\mathbf j}}
       \prod_{\{x,y\}\in E}\xi_{\{\phi(x),\phi(y)\}} .
\end{equation}
Here $\phi:W\hookrightarrow[n]$ assigns distinct labels to all vertices;
the boundary labels are fixed, and the others are summed over. We call this
the \emph{globally injective} model. There is no implicit normalization.
If the boundary labels cannot be extended to such an injection, the entry
is zero. In particular, assignments must agree on a common boundary vertex.
\end{definition}
The graph is fixed throughout; constants may depend on it, but not on $n$.
This convention is essential. A statement for every fixed graph is not a
uniform statement for graphs whose sizes grow with $n$.

The classification asks for graph-dependent exponents satisfying
\begin{equation}\label{eq:intro-question}
 \E\norm{M_\alpha}
 =\Theta_\alpha\!\left(n^{f(\alpha)}(\log n)^{g(\alpha)}\right).
\end{equation}
The missing logarithm can reflect a genuine extremal phenomenon rather
than slack in a concentration argument. The next statistic distinguishes
these possibilities.

\begin{definition}[Separators and active components]\label{def:separator}
A set
$S\subseteq W$ separates $U$ from $V$ if every path from $U$ to $V$
meets $S$; vertices of the boundaries themselves may be deleted.
In particular, every common boundary vertex belongs to every separator.
Let $s$ be the minimum size of such a separator. A component of
$\alpha-S$ is \emph{active} if it meets neither surviving boundary and is
adjacent to $S$. Let $a(S)$ be the number of active components, and set
\begin{equation}\label{eq:intro-invariant}
 s=\min_{S:U\mid V}|S|,
 \qquad
 a_*=\max_{\substack{S:U\mid V\\|S|=s}}a(S).
\end{equation}
\end{definition}
A component already disconnected from both boundaries before the cut is
not active. Such a component is a scalar random factor, with a different
role in moment bounds.

Let $h$ denote the number of isolated vertices outside $U\cup V$. Each
such vertex contributes a summation without a random edge. The
classification has the following form.
\begin{theorem}[Sharp graph-matrix norm]\label{thm:intro-graph}
For every fixed admissible simple shape $\alpha$, there are positive
constants $c_\alpha,C_\alpha,n_\alpha$ such that, for $n\ge n_\alpha$,
\begin{equation}\label{eq:intro-sharp}
 c_\alpha n^{(v+h-s)/2}(\log n)^{a_*/2}
 \ \le\ \E\norm{M_\alpha}\ \le\
 C_\alpha n^{(v+h-s)/2}(\log n)^{a_*/2}.
\end{equation}
\end{theorem}
The formula turns a spectral extremum over a growing matrix into two
optimizations on a fixed graph. First minimize the number
of deleted vertices: with $v,h$ fixed, each unit of separator size changes
the polynomial exponent by one half. Only \emph{among those minimum cuts} do we maximize
the active-component count: each counted piece contributes a half-power
of $\log n$. A larger cut cannot trade a loss in the power of $n$ for a
fixed logarithmic gain. The number of edges and the degrees of the pieces
do not enter the exponents separately; they matter through these cut
statistics and through the shape-dependent constants. This is the
reason that shapes with the same polynomial scale can have different
logarithmic behavior. Unlike an unspecified polylogarithmic upper bound,
the formula decides which logarithms must remain and certifies their size
by a matching lower bound.

The rule in \cref{eq:intro-invariant} is finite. Enumerate vertex subsets,
test separation, retain those of minimum size, and count their active
components. In particular, the logarithmic exponent is not defined through
a limit of moment optimizations or through a list of small graphs. The
same rule applies to a fixed graph of any size.

Detached components explain why adjacency to the cut is part of the
definition. In the auxiliary \emph{typed model}, each vertex role has a
separate label set and each edge its own independent array. A detached
component there multiplies every entry by the same random scalar: its
size contributes a power of $n$, but the norm does not maximize it over
$n$ independent choices. Counting it as active would introduce a spurious
logarithm. \Cref{sec:models} compares this typed model with the original
globally injective matrix; it does not factor the latter entrywise.

The formula recovers the contrast in \cref{eq:intro-log-contrast}.
For the sign matrix $\Xi$, the single-edge shape has $v=2$, $h=0$,
$s=1$, and $a_*=0$. For the diagonal row-sum matrix $D$, the same edge
has $v=2$, $h=0$, and $s=1$, but deleting the common boundary vertex
exposes one active component, so $a_*=1$. Even the polynomial scale and
minimum cut agree; the active count detects the square-root logarithm.

The distinction persists when the boundary placement and the usual
coarse graph parameters are held fixed. Take the path $u-x-y-v$ with
boundaries $u,v$ and add two leaves. Attaching both leaves to $x$ gives
$a_*=2$; attaching one to $x$ and one to $y$ gives $a_*=1$. Both trees
have six vertices, five edges, and minimum separator size one, but
\begin{equation}\label{eq:intro-chain-separation}
 \E\norm{M_{\rm concentrated}}\asymp n^{5/2}\log n,
 \qquad
 \E\norm{M_{\rm distributed}}\asymp n^{5/2}\sqrt{\log n}.
\end{equation}
In their equal-size typed coefficient representations, even the two
scalar flattening summaries $\sigma,\nu$ of
\citet{BandeiraEtAl2025} agree. \Cref{sec:separating-family} defines
these summaries and proves an infinite family of such separations.
The maximum over minimum cuts detects information those scalars lose:
the location at which several fluctuations can reinforce one another.

The same family shows what changes in the estimate. On a longer path,
attach $q_j$ leaves at internal site $j$, and put $Q=\sum_jq_j$.
In the typed model, summing leaves gives independent diagonal gates
between sign matrices. Multiplying their separate norm bounds pays
$(\log n)^{Q/2}$. The sharp rule instead pays
$(\log n)^{\max_jq_j/2}$: charges add at one cut and compete across cuts.
The gain can therefore grow with the number of decorated sites in this
family: separate maximization adds charges that the matrix cannot realize
together. The elementary product bound is sharp when all leaves are at
one site. \Cref{app:separation-comparison} also calculates specified direct
literature bounds on the same family. These comparisons concern those
outputs and the two scalar summaries, rather than the full flattening
profile or every possible refinement of earlier methods.

There is also an obstruction to reconstructing the answer from ordinary
expected traces. An isolated diagonal root with a detached edge and a
rooted triangle have identical expected trace moments at every order,
yet their expected norms differ by $\sqrt{\log n}$. Their diagonal
marginals agree; their joint dependence does not. The exact finite-$n$
identities and the norm comparison appear in \cref{sec:trace-obstruction}.
Thus the input to our rule is the shape and its boundaries, not merely
its dimension, rank, or expected trace sequence.

\subsection{A structural norm theory beyond graph edges}

The next contribution is to separate the sources of spectral growth.
Factor scopes specify which labels a random variable can see; scalar
tails determine the size of fluctuations selected at a cut; unequal role
sizes change the cost of that cut. We prove extensions that track these
effects in explicit structural models. The common mechanism is the
separator sequence in the high-moment expansion: it retains enough
information to reweight the analysis when the input changes. Gaussian
and fixed-degree Hermite comparisons also retain control over a growing
moment window, which is needed to pass from moment estimates to the
operator norm.

First replace each edge occurrence by a nonempty factor scope
$e\subseteq W$ with its own independent random array. A matrix entry
multiplies the selected array coordinates and sums over middle-role
labels. The two-section graph joins roles that share a scope; its
separators therefore record the factor incidence. Assume every role has
a label set of size comparable to $n$. For each factor occurrence, let
its coordinate law be fixed,
symmetric, and of variance one, and suppose
\[
 \Lpnorm{\xi_e}q\asymp_e q^{\theta_e/2},
 \qquad q\ge2,\qquad 0\le\theta_e\le2.
\]
These are two-sided, all-order assumptions on a law that does not change
with $n$. Define
\begin{equation}\label{eq:intro-tail-charge}
 b_*=
 \max_{\substack{S:U\mid V\\|S|=s}}
 \left(a(S)+\sum_{e\subseteq S}\theta_e\right).
\end{equation}
The second term records factors supported wholly on the separator.
Their extremes can contribute to the same label selection as the active
components.

\begin{theorem}[Sharp norm for the specified factor model]
\label{thm:intro-factors}
After the explicit preprocessing of unused roles and detached scalar
components, a fixed independent-factor shape with the laws above has
\[
 \E\norm{M}
 \asymp
 n^{(v+h-s)/2}(\log n)^{b_*/2}.
\]
The constants may depend on the shape, the comparison constants for the
role sizes, and the fixed factor laws.
\end{theorem}

The formula gives a precise interaction between structure and tails:
component fluctuations and heavy factors add at the same cut, and the
resulting charges compete across minimum cuts. Proving it requires a
refinement of the total count in \cref{eq:intro-count}. We retain the
separator sequence $(S_k)$, assign each layer its factor moments, and
only then sum over sequences. This reusable part of the argument
distinguishes minimum separators that tie in size but carry different
tail charges.

Symmetry can be dropped for bounded noise by a different argument.
For independent centered variance-one coordinates bounded by $K$,
\cref{thm:bounded-comparison} shows that replacing all $Q$ occurrence arrays by
signs changes every $L^p$ norm by at most $(2K)^Q$ in either direction.
This is standard symmetrization and contraction applied to the original
coefficient tensor, before simplifying its sign reference. It extends
the sharp sign scale to dense centered Bernoulli and bounded-erasure
factors, even with coordinate-dependent laws, but not to a shared
unordered-edge array without a separate comparison argument.

A separate result, \cref{prop:profile}, treats certain $n$-dependent laws
through a finite moment profile. The resulting moment-weighted cuts give
an upper bound, with an expectation consequence under the stated
cut-stability hypothesis. This is useful beyond a fixed distribution,
but is not a matching sparse-law theorem. Detached random factors retain
their actual moments in the trace, and an identity matrix factor retains
its dimension; the profile does not absorb either into an unspecified
constant.

Changing the dimensions affects the optimization itself, not just its
constants: deleting a role with more labels incurs a larger polynomial
cost, and a tie between cheapest cuts can determine the logarithm.
\Cref{thm:aspects} makes this precise for the typed independent-sign model
with original factor scopes of arity at least two, nonempty role sets of
size $m_x\asymp n^{\beta_x}$ for fixed $\beta_x\ge0$, and deterministic
real local amplitudes satisfying
\[
 c_e n^{\gamma_e}\le |a_{e,n}(x_e)|\le C_e n^{\gamma_e},
 \qquad 0<c_e\le C_e<\infty.
\]
Each amplitude sees only the coordinates of its own occurrence.
Write $\Gamma=\sum_e\gamma_e$. In the finite normal form of
\cref{sec:weights}, bounded roles are eliminated, boundary factors are
removed in the sign reference, and deterministic isolated roles and
detached random components are recorded separately. Let $B_+$ be the sum
of positive role exponents, $h_\beta$ the sum of exponents of deterministic
isolated roles, and $d_0$ the number of detached random components.
In the remaining core, put
\[
 \kappa=\min_{S:U\mid V}\sum_{x\in S}\beta_x,
 \qquad a_*=\max_{\sum_{x\in S}\beta_x=\kappa}a(S).
\]
The theorem gives the explicit sharp scale
\begin{equation}\label{eq:intro-weighted-scale}
 \E\norm H\asymp
 n^{\Gamma+(B_++h_\beta-\kappa)/2}(\log(2n))^{a_*/2}.
\end{equation}
Its logarithmic-window upper moments have the additional factor
$q^{d_0/2}$. The optimization retains every minimum-cost separator, including ties.
Bounded roles ($\beta_x=0$) are included through a finite reduction.
\Cref{sec:weights} explains both the reduction and a concrete transition
where a change of minimum face changes the logarithmic exponent.
For rational input exponents the statistics admit a terminating exact
evaluation procedure; arbitrary real exponents require a representation
permitting exact comparisons.

For Gaussian inputs, the next issue is the precision of the estimate
as the moment order grows. Fixed-order limits alone do not control the
operator norm; the extension retains the sharp scale throughout a
logarithmic moment window. Consider a fixed coefficient-one independent-factor
template with $Q$ nonempty scopes, comparable role sizes, and standard
real Gaussian coordinates; unary scopes are permitted. To express this
model through sign statistics, for each occurrence
adjoin a private middle role of size $n$ and replace the occurrence by a
sign factor on the enlarged scope. Let
$(\widehat f,\widehat g,\widehat d_0)$ be the sign norm and detached-component
statistics of this lifted template. \Cref{thm:gaussian-lift} gives
\begin{equation}\label{eq:intro-gaussian-window}
 \begin{aligned}
 \E\norm{H_G}&\asymp n^{\widehat f-Q/2}(\log(2n))^{\widehat g},\\
 \Lpnorm{\norm{H_G}}q&\asymp
 n^{\widehat f-Q/2}(\log(2n))^{\widehat g}q^{\widehat d_0/2},
 \qquad 2\le q\le C_0\log(2n).
 \end{aligned}
\end{equation}
The second statement is uniform throughout this window for each fixed
$C_0>0$ and all sufficiently large $n$, with constants allowed to depend
on $C_0$ and the fixed template. The same result holds for standard
circular complex coordinates. The private roles retain fluctuations
from unary factors that cannot be removed as sign isometries; detached
components supply the separate $q^{\widehat d_0/2}$ factor. This controls
the sharp scale, while exact edge constants require the additional
comparison described below.

Fixed-degree Hermite coordinates are covered by a separate all-moment
comparison. Replace occurrence $e$ by
$\operatorname{He}_{d_e}(g)/\sqrt{d_e!}$, using mutually independent
standard real Gaussian coordinates and fixed positive integers $d_e$.
Let $H_{\rm dec}$ replace that occurrence by $d_e$ independent Gaussian
occurrences on the same scope. With
$A=\prod_e d_e^{d_e/2}/\sqrt{d_e!}$, \cref{thm:hermite} states
\begin{equation}\label{eq:intro-hermite-comparison}
 A^{-1}\Lpnorm{\norm{H_{\rm dec}}}q
 \le\Lpnorm{\norm{H_{\rm He}}}q
 \le A\Lpnorm{\norm{H_{\rm dec}}}q,
 \qquad q\ge1.
\end{equation}
Consequently the norm and logarithmic-window moment exponents are those
of an explicitly expanded Gaussian template. The comparison is uniform
in dimension, not in the degrees. In particular, it accommodates fixed
degrees above two and does not require the Hermite coordinate itself to
have a symmetric law. This is a route beyond the separate tail theorem's
$\theta_e\le2$ assumption, not a relaxation of that assumption inside
its proof.

The same cut competition remains visible in concrete matrices.
The Gaussian gated chain in \cref{sec:extensions} recovers the
maximum-over-cuts rule of the sign leaf family through unary tail
charges. The gated Khatri--Rao product instead adds degrees at a shared
column role and has a short independent proof. These examples link the
distributional extension to the same structural phenomenon.

Two obstructions explain why the hypotheses carry mathematical content.
The coefficient example in \cref{sec:weights} has identical entry
magnitudes and random incidence but different operator scales: local
amplitude control is not a substitute for arbitrary coefficient geometry.
For a specified sparse iid family, \cref{prop:sparse-phase} records the
complete signed-Bernoulli phase diagram, including the
$(\log n/\log\log n)^{1/2}$ correction at density $1/n$ and the rare-occupancy
regime. This classical critical-density obstruction \citep{Seginer2000}
delineates the fixed-law domain. The positive statements above are
separate, proved extensions with their own hypotheses, not an assertion
that every combination of their assumptions defines one covered model.

\subsection{Application: degree-four SoS at the square-root scale}

The first application shows how removing a spectral logarithm changes
the scale of a computational lower bound. For the degree-four SoS
relaxation of clique in a dense random graph, we establish feasibility
through a constant multiple of $\sqrt n$. The candidate moment matrix
is a structured sum of graph-matrix blocks. Turning their estimates
into feasibility requires comparison with different positive scales
on different subspaces, together with all Boolean, nonedge,
normalization, and size constraints.

The quantitative conclusion of \cref{thm:sos} is simultaneous feasibility
through a constant multiple of the square-root scale: for some absolute
$c>0$, with probability at least $1-n^{-10}$ over $G\sim G(n,1/2)$,
\[
 \text{every real }9\le k\le c\sqrt n
 \quad\text{admits a feasible degree-four clique pseudoexpectation.}
\]
The construction uses the quartic correction of
\citet{HopkinsKothariPotechin2015}. Relative to the
$\widetilde\Omega(\sqrt n)$ scale in the cited degree-four works
\citep{HopkinsKothariPotechin2015,RaghavendraSchramm2015}, the issue is
the logarithmic loss, not a change of polynomial exponent. The statement
concerns this relaxation and includes its full constraints.

The proof separates the incidence space from its kernel and retains an
exact positive square before estimating errors. Only critical matching
and path configurations need log-free estimates; the others have enough
polynomial slack. A final algebraic extension enforces the full moment
constraints on one graph event. The critical estimates also have short
independent proofs, so this application illustrates the selective use of
sharp bounds rather than a logical need for every case of the classification.

\subsection{Application: tensor-network thresholds and exact spectral edges}

The second application asks how accurately a random tensor network
preserves lengths and how large an eigenvalue its output state can have.
We obtain necessary and sufficient deviation thresholds, entropy bounds
with bounded additive error, and, under the connected-network hypotheses
below, an exact limiting largest eigenvalue. Consider fixed Gaussian random tensor networks,
with independent standard circular complex coordinates at each tensor,
independent tensors, and a common edge dimension $N$. Every connected
component meets a terminal; the nontrivial deviation statement assumes
at least one tensor vertex. Internal edges, including parallel edges,
are distinct indices, and open edges end at distinct degree-one terminals.
Here normalization is especially consequential. If $H_N$ is the
contraction map, mean-Gram normalization chooses $M_N$ so that
$\E M_N^*M_N=I$. Sample-state normalization instead considers
$\rho_A=H_NH_N^*/\norm{H_N}_F^2$. These answer different questions: approximate
isometry concerns the former, while output-state entropies concern the
latter. Random tensor network methods and their spectral consequences
have been developed in several settings
\citep{HaydenEtAl2016,Hastings2016,ChengEtAl2024,
FitterLoulidiNechita2024}.

For equal edge dimension $N$, write $b$ for the number of input legs.
For a nonempty set $T$ of tensor vertices, let $c(T)$ be its internal
edge boundary and let $a(T),b(T)$ count its output and input legs. The
relevant signed cut gap is
\begin{equation}\label{eq:intro-rtn-gap}
 \Delta=\min_{\varnothing\ne T}\bigl(c(T)+a(T)-b(T)\bigr).
\end{equation}
Write $F_N=M_N^*M_N-I_{N^b}$ for the mean-Gram deviation. All Schatten
norms here are unnormalized. \Cref{thm:rtn} gives, for each fixed
$1\le t<\infty$ and $\Delta\ge0$,
\begin{equation}\label{eq:intro-rtn-scales}
 \E\|F_N\|_{S_t}\asymp_{\mathcal G,t}N^{b/t-\Delta/2},
 \qquad
 \E\|F_N\|_{\rm op}\asymp_{\mathcal G}N^{-\Delta/2}
 \quad(\Delta>0).
\end{equation}
For every sign of $\Delta$, the convergence criteria are both necessary
and sufficient:
\begin{equation}\label{eq:intro-rtn-criteria}
 \|F_N\|_{S_t}\xrightarrow{\Prob}0
 \ \Longleftrightarrow\ \Delta>2b/t,
 \qquad
 \|F_N\|_{\rm op}\xrightarrow{\Prob}0
 \ \Longleftrightarrow\ \Delta>0.
\end{equation}
For the sample-normalized state, the relevant statistic is instead the
ordinary minimum input-output edge cut $s_{\rm cut}$. It is distinct from
the signed gap $\Delta$. \Cref{thm:rtn-state} gives a constant
$C_{\mathcal G}$ such that, with probability
$1-O_{\mathcal G}(N^{-1})$,
\begin{equation}\label{eq:intro-rtn-entropy}
 \begin{aligned}
 N^{-s_{\rm cut}}&\le\|\rho_A\|_{\rm op}
                    \le C_{\mathcal G}N^{-s_{\rm cut}},\\
 s_{\rm cut}\log N-\log C_{\mathcal G}
 &\le H_\alpha(\rho_A)\le s_{\rm cut}\log N
                    \qquad(0\le\alpha\le\infty).
 \end{aligned}
\end{equation}
All entropy orders hold on the same event. Here $H_\alpha$ is the
R\'enyi entropy, with the von Neumann entropy at $\alpha=1$, log rank at
$\alpha=0$, and min-entropy at $\alpha=\infty$. The error is bounded
additively, not just negligible relative to $\log N$. The theorem thus
controls the largest eigenvalue and the entropy endpoints, rather than
only a fixed moment or a limiting empirical spectral distribution.
It uses concentration of the realized Frobenius normalization as well
as a cut-rank bound and the uniform Gaussian norm estimate.

The scale bound leaves a finer question: does the largest eigenvalue
actually approach the edge predicted by the limiting bulk spectrum?
Fixed normalized moments alone cannot exclude a bounded number of
outliers. We resolve this question for the connected, loopless,
independent complex-Gaussian network with $A,B\ne\varnothing$ and
common edge dimension $N$. Let $\mu_G$ be the compact
minimum-energy moment law of
\citet[Theorems~4 and~7]{FitterLoulidiNechita2024}, and let
$E_G=\sup\supp\mu_G$. Then \cref{thm:rtn-exact-edge} gives
\begin{equation}\label{eq:intro-rtn-exact}
 \begin{gathered}
 N^{s_{\rm cut}}\lambda_{\max}(\rho_A)
       \xrightarrow{\Prob}E_G,\\
 H_\infty(\rho_A)
       =s_{\rm cut}\log N-\log E_G+o_{\Prob}(1).
 \end{gathered}
\end{equation}
The new information is the exact right endpoint, not a new claim
to the already-known fixed-moment law. Even at the same minimum
cut, the constant can change: a single square Gaussian gate has
$s_{\rm cut}=1$ and $E_G=4$, whereas an $N^2$ by $N$ gate has
$s_{\rm cut}=1$ and $E_G=1$. The cut dimension alone cannot
distinguish their min-entropy corrections. The proof uses a flow-quotient
comparison with an independently constructed Gaussian circuit and the
tensor-GUE strong-convergence theorem of
\citet{ChenGarzaVargasVanHandel2026}; it is not a direct substitution into
the graph-matrix norm formula. The cut-gap, entropy, and exact-edge
statements retain their respective normalizations and hypotheses. In
particular, the narrower last result does not replace the common-event
entropy bound.

\subsection{Proof ideas: sharp scales, structured extensions, and exact edges}

Two obstacles govern the proofs. At the scale level, separately maximizing
local fluctuations can overcount the logarithm, while counting only
leading moment configurations misses the accumulated contribution of
defects. Separator layers resolve both issues and support the structured
extensions. At the exact-edge level, even a sharp estimate with an
unspecified constant loses the endpoint; a comparison at an auxiliary
dimension retains it. We describe these mechanisms below. The SoS
application combines the required scale estimates with the positive
square and subspace budget described above.

\paragraph{Counting every defect, not only the leading states.}
In $\E\tr((MM^*)^p)$, the same random input occurs in all $2p$ copies.
For signs, a label assignment survives precisely when each primitive
variable has even multiplicity. Record its exact equality partitions,
and call the loss in free labels relative to the maximum its defect
$\delta$. Each lost label costs a factor $n^{-1}$, but there can be many
ways to lose it. Counting only defect zero cannot control a moment whose
order grows like $\log n$. The key conversion associates intervals to
role partitions and reads a
separator $S_k$ at each of the $p-1$ integer layers. The label deficit is
exactly the accumulated excess cut size:
\begin{equation}\label{eq:intro-coarea}
 \delta=\sum_{k=1}^{p-1}(|S_k|-s).
\end{equation}
A backbone of disjoint boundary paths controls the path partitions;
an encoding of off-path components charges repairs to the same deficit.
The path count is calibrated by evaluating its exact positive expansion
at a dimension of order $p^2$, where independent-matrix moment bounds
apply; it does not require an explicit listing of states.
The encoding keeps the original backbone as well as its modifications,
so that reconstruction remains injective. Active components account for
the choices that remain free at minimum layers. For a core with $r$ roles,
the uniform estimate is
\begin{equation}\label{eq:intro-count}
 c_{\alpha,p,\delta}
 \le C_r^{\,2p}
 p^{\,a_*p+(3r^2+10r+2)\delta}.
\end{equation}
Here $c_{\alpha,p,\delta}$ counts legal exact states, $C_r$ depends only
on $r$, and the bound holds for every integer $p\ge2$ and $\delta\ge0$.
The important features are the coefficient $a_*$ on $p$ and a defect cost
independent of $p$. Summing the label-weighted counts gives a geometric
tail with ratio $p^{3r^2+10r+2}/n$. At $p\asymp\log n$, taking a $2p$-th
root turns $p^{a_*p}$ into the required $(\log n)^{a_*/2}$.
\Cref{sec:counting,app:counting-detail} give the expansion, intervals,
and lossless encoding. The layerwise refinement, rather than the scalar
count alone, is the input to the tail-charge extension.

\paragraph{Making all lower-bound fluctuations large at one label.}
Fix a minimum separator. Conditional flattening exposes products of
active-component fluctuations indexed by its labels. The norm must find
one label where \emph{all} components are large: multiplying their
separate maxima would not prove this. Moreover, candidate labels reuse
internal randomness. We first condition on all of that randomness,
retain a sufficiently large conditional good set, and only then expose
fresh crossing variables. A simultaneous exponential tilt supplies the
matching logarithmic gain. Detached components are handled as scalars
before taking core trace moments; the trace collision above explains why
this preprocessing cannot be omitted. \Cref{sec:models,sec:lower} establish
the transfers to the globally injective model and the synchronized lower
bound. Together with the count, they prove the classification.

\paragraph{Why the cut calculus survives changes of model.}
The extensions retain the separator sequence while changing its layer
weights. For independent factors, replace every scope by a clique in its
two-section graph. Factor legality then implies graph legality, but a scope
$e$ contributes its tail weight precisely on layers with $e\subseteq S_k$.
We therefore count states with the full sequence $(S_k)$ still visible.
Minimum layers contribute at most $b_*$ each, while nonminimum layers are
charged to separator excess. For the matching lower bound, conditional
flattening keeps the shared internal arrays fixed and exposes a crossing
occurrence for each active component, allowing one cut to realize both the
component and tail charges.

Unequal dimensions replace cardinality by the weighted cut cost
$\sum_{x\in S}\beta_x$, and weighted coarea again converts label loss into
accumulated excess. An integer weight vector preserving the entire minimum
face lets us replace each role by a fiber of clones. Minimum clone separators
use whole fibers and retain the same active-component maximum, so the original
count controls the weighted defects; a separate flattening gives the lower
bound at the actual dimensions. Bounded roles require a finite reduction.

For Gaussian factors, a private-role lift represents each occurrence by a
normalized sign sum. Positive trace comparison is uniform for
$p=O(\log n)$, and Banach-space contraction gives the reverse norm bound,
retaining the sharp scale and detached-scalar contribution. Polarization
similarly compares each fixed-degree Hermite coordinate with a product of
independent Gaussian occurrences in every $L^q$. These dimension-free
reductions retain their fixed-template and fixed-degree scope.
\Cref{sec:extensions,sec:weights,sec:gaussian} state the results, and
\cref{app:factor-detail,app:weights-detail,%
app:gaussian-detail} give the complete reductions.

\paragraph{Preserving an exact endpoint by changing the comparison dimension.}
For the RTN route, a maximum terminal flow defines a directed graph;
its full strongly connected components form a balanced acyclic quotient
$Q$. This is a comparison of Wick counts, not a claim that contracting
original tensors produces independent Gaussian gates. Let $F_G(p,N)$
be the normalized moment sum over all defects. The comparison bounds it
by $F_Q(p,p^2)/(1-p^{K_G}/N)$, with fixed $K_G=4v+2+8R$ and
$N>p^{K_G}$; here $v$ counts tensor vertices and $R$ counts internal
edges and open legs. The reference network uses the auxiliary dimension
$p^2$, not the original dimension $N$.

\Cref{thm:positive-fiber} states the underlying comparison for weighted
positive expansions: a single map with controlled target defect and
weighted fibers gives a geometric denominator. The global SCC map is
the substantive RTN input to that elementary inequality. The sign-path
calibration uses positivity directly, rather than the same fiber map.

The quotient is realized by independent Ginibre gates on potentially
overlapping tensor factors. Tensor-GUE strong convergence identifies
its limiting norm; concentration and a tail estimate turn this into
$F_Q(p,p^2)\le2(E_G+\eta)^p$ for fixed $\eta>0$ and large $p$.
Thus the comparison keeps the exponential base $E_G$, which a generic
shape-dependent constant would lose. Logarithmic moments exclude fixed
right outliers; concentrated fixed moments supply the matching lower
edge. \Cref{sec:rtn-edge} states the comparison and explains the mechanism;
\cref{app:edge-detail} proves both edge bounds and the passage through
the random Frobenius normalization.

\subsection{Dependencies and reading guide}

The proofs follow the distinction between a sharp scale and an exact
edge. \Cref{sec:models,sec:counting,sec:lower} develop the model,
separator coarea mechanism, and matching lower bound, with the separating
family in \cref{sec:separating-family}.
\Cref{sec:extensions,sec:weights,sec:gaussian} then explain how the
analysis changes with factor laws and label-space geometry. The gated
chain and columnwise-product examples illustrate the tail rule.
\Cref{sec:sos,sec:rtn} combine spectral estimates with the additional
algebraic and normalization arguments; \cref{sec:rtn-edge} develops the
separate constant-preserving comparison for the exact edge.

The detailed proofs are organized by this same dependency order in the
appendices, starting with the theorem-by-theorem guide in
\cref{app:proof-guide}. They include complete counting and encoding,
conditional probability estimates, the factor and dimension reductions,
the SoS algebra and positivity budget, and both RTN proof chains.
The columnwise-product example has an independent proof, and the
Wishart and sparse calibrations are also included. Thus a reader can
follow the main argument first and then check any theorem's complete
derivation without consulting a supplementary repository.
\label{end:introduction}

\section{Models and reductions}\label{sec:models}

A shape is a finite simple graph $G=(W,E)$ together with ordered,
internally distinct tuples $U,V$ of vertices. A vertex can occur in both
tuples. We identify a tuple with its underlying set when discussing
connectivity. All such shapes are admissible in the sign theorem.
An empty tuple indexes a one-dimensional coordinate space. No
symmetrization or normalization is implicit.

\begin{definition}[Typed sign model]\label{def:typed-model}
In the \emph{typed model}, each role $x\in W$ has a finite label set
$I_x$ of size $m_x$. Each edge $e=\{x,y\}$ has its own independent sign
array $\xi^e:I_x\times I_y\to\{-1,1\}$, and all its coordinates are
independent. The matrix $H_G$ sums over the labels of $W\setminus(U\cup V)$.
Agreement on common boundary roles is imposed exactly as in
\cref{eq:intro-graph-matrix}. Different roles do not share label
coordinates, even if their label sets have the same cardinality.
Balanced sizes mean $c_x n\le m_x\le C_x n$ for fixed positive constants.
\end{definition}

A separator meets every $U$--$V$ path, including paths of length zero.
For a separator $S$, let $a_G(S)$ count the components $K$ of $G-S$
such that $K\cap(U\cup V)=\varnothing$ and $N_G(K)\cap S\ne\varnothing$.
Set $s_G=\min_{S:U\mid V}|S|$ and
$a_G^*=\max_{|S|=s_G}a_G(S)$. These definitions apply to disconnected
graphs and empty boundaries.

\Needspace{7\baselineskip}
\begin{lemma}[Elementary reductions]\label{lem:reductions}
The following operations preserve the asserted separator statistics and
give the stated matrix factors.
\begin{enumerate}
\item An edge contained in $U$, or contained in $V$, can be removed in
the sign model by a diagonal sign isometry.
\item If $h$ isolated middle roles are deleted from a globally injective
shape on $v$ roles, the exact multiplier is
$\fall{n-(v-h)}h$.
\item In the typed model, connected components give a Kronecker product,
up to row and column permutations. A detached component is a scalar
factor, and an isolated middle role contributes $m_x$.
\item Deleting isolated middle roles and detached components changes
neither $s_G$ nor $a_G^*$.
\end{enumerate}
\end{lemma}
\begin{proof}
An edge on one boundary contributes the same sign to every entry in
its corresponding row or column. Its diagonal multiplier squares to
the identity. Such an edge cannot create a new boundary-to-boundary
path after a separator is deleted: the suffix after its last boundary
endpoint would already be such a path. It can merge only components
that meet that boundary, and therefore cannot change the active count.

For an injection of the remaining $v-h$ roles, exactly $n-(v-h)$
labels remain. Ordered injective choices for the deleted roles give
the falling factorial in the second assertion. In the typed model
the sums and products separate over connected components, and their
primitive arrays are independent. Finally, a minimum separator never
uses a vertex of a component that meets neither boundary: deleting
that vertex from the separator would still separate the boundaries.
Such a detached component is not adjacent to a minimum separator.
Isolated middle vertices have the same property.
\end{proof}

\begin{definition}[Boundary core]\label{def:boundary-core}
We call the graph remaining after the deletions in the fourth assertion
the \emph{boundary core}. Every component of the core meets a boundary,
and it has no isolated middle role. An empty core represents the scalar
$1$. Isolated boundary roles are retained.
\end{definition}

\begin{lemma}[Global upper transfer]\label{lem:color-upper}
Suppose a typed norm upper bound holds uniformly for $0\le m_x\le n$.
It transfers to the globally injective sign model with a constant
depending only on the number of roles. For $q\ge1$, the same is true of an $L^q$
bound, with no additional $q$-dependent constant.
\end{lemma}
\begin{proof}
For $v\ge1$, color every ambient label independently and uniformly
by a role of $W$. Let $H_c$ be the zero-padded matrix of injections
respecting the role colors. Each fixed injection survives with
probability $v^{-v}$, so, pointwise in the ambient signs,
\begin{equation}\label{eq:color-average}
 M_G=v^v\E_c H_c.
\end{equation}
For fixed colors the role classes are disjoint. Different edge roles
therefore use disjoint unordered ambient pairs, giving the typed
independence assumptions. Convexity proves the expected-norm claim;
Minkowski proves the $L^q$ claim. A zero class gives the zero matrix.
For $v=0$, the matrix is the scalar $1$.
\end{proof}

For the lower transfer, fix balanced disjoint role classes rather than
averaging over them. Boundary compression alone does not isolate
the desired middle-role assignments. The following Fourier extraction
does so.

\begin{lemma}[Global lower transfer]\label{lem:color-lower}
After isolated middle roles are removed, there is a positive
shape-dependent constant $c_G$ such that
$\Lpnorm{\norm{M_G}}q\ge c_G\Lpnorm{\norm{H_G}}q$ for every $q\ge1$,
where $H_G$ uses fixed balanced disjoint role classes.
\end{lemma}
\begin{proof}
Compress to rows and columns in the prescribed boundary classes.
For every target edge-color pair introduce an auxiliary sign and
multiply all ambient signs in that pair of classes by it. Extract the
Fourier coefficient containing every target edge-color sign. Fourier
extraction is an average of signed matrices, and each auxiliary sign
flip preserves the law of the ambient signs. Compression is
contractive.

There are exactly $|E|$ edges in each monomial. To survive this Fourier
coefficient, each of the $|E|$ target pairs must occur oddly and hence
exactly once. The induced map from roles to role colors is therefore
an edge-bijective endomorphism. Every nonisolated target role is in
its image. Isolated boundary roles are fixed by compression, and
there are no isolated middle roles. Thus the role map is surjective
and hence bijective. It is a boundary-fixing automorphism.

Reindexing the middle labels shows that every such automorphism gives
the same typed matrix. The extracted coefficient is consequently
$|\operatorname{Aut}_{U,V}(G)|H_G$. Apply Minkowski to the contractive
average. The automorphism count is positive and independent of $q$.
\end{proof}

These transfers do not assert independence of different connected
components in the globally injective model. Component independence is
used only in the typed model, before the transfers.

\begin{proposition}[Finite evaluation of the exponents]\label{prop:terminal}
For every fixed admissible shape, the pair
$((v+h-s_G)/2,a_G^*/2)$ is computable by a terminating finite procedure.
\end{proposition}
\begin{proof}
Record isolated middle roles and detached components, then enumerate
all subsets of the boundary core. Test separation by reachability,
retain every subset of minimum size, and count active components in
each. There are at most $2^{|W|}$ subsets and finitely many graph
operations per subset. The formulas use only these finite data.
The proof makes no claim of polynomial running time in the shape size.
\end{proof}

The norm classification cannot be recovered from the full trace sequence
alone. The collision example in \cref{sec:trace-obstruction} supplies the
complete construction and calculation; the obstruction is structural,
not a loss introduced by truncating the moment sequence.

\subsection{Equal coarse data, different logarithmic scales}
\label{sec:separating-family}

The active-component statistic detects information not contained in the
minimum separator size, or even in two standard scalar summaries of
coefficient flattenings. The following family isolates that information
without changing the ambient sign law or the matrix boundaries.

\begin{definition}[Leaf-decorated chain]\label{def:decorated-chain}
Fix $L\ge1$ and $\mathbf q=(q_1,\ldots,q_L)\in\N_0^L$.
The shape $\alpha_{L,\mathbf q}$ consists of the path
$x_0-x_1-\cdots-x_L-x_{L+1}$, with boundaries $U=(x_0)$ and
$V=(x_{L+1})$, and $q_j$ distinct leaves attached to $x_j$.
Write $Q=\sum_jq_j$, $q_{\max}=\max_jq_j$, $v=L+2+Q$, and
$d=L+1+Q$. These parameters are fixed before the label dimension grows.
\end{definition}

\begin{proposition}[Exact chain statistics and sharp scale]
\label{prop:chain-statistics}
The shape $\alpha_{L,\mathbf q}$ has $v$ vertices, $d$ edges, no isolated
middle role and no detached component. Its minimum separator size is one,
and its maximum active count over minimum separators is $q_{\max}$.
Consequently, in the globally injective unordered-edge sign model,
\begin{equation}\label{eq:chain-sharp}
 \E\norm{M_{\alpha_{L,\mathbf q}}}
 \asymp_{L,\mathbf q}
 n^{(L+1+Q)/2}(\log n)^{q_{\max}/2}.
\end{equation}
The same scale holds in the balanced typed model.
\end{proposition}
\begin{proof}
The unique boundary-to-boundary path is the displayed main path, so its
singletons are exactly the minimum separators. Deleting an internal
$x_j$ leaves two components meeting the respective boundaries and
$q_j$ isolated leaves adjacent to the cut. Only these leaves are active.
Deleting an endpoint exposes no active component. Hence $s=1$,
$h=0$, and $a_*=q_{\max}$. Substitute in \cref{thm:intro-graph};
the balanced typed assertion is established by its upper and lower
proofs in \cref{sec:counting,sec:lower}.
\end{proof}

To compare coefficient data, use the equal-size typed model. Its
deterministic coefficient tensor $\mathcal A$ has one noise axis of
size $n^2$ for each edge and two matrix axes of size $n$. Its entries
enforce equality of repeated role labels. In the conventions of
\citet[Section~2.2]{BandeiraEtAl2025}, let $\sigma(\mathcal A)$ be
the largest flattening norm with the matrix axes on opposite sides,
and let $\nu(\mathcal A)$ be the largest with both matrix axes on the
column side and a nonempty set of noise axes on the row side. We use
$\nu$ for that reference's $v(\mathcal A)$ to avoid a collision with
the vertex count. \Cref{lem:role-frontier} computes both quantities.

\begin{corollary}[Same coarse parameters, different sharp logarithms]
\label{cor:coarse-separation}
Fix $L\ge2$ and $r\ge1$. The concentrated decoration
$(Lr,0,\ldots,0)$ and the balanced decoration $(r,\ldots,r)$ have the
same $v,d,|U|,|V|,s,h$, and the same functions $\sigma,\nu$ for their
equal-size typed coefficient tensors. Nevertheless, their globally
injective expected norms have ratio
\begin{equation}\label{eq:concentrated-balanced}
 \Theta_{L,r}\bigl((\log n)^{(L-1)r/2}\bigr).
\end{equation}
For $L=2,r=1$, both trees have six vertices and five edges; their scales
are $n^{5/2}\log n$ and $n^{5/2}\sqrt{\log n}$, respectively.
\end{corollary}
\begin{proof}
Both have $Q=Lr\ge1$, and \cref{lem:role-frontier} gives
$\sigma=\nu=n^{(v-1)/2}$ in each case. All remaining equalities follow
from the definition. The maximum decorations are $Lr$ and $r$;
divide the two-sided estimates in \cref{eq:chain-sharp}.
\end{proof}

There is a concrete matrix interpretation of the difference. In the
typed model, summing the leaf labels first gives the exact identity
\begin{equation}\label{eq:typed-decorated-product}
 H_{\alpha_{L,\mathbf q}}
 =A_0\widetilde D_1A_1\cdots\widetilde D_LA_L,
 \qquad
 \widetilde D_j(i,i)=\prod_{h=1}^{q_j}
          \left(\sum_{a=1}^n\xi^{j,h}_{i,a}\right),
\end{equation}
where the $A_j$ and leaf arrays are mutually independent sign arrays.
An empty product is one. The globally injective matrix does not have
this independent product identity; \cref{lem:color-upper,lem:color-lower}
are the transfers between the models.

The direct literature bounds and an adapted elementary bound make
different uses of this structure. After removing the common polynomial
factor $n^{(v-1)/2}$, Theorem~20 of \citet{TulsianiWu2025}, evaluated
at logarithmic Schatten order, gives $(\log n)^v$.
The smaller direct output of Theorems~2.4 and~2.5 of
\citet{BandeiraEtAl2025}, as well as their graph-specific Theorem~4.8,
gives $(\log n)^{d/2}$ when $Q\ge1$. The complete substitutions are
in \cref{app:separation-comparison}.

Even without those general bounds, multiplying the individual matrix
and gate norms in \cref{eq:typed-decorated-product} gives
\begin{equation}\label{eq:chain-product-upper}
 \E\norm{M_{\alpha_{L,\mathbf q}}}
 \le C_{L,\mathbf q}n^{(L+1+Q)/2}(\log n)^{Q/2}.
\end{equation}
This bound is already sharp for concentrated decorations. For distributed
decorations it pays $(Q-q_{\max})/2$ extra logarithmic powers: it charges
the maximum at every gate separately. The separator formula charges
leaves together at one cut, then maximizes over the competing cuts.
Thus $Q$ is replaced by $q_{\max}$, not by zero.

\Cref{app:separation-detail} proves the coefficient formula, all
literature substitutions, and the elementary bound. The separation
concerns the two specified scalar flattening summaries of the typed
representation, not the complete flattening profile or the capabilities
of every adapted use of earlier methods. The Gaussian chain in
\cref{ex:gated-chain} will exhibit the same maximum-over-cuts rule
through separator-supported tail charges rather than leaf sums.

\section{Replica partitions and the upper bound}\label{sec:counting}

The trace method must retain two kinds of information at once:
the number of free labels and the number of states realizing that
freedom. Expand $\tr((MM^*)^p)$ using exact equality partitions of
the $2p$ replicas at each role. A surviving sign assignment contributes
once, with a falling-factorial label count. It is not weighted by the
number of Gaussian pairings that could refine its equality blocks.
\Cref{app:counting-detail} includes a complete fourth-moment example
illustrating this distinction.

The mechanism linking these states to graph structure is an interval.
With $\tau$ the cyclic row matching, an even partition $\pi$ receives
\begin{equation}\label{eq:intro-interval}
 I(\pi)=
 \bigl[\,|\pi|-|\tau\vee\pi|,\ p-|\tau\vee\pi|\,\bigr].
\end{equation}
Neighboring roles have compatible intervals, so each integer layer
defines a boundary separator. The coarea identity below equates the
total label defect with the sum of the excess separator sizes. We
prove this structural step here, explain how it pays for the encoding,
and retain the final moment-to-norm conversion. The complete path count,
off-path reconstruction, and seed-entropy calculation are in
\cref{app:counting-detail}.

\subsection{Exact expansion and the counting theorem}

Throughout this section $G$ is a boundary core on $r$ roles. Let
$p\ge2$, and index replicas by $\{0,\ldots,2p-1\}$. Define the matchings
\[
 \tau=\{\{1,2\},\{3,4\},\ldots,\{2p-1,0\}\},
 \qquad
 \eta=\{\{0,1\},\{2,3\},\ldots,\{2p-2,2p-1\}\}.
\]
Their join is the one-block partition. A partition is \emph{even}
if all its blocks have even size. We use $\pi\wedge\psi$ for common
refinement and $\pi\vee\psi$ for the equivalence relation generated
by their union.

\begin{definition}[Legal replica state]\label{def:legal-state}
A legal state assigns an even partition $\pi_x$ to each role, with
$\tau$ refining $\pi_x$ on $U$, $\eta$ refining $\pi_x$ on $V$,
and $\pi_x\wedge\pi_y$ even for each edge $xy$.
Let $|\pi|$ denote its number of blocks.
\end{definition}

\begin{lemma}[Exact positive expansion]\label{lem:exact-trace}
For the typed sign model,
\begin{equation}\label{eq:exact-trace}
 \E\tr((H_GH_G^*)^p)
 =\sum_{\boldsymbol\pi\ {\rm legal}}
       \prod_{x\in W}\fall{m_x}{|\pi_x|}.
\end{equation}
\end{lemma}
\begin{proof}
Expand the trace into label assignments on the $2p$ replicas.
Matrix multiplication imposes $\tau$ on row roles and $\eta$ on
column roles. At each role record the exact equality partition of
its labels. For an edge, the primitive coordinate classes are exactly
the cells of the meet of its endpoint partitions. Independence and
symmetry make its expected product zero unless every such cell is
even; if all are even the expectation is one. A nonisolated role
then has an even partition as a union of even cells. An isolated
core role is a boundary role and its trace constraint also forces
evenness. Assigning distinct labels to the blocks gives the falling
factorials. Each label assignment occurs once.
\end{proof}

\begin{definition}[Replica defect]\label{def:replica-defect}
Write $s=s_G$, $a_*=a_G^*$ and define the defect by
\begin{equation}\label{eq:defect}
 \sum_x|\pi_x|=p(r-s)+s-\delta.
\end{equation}
Let $c_{G,p,\delta}$ count the states of this defect.
\end{definition}

\begin{theorem}[Uniform all-defect count]\label{thm:count}
For all integers $p\ge2$ and $\delta\ge0$,
\begin{equation}\label{eq:uniform-count}
 c_{G,p,\delta}\le C_r^{2p}p^{a_*p+K_r\delta},
 \qquad K_r=3r^2+10r+2,
\end{equation}
where one may take
$C_r=200^r(2r)^{2r^2}4^{r^2}$ for $r\ge1$.
For $r=0$, take $C_0=1$ and $K_0=0$.
\end{theorem}

\subsection{Partition intervals}

\begin{definition}[Partition interval]\label{def:partition-interval}
For an even partition put
\[
 c(\pi)=|\tau\vee\pi|,\quad
 \ell(\pi)=|\pi|-c(\pi),\quad
 u(\pi)=p-c(\pi),\quad d(\pi)=p-|\pi|.
\]
Thus $I(\pi)=[\ell(\pi),u(\pi)]$ is an interval in $[0,p-1]$
of integer length $d(\pi)$.
\end{definition}

\begin{lemma}[Refinement and compatibility]\label{lem:interval}
If $\theta$ refines $\pi$ and both are even, then
$I(\theta)\subseteq I(\pi)$. If $\pi\wedge\psi$ is even, their
intervals intersect.
\end{lemma}
\begin{proof}
Coarsening $\theta$ to $\pi$ takes $|\theta|-|\pi|$ binary merges.
Each merge decreases the number of blocks after joining with $\tau$
by zero or one. Consequently
\[
 0\le|\tau\vee\theta|-|\tau\vee\pi|
       \le|\theta|-|\pi|.
\]
Substitution gives both endpoint inequalities. If the meet is even,
pair inside each of its cells to obtain a perfect matching $\rho$
refining both partitions. Its interval is the singleton
$\{p-|\tau\vee\rho|\}$, which is contained in both intervals.
\end{proof}

\begin{lemma}[Separator coarea identity]\label{lem:coarea}
For each legal state, the sets
\[
 S_k=\{x:\ell(\pi_x)<k\le u(\pi_x)\},\qquad 1\le k\le p-1,
\]
are separators, and
\begin{equation}\label{eq:coarea}
 \delta=\sum_{k=1}^{p-1}(|S_k|-s)\ge0.
\end{equation}
\end{lemma}
\begin{proof}
On $U$ we have $\ell=0$, and on $V$ we have $u=p-1$.
A vertex outside $S_k$ is either below the layer, with $u<k$,
or above it, with $\ell\ge k$. Compatible intervals intersect,
so no edge joins the below class to the above class. A surviving
row boundary vertex is below, and a surviving column boundary
vertex is above. Common boundary vertices have the full partition
and belong to every layer. This proves separation, including
length-zero paths. Each vertex belongs to exactly $d(\pi_x)$
layers. Summing and using \cref{eq:defect} gives the identity.
\end{proof}

\subsection{How coarea controls the encoding}

The identity above converts a loss of labels into a budget for larger
separator layers. To use that budget, choose vertex-disjoint paths
between the boundaries and first count the partitions on this backbone.
A path has no active component. Its count has only a constant raised
to the moment order, together with a polynomial cost for each lost
label; \cref{lem:path-count} proves this by comparison with a product
of independent sign matrices.

The reusable step is positive evaluation at a chosen dimension:
each exact state in one defect layer contributes a known minimum
weight, so a bound on the total moment bounds the number of states.
\Cref{lem:positive-calibration} isolates that step. Its positivity
concerns exact state weights, not coefficients after expanding
falling factorials into powers of the dimension.

The remaining components are encoded by a seed matching and local
changes from that seed. Boundary-touching components have prescribed
seeds. A boundary-free component can have a free seed only across
layers in which it is active. Coarea therefore charges its free
pairing choices to the active-component statistic, while every repair
is charged to the defect. The crucial counting direction is to bound
the range of possible repaired backbones and then the fibers over
each one. One must not assume that a repair is independent of the seed,
or reconstruct the original backbone from its repair alone.

\Cref{app:counting-detail} gives the path estimate, the deterministic
record and decoder, the seed-entropy bound, and the complete exponent
budget. Together these prove \cref{thm:count}; the conversion of that
count into an operator bound is short enough to retain here.

\subsection{From counting to a norm bound}

\begin{proposition}[Core upper bound]\label{prop:core-upper}
For fixed $C_0>0$, uniformly for $m_x\le n$ and
$2\le q\le C_0\log(2n)$,
\[
 \Lpnorm{\norm{H_G}}q
 \le C_{G,C_0}n^{(r-s)/2}(\log(2n))^{a_*/2}
\]
for sufficiently large $n$.
\end{proposition}
\begin{proof}
The exact expansion and \cref{thm:count} give
\begin{equation}\label{eq:geometric-trace}
 \E\tr((H_GH_G^*)^p)
 \le C_r^{2p}n^{p(r-s)+s}p^{a_*p}
       \sum_{\delta\ge0}(p^{K_r}/n)^\delta .
\end{equation}
The exact sum is finite; extending it is an upper bound. For any
fixed logarithmic window the ratio is at most $1/2$ eventually.
Choose $p=\max(2,\lceil\log(2n)\rceil,\lceil q/2\rceil)$.
The $2p$-th root bounds the $L^q$ norm and leaves a dimension
factor $n^{s/(2p)}\le e^{s/2}$. This proves the assertion.
\end{proof}

For a nonisolated detached connected typed component $K$ on $k$ roles,
orthogonality of distinct edge monomials gives
$\E Z_K^2=\prod_{x\in K}m_x$, so $\E|Z_K|\le n^{k/2}$.
These expectation factors multiply independently of the core.
Isolated middle roles contribute at most $n^h$.
Consequently the full typed expected norm is at most
$C_Gn^{(v+h-s)/2}(\log(2n))^{a_*/2}$.
\Cref{lem:color-upper} proves the upper half of
\cref{thm:intro-graph}. Detached factors are restored here in
expectation, not inside the large core trace moment.

\section{The matching lower bound}\label{sec:lower}

\subsection{Proof overview: synchronized separator fluctuations}

The polynomial exponent has a dimension-versus-rank interpretation.
The graph supplies summation variables, while a separator limits the
number of labels that must be shared by the row and column sides.
The resulting scale is $n^{(v+h-s)/2}$. This interpretation alone does not
explain the logarithm: two graphs with the same vertex count and separator
size can have different norm scales.

Fix a minimum separator $S$. After conditioning on labels in $S$, an
active component $K$ produces a scalar fluctuation $T_K(x_S)$. Several
active components produce a product
\[
 w(x_S)=\prod_{K\ {\rm active}}T_K(x_S).
\]
The operator norm can select a favorable separator label. A single
component has a square-root logarithmic extreme-value gain; $a(S)$
components can produce $a(S)$ such gains at the \emph{same} separator
label. This synchronization is the lower-bound mechanism. The proof
cannot replace a maximum of a product by a product of separate maxima:
those separate maxima might occur at incompatible labels.

There is another dependence issue. Candidate separator labels reuse the
internal randomness of each active component. The lower bound therefore
conditions on \emph{all} internal random variables before selecting a
family of fresh trials. On a sufficiently large conditional good set,
the remaining crossing variables yield independent trials with useful
success probabilities. An exponential tilt then makes all active
components simultaneously large. The order of these operations is part
of the proof, not a matter of presentation.

For the upper bound, the same logarithm appears as entropy in a trace
expansion. Near-leading replica configurations possess a limited number
of pairing choices that do not cost a free label. The active-component
statistic counts precisely this residual freedom. Configurations with
fewer labels incur a factor $n^{-\delta}$, where $\delta$ is their defect.
The crucial estimate makes their additional entropy only polynomial in
the moment order, with an exponent linear in $\delta$. At a moment order
of order $\log n$, the label loss then dominates the entropy increase.

We first work with balanced typed classes and remove isolated middle
roles. Fix a minimum separator $S$. Let $D$ be the union of its active
components, and write $C=W\setminus D$. The proof first exposes a
maximum over separator labels, then produces a simultaneous large
value of all active components.

\subsection{Conditional flattening}

\begin{lemma}[A lower stacking inequality]\label{lem:stacking}
Let $A_i$ be deterministic rectangular matrices and $\eps_i$
independent uniform signs. Then
\[
 \E\norm{\sum_i\eps_iA_i}\ge3^{-1/2}
 \max\left\{\norm{\sum_iA_iA_i^*}^{1/2},
             \norm{\sum_iA_i^*A_i}^{1/2}\right\}.
\]
For a decoupled degree-$q$ matrix chaos with coefficient tensor $A$,
every flattening that keeps its original row and column indices on
opposite sides satisfies
\begin{equation}\label{eq:stacking}
 \E\norm{X_A}\ge3^{-q/2}\norm{A_{[R\mid C]}}.
\end{equation}
The assertion also holds conditionally if $A$ is measurable with
respect to a sigma-field independent of the integrated sign groups.
\end{lemma}
\begin{proof}
For Hilbert-space vectors $v_i$, let $Z=\norm{\sum_i\eps_iv_i}$ and
$a^2=\sum_i\norm{v_i}^2$. Expansion gives $\E Z^2=a^2$ and
$\E Z^4\le3a^4$. Interpolation yields
$\E Z\ge(\E Z^2)^{3/2}/(\E Z^4)^{1/2}\ge a/\sqrt3$.
Apply this to the transpose of the random matrix acting on a top
eigenvector of $\sum_iA_iA_i^*$, and then to the transposed problem.
For the second claim, integrate one independent noise group at a
time. Choose vertical or horizontal stacking according to the side
of the prescribed flattening receiving that group. Each step costs
at most $\sqrt3$. Conditioning freezes the other coefficient data
and leaves exactly the same argument.
\end{proof}

Condition on all arrays with an endpoint in $D$. Summing each active
component produces a weight $T_K(x_S)$, and their product is
$w(x_S)$. All remaining edge arrays are independent of these weights.

\begin{lemma}[Exact separator flattening]\label{lem:separator-flattening}
For $q_0$ remaining edge types,
\begin{equation}\label{eq:weighted-flattening}
 \E_{\rm fresh}\norm{H_G}
 \ge3^{-q_0/2}
 \left(\prod_{x\in C\setminus S}m_x\right)^{1/2}
       \max_{x_S}|w(x_S)|.
\end{equation}
\end{lemma}
\begin{proof}
Assign the components of $G-S$ meeting $U$ to the row side and
those meeting $V$ to the column side. A component cannot meet both.
Assign any detached component to the row side. The active components
have already been removed into $w$. Put fresh edges within or
incident to a row-side component on the row side, and do the
analogous operation on the column side. Put edges contained in $S$
on the row side. Keep the original matrix indices on their
respective sides.

Let $R_0,C_0$ be the sets of roles observed by the resulting primitive
coordinates. Then $R_0\cup C_0=C$ and $R_0\cap C_0=S$.
For the union, no isolated middle role remains, and deleting $D$
cannot isolate a role outside $S$. For the reverse inclusion in
the intersection, use inclusion-minimality of $S$: for every
$z\in S$, there is a boundary-to-boundary path meeting $S$ only
at $z$. This path observes $z$ on both sides, through an edge
or through an original boundary coordinate. It cannot enter an
active component and then reach a boundary without returning to
$z$ or meeting another separator vertex. A common boundary role
is observed on both sides directly.

The coefficient tensor has value $w(x_S)$ when repeated role
coordinates agree, and zero otherwise. The assignment is recoverable
from these coordinates, so there is no additional summation
multiplicity. Delete consistency-zero rows and columns and permute
the rest. Its flattening is exactly
\[
 \bigoplus_{x_S}w(x_S)\,
             \one_{M_L}\one_{M_R}^{\,*},
 \quad
 M_L=\prod_{x\in R_0\setminus S}m_x,\quad
 M_R=\prod_{x\in C_0\setminus S}m_x .
\]
Its norm is $\sqrt{M_LM_R}\max|w|$. Apply
\cref{lem:stacking} conditionally to the fresh arrays.
\end{proof}

\subsection{From flattening to simultaneous witnesses}

The flattening reduces the lower bound to making all active-component
scalars large at one separator tuple. For an active component, condition
on its entire internal array. Second and fourth moments give a
positive-measure set of internal realizations on which the remaining
crossing sum has the correct variance. Higher fixed moments rule out a
single dominant coefficient. A weighted exponential tilt then gives a
polynomially small, rather than exponentially small, chance of reaching
the square-root logarithmic scale.

Choose separator tuples that are distinct in every separator coordinate.
Conditionally, their crossing arrays are disjoint; the component events
are also independent across components. Taking the tilt constant small
enough makes the combined success exponent smaller than the number-of-trials
exponent. Thus at least one tuple makes all factors large together.
\Cref{app:lower-detail} proves the conditional good-set estimates,
the tilt inequality, and this synchronization, with their quantifiers
and failure probabilities. No product of separate maxima is used.

\begin{proof}[Completion of \cref{thm:intro-graph}]
Insert \cref{eq:synchronized} into
\cref{eq:weighted-flattening}. If $r$ roles remain after deleting
isolated middle roles, the power of $n$ is
$(r-|D|-s)/2+|D|/2=(r-s)/2$. Choose a minimum separator
maximizing $a(S)$. This proves the balanced typed lower bound,
including detached components among the fresh arrays.
Apply \cref{lem:color-lower}, then restore
$\fall{n-(v-h)}h\asymp n^h$. The exponent becomes
$h+(v-h-s)/2=(v+h-s)/2$.
If there are no residual roles, the matrix is the scalar
$\fall nh$. Edgeless residual shapes consist of all-ones
factors and common-coordinate identity factors, and the same
formula is immediate. Together with the upper bound in
\cref{sec:counting}, this proves the theorem.
\end{proof}

\section{Sharp norms for independent factor chaoses}\label{sec:extensions}

Replacing graph edges by nonempty factor scopes puts tensors, repeated
independent interactions, and random diagonal gates in one incidence
model. The new logarithmic contribution comes from factors supported
entirely on a separator: their extremes add to the fluctuations of active
components outside the cut.

Three routes govern the distributions. Fixed symmetric tail-regular laws
admit sharp bounds with explicit charges. Bounded centered laws, without
symmetry or identical distributions, inherit the sign scale by
\cref{sec:bounded}. A finite moment profile gives dimension-dependent laws
an upper bound with a weighted cut cost. Each route requires independent
occurrences, even for equal scopes. The chain and columnwise-product
examples explain the charges; the sparse example delineates the fixed-law
conclusion.

\subsection{The factor model and its exact preprocessing}

\begin{definition}[Independent-factor chaos]\label{def:factor-chaos}
An independent-factor shape consists of a finite role set $W$,
boundaries $U,V$, and a finite list $\mathcal F$ of nonempty
scopes. Equal scopes can occur repeatedly; each occurrence has
its own independent array, with independent identically distributed
coordinates. Define
\begin{equation}\label{eq:factor-model}
 H=\sum_{x\in\prod_{v\in W}I_v}
       \prod_{e\in\mathcal F}\xi^{(e)}_{x_e}\,
                    e_{x_U}e_{x_V}^{\,*}.
\end{equation}
This is a typed Cartesian sum. No global injectivity restriction
or shared-array identification is implicit. The two-section graph
joins two distinct roles when they occur in one factor. Unary
factors mark a random singleton without adding an edge.
\end{definition}

Remove unused middle roles and then every factor-connected
component disjoint from both boundaries, including unary-only
singletons. If their scalar sums are $Z_1,\ldots,Z_J$, then
\begin{equation}\label{eq:factor-decomposition}
 H=d(n)\prod_{j=1}^JZ_j\,H_c,\qquad
 d(n)=\prod_{\text{unused middle }v}m_v .
\end{equation}
All random factors in this identity are independent. The core keeps
every boundary coordinate, including unused common-boundary
identity factors.

Write $r$ for the number of roles in $H_c$, and let $s$ be the
minimum separator size in its two-section graph. For such a separator,
$a(S)$ counts attached components meeting neither surviving boundary,
as in the graph model. These are the core statistics used below.

\subsection{Tail-regular factors}

Assume $m_v\asymp n$ and the fixed symmetric variance-one laws satisfy
$\|\xi_e\|_q\asymp_e q^{\theta_e/2}$ for every $q\ge2$, with
$0\le\theta_e\le2$. For a core separator define
$\Theta(S)=\sum_{e\subseteq S}\theta_e$ and
$b_*=\max_{|S|=s}(a(S)+\Theta(S))$.

\begin{theorem}[Fixed-law factor theorem]\label{thm:tail-factor}
With $v$ the original role count and $h$ the unused middle count,
\[
 \E\norm H\asymp n^{(v+h-s)/2}(\log n)^{b_*/2}.
\]
For every fixed $A>0$, the core satisfies
\[
 \Prob\{\norm{H_c}>
 C_A n^{(r-s)/2}(\log n)^{b_*/2}\}\le n^{-A}
\]
for all sufficiently large $n$.
\end{theorem}

The proof is in \cref{app:factor-detail}. Its counting input is a
layerwise strengthening of \cref{thm:count}, not an assumption
that a factor meet is equivalent to all its pairwise meets.
Only the forward implication is used for the upper bound.
The lower bound chooses one crossing occurrence for each active
component and conditions on the complete shared internal arrays.

Boundary factors with nonconstant absolute values are retained.
For example, an independent Gaussian factor supported on a
separator has $\theta=1$ and supplies a square-root logarithmic
charge. A sign factor there has $\theta=0$. Removing both as
diagonal isometries would therefore give the wrong theorem.

\begin{example}[A chain of independent Gaussian gates]\label{ex:gated-chain}
Consider $n\times n$ matrices $G_0,\ldots,G_L$ with iid standard
real Gaussian entries and diagonal gates
$D_j=\diag(\prod_{h=1}^{q_j}z_i^{(j,h)})$, where all matrix
entries and all $z_i^{(j,h)}$ are mutually independent standard
real Gaussians. Here $L\ge1$ and $q_j\ge0$ are fixed integers.
The path representation of $G_0D_1G_1\cdots D_LG_L$ has singleton
minimum separators, no active components after deleting them,
and unary charge $q_j$ at the $j$th internal role. Thus
\begin{equation}\label{eq:gated-product}
 \E\norm{G_0D_1G_1\cdots D_LG_L}
 \asymp n^{(L+1)/2}(\log n)^{\max_jq_j/2}.
\end{equation}
This example illustrates how separator-local tails, rather than
the sum of all gate degrees, determine the logarithm.
\end{example}
The contrasting columnwise product in \cref{sec:khatri-rao} places all
gates at one role, where their charges do add.
The sign leaf family in \cref{sec:separating-family} gives a complementary
realization: there leaf sums create the gates, whereas here unary
Gaussian occurrences carry the charges directly. Both optimize over
the same competing cuts. No equality in distribution or growing-order
central-limit replacement between the two models is used.

\subsection{Dimension-dependent laws and moment-weighted cuts}

A distinct result accommodates a finite moment profile for laws
that may depend on $n$. Retain comparable role sizes $m_v\asymp n$
and real symmetric variance-one coordinates, independent by coordinate
and occurrence as in \cref{eq:factor-model}. Suppose, for fixed
$C_0>0$, fixed nonnegative $\beta_e,\gamma_e$, and fixed constants
$C_e$ independent of $n$,
\[
 \E|\xi_e|^{2t}
 \le C_e^{2t}n^{\beta_e(t-1)}t^{\gamma_e(t-1)},
 \qquad 1\le t\le C_0\log n,
\]
for every integer $t$ in the displayed window and all sufficiently
large $n$. Put
$\kappa(S)=|S|-\sum_{e\subseteq S}\beta_e$ and
$b(S)=a(S)+\sum_{e\subseteq S}\gamma_e$.

\Needspace{14\baselineskip}
\begin{proposition}[Core profile bound]\label{prop:profile}
For every integer $2\le p\le C_0\log n$,
\begin{equation}\label{eq:profile}
 \E\tr((H_cH_c^*)^p)
 \le C^{2p}\sum_{(S_k)}
 n^{pr-\sum_k\kappa(S_k)}
 p^{\sum_k b(S_k)+K'_r\sum_k(|S_k|-s)},
 \qquad K'_r=3r^2+11r+2,
\end{equation}
where the sum ranges over separator sequences of length $p-1$.
If every minimizer of $\kappa$ has cardinality $s$, then, with
$\kappa_*=\min\kappa$ and
$b_*^{\rm mw}=\max_{\kappa(S)=\kappa_*}b(S)$,
\[
 \E\norm{H_c}\le
 Cn^{(r-\kappa_*)/2}(\log n)^{b_*^{\rm mw}/2}.
\]
A corresponding $n^{-A}$ upper tail holds for every fixed $A>0$
by changing the threshold constant within the same moment window.
\end{proposition}

For the full matrix one must multiply the trace bound by
$d(n)^{2p}\prod_j\E|Z_j|^{2p}$. The profile hypothesis alone does
not supply sharp $L^1$ estimates for arbitrary dimension-dependent
detached factors. No matching sparse-law lower bound is asserted.

\subsection{Why the structural rule survives the change of law}

Two inputs replace sign parity. First, occurrence-wise independence keeps
the trace expansion positive for symmetric laws. Each surviving factor
cell has a moment cost, and the layerwise count charges that cost only
to separators containing its whole scope. This gives the tail-weighted
cut statistic instead of a bound involving the largest scalar moment
everywhere. Second, the lower bound separates the internal component
fluctuations from the factors supported entirely on the separator.
The former use a conditional moderate deviation; the latter use the
assumed lower moment growth.

The distinction between expectation and high moments is also essential.
Detached components contribute their first absolute moments to the
expected norm, but each can add a square-root moment factor to a tail
bound. They must be restored after the core estimate, not at the trace
order chosen to detect its largest singular value.
\Cref{app:factor-detail} supplies the complete detached-component count,
fixed-law moment inequalities, layerwise refinement, weighted trace
sum, and conditional lower argument. It also proves the profile bound
without assuming that a dimension-dependent law is a fixed law.

\subsection{Sharp factor norms for bounded noise without symmetry}\label{sec:bounded}

Symmetry is needed for the positive replica expansion, but not for every
distributional extension of its conclusion. For bounded noise, a
coefficient-level comparison transfers the sign estimates before any
shape reduction. This also permits different laws at different coordinates.
The structural conclusion is stated first: the sign-reference exponents
continue to govern the expected norm and logarithmic-window moments.
We then prove the dimension-free comparison that transfers them. Its
method is standard symmetrization and contraction; the sharp reference
scale is the input supplied by the preceding theory.

\begin{corollary}[Sharp bounded-noise factor norms]\label{cor:bounded-factor}
Fix a typed independent-factor shape with scopes of arity at least two
and comparable role sizes. Replace its sign coordinates by mutually
independent centered variance-one coordinates bounded by a common
$K$, allowing their laws to vary with the coordinate and with $n$.
Compute $f,g,d_0$ in the sign reference: remove boundary-contained
sign factors, remove unused middle roles, split detached random
components, and use the minimum-separator statistics of the resulting
core. Here $d_0$ counts those detached random components.
Then
\[
 \E\norm{H_X}\asymp n^f(\log(2n))^g.
\]
For every fixed $C_0>0$ and $1\le p\le C_0\log(2n)$,
\begin{equation}\label{eq:bounded-factor-moments}
 \Lpnorm{\norm{H_X}}p
 \asymp n^f(\log(2n))^g p^{d_0/2}.
\end{equation}
The implicit constants depend only on the fixed shape, $K$, the
role-size comparison constants, and $C_0$. For each fixed $A>0$,
\[
 \Prob\{\norm{H_X}>C_A n^f(\log(2n))^{g+d_0/2}\}\le n^{-A}
\]
for all sufficiently large $n$.
\end{corollary}

\begin{proof}
Use rectangular matrices with operator norm as $B$, and apply
\cref{thm:bounded-comparison} to the coefficient tensor of the original
factor network. The number $Q$ counts occurrences \emph{before} any
boundary reduction. The sign norm and moment estimates proved above
then give the expectation and \cref{eq:bounded-factor-moments}.
Markov's inequality at $p=\log(2n)$ gives the tail after enlarging
$C_A$. The bounded-noise matrix itself is never simplified by deleting
boundary factors: these factors need not be isometries and may vanish.
\end{proof}

The comparison used in this proof applies to the same coefficient tensor
under two laws. Stating it in a Banach space makes clear that its constants
do not depend on the number of matrix rows, columns, or primitive variables.

\begin{theorem}[Dimension-free bounded-noise comparison]\label{thm:bounded-comparison}
Let $B$ be a real Banach space, let $Q\ge0$ be an integer, and let
$J_1,\ldots,J_Q$ be finite sets. For deterministic $B$-valued coefficients set
\[
 F(X)=\sum_{i_t\in J_t}a_{i_1,\ldots,i_Q}
                         \prod_{t=1}^Q X_{i_t}^{(t)}.
\]
Assume all primitive coordinates are mutually independent and satisfy
$\E X_i^{(t)}=0$, $\E(X_i^{(t)})^2=1$, and $|X_i^{(t)}|\le K$ almost surely,
where $K\ge1$ is common to all coordinates. No symmetry or identical-law
assumption is imposed. If $F(\eps)$ uses independent uniform signs and the
same coefficient tensor, then, for every $1\le p<\infty$,
\begin{equation}\label{eq:bounded-comparison}
 (2K)^{-Q}\Lpnorm{F(\eps)}p
 \le \Lpnorm{F(X)}p
 \le (2K)^Q\Lpnorm{F(\eps)}p.
\end{equation}
The constants are independent of $p$, the index-set sizes, and $B$.
For $Q=0$ the comparison is equality.
\end{theorem}

The complete proof is given in \cref{app:bounded-detail} (page~\pageref{app:bounded-detail}).

For example, if $Z\sim\operatorname{Bernoulli}(\vartheta)$ and
$\eta\le\vartheta\le1-\eta$ for fixed $0<\eta\le1/2$, then
$(Z-\vartheta)/\sqrt{\vartheta(1-\vartheta)}$ satisfies the hypothesis
with $K=\sqrt{(1-\eta)/\eta}$. The parameters may vary by coordinate
and dimension within this interval. Independent erasures
$\eps Z/\sqrt{\vartheta}$ with $\vartheta\ge\eta>0$ also qualify.

This comparison is a transfer between laws for the \emph{same}
coefficient tensor, not a classification of arbitrary tensors by
incidence. It requires separate linearity and independent groups;
it neither replaces $X_i^2$ by an independent product nor transfers
the theorem automatically to a globally shared unordered-edge array.
If the density tends to zero, $K$ can diverge and the constants in
\cref{eq:bounded-comparison} can change polynomial exponents.

\subsection{A sparse boundary case: the complete iid phase diagram}
\label{sec:sparse}

The last restriction has a concrete consequence even for a single
matrix. The following phase diagram concerns one specified iid family,
not all sparse factor networks. It is an elementary consequence of
Seginer's row/column norm comparison and binomial occupancy estimates.
The critical-density obstruction already appears in the introduction
of \citet{Seginer2000}; we include it to delineate the fixed-law theory,
not as a new sparse-matrix phenomenon.

\begin{proposition}[Sparse signed Bernoulli phase diagram]\label{prop:sparse-phase}
Let $X$ be $n\times n$ with independent entries
$X_{ij}=\rho^{-1/2}\eps_{ij}B_{ij}$, where the signs and
$B_{ij}\sim\operatorname{Bernoulli}(\rho)$ are mutually independent.
For fixed $\beta\ge0$ put $\rho=n^{-\beta}$. As $n\to\infty$,
\begin{equation}\label{eq:sparse-phase}
 \E\norm X\asymp_\beta
 \begin{cases}
 n^{1/2},&0\le\beta<1,\\
 n^{1/2}(\log n/\log\log n)^{1/2},&\beta=1,\\
 n^{\beta/2},&1<\beta\le2,\\
 n^{2-\beta/2},&\beta>2.
 \end{cases}
\end{equation}
\end{proposition}

The full proof, including expectation tails in every regime, is in
\cref{app:sparse}. At $\beta=1$, no fixed pair $f,g$ represents this
scale as $\Theta(n^f(\log n)^g)$: successive logarithmic comparisons
force $f=g=1/2$, leaving a ratio $(\log\log n)^{-1/2}$ tending to zero.
For $\beta>2$, the matrix is zero with probability tending to one;
the expectation instead comes from a rare occupied entry. Neither
behavior contradicts \cref{thm:tail-factor}, whose laws are fixed
independently of $n$. Nor does this example establish that its positive
class is maximal. Here $\E X_{ij}^4=\rho^{-1}$ already determines the
density, so occupancy is not additional information missing from the
full permitted moment data.

\subsection{Worked example: a gated columnwise tensor product}\label{sec:khatri-rao}

The chain in \cref{eq:gated-product} and the columnwise product below
isolate two different geometries of the same tail rule. Along the chain,
different gates sit at different singleton cuts, so the largest gate
degree determines the logarithm. In the columnwise product, all gates
share one column role and their charges add at that cut. The following
independent proof makes this distinction visible directly in the matrix,
through a selected column and a positive-semidefinite Schur product.

\begin{theorem}[Gated Khatri--Rao product]\label{thm:khatri-rao}
Fix integers $r,q\ge1$. Let $A_1,\ldots,A_r$ be independent $n\times n$
real standard Gaussian matrices. Their columnwise tensor product
$K\in\R^{n^r\times n}$ is defined by
\[
 K[:,j]=A_1[:,j]\otimes\cdots\otimes A_r[:,j].
\]
Let $D$ be diagonal with $D_{jj}=\prod_{h=1}^q g_{jh}$, where all
$g_{jh}$ are independent standard real Gaussians independent of the
matrices. Then
\begin{equation}\label{eq:khatri-rao}
 \E\norm{KD}\asymp_{r,q} n^{r/2}(\log(2n))^{q/2}.
\end{equation}
\end{theorem}

The complete proof is given in \cref{app:khatri-detail} (page~\pageref{app:khatri-detail}).

In the factor representation, the column role is the center of a
star, the $r$ row roles are leaves, and the center carries $q$ unary
Gaussian gates. For $r\ge2$ the center is the unique minimum separator.
For $r=1$ both endpoints are minimum separators, with charges $q$ at
the gated endpoint and zero at the other. Maximizing over the full
minimum-separator family gives the logarithmic exponent $q/2$ in
both cases. Thus the maximum over cuts in the chain and the sum of
charges at one cut in the star are two manifestations of the same
separator-supported tail statistic. The proof in \cref{app:khatri-detail}
establishes this example independently of the general counting argument.

\section{Local weights and polynomial aspect ratios}\label{sec:weights}

The next extension changes the geometry of the label spaces and the
deterministic amplitudes. It is useful to separate these two changes.
A role with $n^{\beta_v}$ labels contributes a cost $\beta_v$ to a cut;
a local amplitude of size $n^{\gamma_e}$ contributes its exponent to the
entire matrix. Thus the polynomial scale is governed by a weighted
separator, but its cost alone still does not determine the logarithm.
All minimizing separators, including ties, must be retained.

The weight restriction is local: each amplitude sees only the labels
of its own factor, and its magnitude has fixed two-sided bounds.
This permits an all-moment contraction argument. It does not permit an
arbitrary coefficient tensor on the full set of roles. We prove the
positive-aspect result first and then eliminate bounded roles by a
finite Fourier argument. This last step is needed for a closed statement
at nonnegative, rather than strictly positive, aspects.

This is a sign-model extension, separate from the tail-law theorem in
\cref{sec:extensions}. We do not silently combine every hypothesis of
the two theorems into a single larger model.

\subsection{Weighted minimum cuts and the sharp formula}

\begin{definition}[Locally weighted sign model]\label{def:weighted-model}
A separate sign-model extension, with all original factor scopes
of arity at least two and nonempty role sets, allows
$m_v\asymp n^{\beta_v}$ with $\beta_v\ge0$ and deterministic
real local amplitudes, with fixed $\gamma_e\in\mathbb R$ and
fixed constants $0<c_e\le C_e<\infty$, satisfying
$c_en^{\gamma_e}\le|a_{e,n}(x_e)|\le C_en^{\gamma_e}$.
The amplitudes depend only on their own factor coordinates.
Write $\Gamma=\sum_e\gamma_e$.
\end{definition}

\begin{definition}[Weighted core and minimum-face statistics]\label{def:weighted-core}
Restrict roles
and boundaries to $P=\{v:\beta_v>0\}$ and restrict each scope to $P$.
Discard empty scopes and scopes contained in one boundary, retaining
their exponents in $\Gamma$. Unary factors on a middle role are
absorbed into an incident higher-arity sign array when one exists.
A graph-isolated middle role with unary factors is one detached
random scalar; without any factor it is deterministic. Let
$B_+=\sum_{v\in P}\beta_v$, let $h_\beta$ be the sum of exponents
of deterministic isolated roles, and let $d_0$ be the number of
detached random components. In the remaining core set
\[
 \kappa=\min_{S:U\mid V}\sum_{v\in S}\beta_v,\qquad
 a_*=\max_{\beta(S)=\kappa}a(S).
\]
\end{definition}

\begin{theorem}[Locally weighted aspect theorem]\label{thm:aspects}
For the fixed typed independent-sign model just specified,
\[
 \E\norm H\asymp n^f(\log(2n))^{a_*/2},
 \qquad f=\Gamma+\frac{B_++h_\beta-\kappa}{2}.
\]
For fixed $C_0>0$ and $2\le q\le C_0\log(2n)$,
\[
 \Lpnorm{\norm H}q
 \le Cn^f(\log(2n))^{a_*/2}q^{d_0/2}.
\]
Consequently the upper-tail logarithmic exponent is
$(a_*+d_0)/2$. For rational input exponents the displayed
statistics have a terminating exact evaluation procedure.
\end{theorem}

The proof in \cref{app:weights-detail} establishes each reduction. In particular, bounded
roles are eliminated by a finite Fourier projection, not by
selecting a summand and assuming that its norm is a lower bound.
For positive aspects, an integer surrogate preserves all
minimum-cost separators, and a clique expansion transfers the
counting theorem. Constants can depend on this expansion; no
uniformity near changing minimum faces is claimed. For arbitrary
real input exponents, effective evaluation requires an exact
comparison representation.

\subsection{Why the whole minimum face matters}

Two features distinguish this result from substituting dimensions into
the equal-size formula. The active count must be maximized over the
\emph{entire} minimum-cost face, including ties. A rational surrogate and
true-twin expansion preserve that face in the counting proof. Also,
$\beta_x=0$ is covered by a finite Fourier reduction, not by a limiting
argument that still needs polynomially many fresh labels. Local amplitudes
are removed by all-moment contraction; the two-sided local hypothesis
is what permits this step.

\begin{example}[A change of minimum-separator face]\label{ex:minimum-face}
Consider the path $u-x-v$ with a leaf attached at $x$, with row boundary
$u$ and column boundary $v$. Give $u,v$, and the leaf exponent one, and
give $x$ exponent $1+\varepsilon$, where $\varepsilon>-1$ is fixed.
The only singleton cuts are the two endpoints and $x$. If
$\varepsilon>0$, the endpoint cuts minimize cost and expose no active
component. If $\varepsilon=0$, all three cuts minimize cost, and the
middle cut exposes the leaf. If $\varepsilon<0$, the middle cut is the
unique minimum. The logarithmic exponent is consequently zero in the
first regime and one half in the other two.
\end{example}

This change is not a contradiction with continuity of a matrix at a
fixed dimension. The aspect vector is fixed before $n$ tends to infinity.
For fixed positive $\varepsilon$, the extra cost $n^{-\varepsilon/2}$
eventually outweighs the logarithmic gain. The theorem is therefore
pointwise in the aspect vector, not uniform across a changing minimum
face. The minimum-face reduction preserves all ties instead of selecting one
convenient minimum cut.
The proof has three logically distinct reductions. Local contraction
removes the amplitudes in every moment norm with constants independent
of the moment order. For positive aspects, an integer surrogate preserves
the entire minimum-separator face, and cloning roles converts its
weighted defect into the ordinary counting problem. Finally, bounded
roles are eliminated by Fourier projection and unary signs are absorbed
by an equality in distribution. These last operations cannot be replaced
by keeping an arbitrary summand of the matrix.
The complete argument, including the unequal-dimension conditional
estimate and the finite evaluation procedure, is in
\cref{app:weights-detail}.

\subsection{Scope of the coefficient restriction}

The local-weight hypothesis is substantive. Let
$Z_n=\sum_{k,l}\eps_{kl}$ and $H_n=Z_nF_n$, where $F_n$ is a
deterministic $n$ by $n$ sign matrix. Then $\E|Z_n|\asymp n$.
For $F_n=J_n$, $\E\norm{H_n}\asymp n^2$.
For an $n$ by $n$ compression of a Sylvester Hadamard matrix
of least order $N\ge n$, one has
$\sqrt n\le\norm{F_n}\le\sqrt{2n}$, and hence
$\E\norm{H_n}\asymp n^{3/2}$. The entry magnitudes and random
incidence are identical in these two constructions.

Choosing the first coefficient matrix for even $n$ and the second
for odd $n$ gives no fixed power-log representation at all.
This does not rule out richer coefficient-sensitive invariants.
It does show why a structural factor theorem cannot be restated
as a theorem for arbitrary global weights.

\section{Gaussian and Hermite factor chaoses}\label{sec:gaussian}

This section supplies the uniform Gaussian input used later for the
operator-norm and sample-state endpoints in \cref{sec:rtn}.
It concerns coefficient-one incidence sums with independent
occurrence arrays and comparable role sizes. Zero-arity
Gaussian factors are excluded; they can instead be multiplied
separately with their own scalar moments. Unary factors are
allowed.

The Gaussian case also belongs to the fixed-law factor theory, but a
private-role lift gives a useful additional conclusion: it transfers the
full logarithmic-window moment statement, including the distinction
between core and detached fluctuations. It provides the precise uniform
input needed for sharp operator-norm scales; fixed-order Gaussian moments
alone do not provide that input. Exact RTN edge constants require the
separate comparison in \cref{sec:rtn-edge}. Hermite coordinates are then
handled by a dimension-free polarization comparison.

\subsection{The uniform Gaussian theorem}

For each of the $Q$ occurrences on scope $e$, add one private middle role
of size $n$ and replace the occurrence by an independent sign factor on
the enlarged scope. Denote this lifted template by $\widehat{\mathcal F}$.
Let $(\widehat f,\widehat g,\widehat d_0)$ be its sign norm and moment
statistics: its expectation scale is
$n^{\widehat f}(\log(2n))^{\widehat g}$ and each detached random component
contributes one factor $q^{1/2}$ in the logarithmic moment window.

\Needspace{13\baselineskip}
\begin{theorem}[Uniform Gaussian lift]\label{thm:gaussian-lift}
For every fixed coefficient-one independent-factor template with nonempty
scopes, comparable role sizes, and standard real Gaussian coordinates,
\[
 \E\norm{H_G}\asymp
 n^{\widehat f-Q/2}(\log(2n))^{\widehat g}.
\]
For every fixed $C_0>0$, uniformly for
$2\le q\le C_0\log(2n)$ and all sufficiently large $n$,
\[
 \Lpnorm{\norm{H_G}}q\asymp
 n^{\widehat f-Q/2}(\log(2n))^{\widehat g}q^{\widehat d_0/2}.
\]
The same assertions hold for standard circular complex coordinates.
Constants depend on the fixed template, role-size comparisons, and,
for the moment assertion, $C_0$.
\end{theorem}

The lift is finite and exact at the level of the comparisons used below.
It is not an appeal to a central limit theorem at growing moment order.
In particular, its private roles record fluctuations that would be lost
if unary Gaussian factors were removed as sign isometries.

The proof uses an exact finite comparison of even moments between a
Gaussian coordinate and a normalized sum of independent signs.
Representing the latter by a private summation role permits a direct
application of the sign calculus. The comparison is uniform at the
logarithmic moment orders needed for operator norms; a central limit
theorem would not provide that uniformity. Sign--magnitude conditioning
gives the reverse norm inequality. All normalization factors and the
choice of the private dimension are proved in
\cref{app:gaussian-lift-proof}.

\subsection{Fixed Hermite coordinates}

For fixed positive integers $d_e$, replace the coordinates of
occurrence $e$ by $\operatorname{He}_{d_e}(g)/\sqrt{d_e!}$,
with independent underlying standard real Gaussian arrays. Here
$e^{tx-t^2/2}=\sum_d\operatorname{He}_d(x)t^d/d!$.
Write $H_{\rm He}$ for the resulting matrix and $H_{\rm dec}$
for the template obtained by replacing each occurrence $e$ by
$d_e$ independent standard real Gaussian occurrences on the same scope.

\begin{theorem}[Fixed-degree Hermite comparison]\label{thm:hermite}
For the fixed independent-occurrence model just specified, put
$A=\prod_e d_e^{d_e/2}/\sqrt{d_e!}$. Then, in every dimension
and for every $q\ge1$,
\[
 A^{-1}\Lpnorm{\norm{H_{\rm dec}}}q
 \le \Lpnorm{\norm{H_{\rm He}}}q
 \le A\Lpnorm{\norm{H_{\rm dec}}}q.
\]
Consequently the expected-norm and logarithmic-window moment
exponents are those of the expanded Gaussian template.
\end{theorem}

This comparison is uniform in the number of labels, not in the
degrees $d_e$. It lets the same structural calculus treat a fixed
Hermite coordinate without mistaking an uncentered Gaussian power
for a homogeneous chaos.
It does not require the Hermite coordinate itself to have a symmetric
law, and it permits fixed degrees above two without changing the
$\theta_e\le2$ hypothesis of the separate tail-law theorem.

The complete proof is given in \cref{app:hermite-proof} (page~\pageref{app:hermite-proof}).

For original arities at least two, the logarithmic
charge becomes
$\frac12\max_{|S|=s}(a(S)+\sum_{e\subseteq S}d_e)$.
This does not cover ordinary powers $g^d$ without
Hermite subtraction, arbitrary sums of chaoses,
shared arrays across occurrences, or growing degrees.

\section{A square-root-scale degree-four SoS certificate}\label{sec:applications}\label{sec:sos}

The first principal application assembles correlated graph-matrix blocks
into a feasible semidefinite moment system. The norm estimates provide
blockwise error scales. The additional structure is algebraic: an exact
positive square protects the large correction, critical log-free bounds
control the remaining errors, and an extension from the pair block
enforces positivity and all constraints on the full moment matrix.

The application concerns the degree-four sum-of-squares relaxation for
clique in $G(n,1/2)$. A feasible pseudoexpectation at a prescribed size
shows that this particular relaxation cannot exclude that size. It is
a statement about a specified proof system, not a computational lower
bound for all algorithms. The quantitative target here is a fixed
positive constant times $\sqrt n$, without a denominator growing
polylogarithmically with $n$.
\subsection{Problem and quantitative statement}

\begin{definition}[Degree-four clique feasibility]\label{def:sos-feasibility}
For a graph $G$, degree-four clique feasibility at size $k$ means
a linear functional $\mathcal L_k$ on polynomials of degree at
most four satisfying normalization, positivity on squares of
degree-two polynomials, and the Boolean, nonedge, and size
constraints through their allowed degrees.
\end{definition}

\Needspace{12\baselineskip}
\begin{theorem}[Square-root-scale degree-four feasibility]
\label{thm:sos}
There exist absolute constants $c>0,n_0$ such that, for
$G\sim G(n,1/2)$ and $n\ge n_0$, with probability at least
$1-n^{-10}$, simultaneously for every real $9\le k\le c\sqrt n$,
there is a functional $\mathcal L_k$ satisfying
\[
 \begin{aligned}
 \mathcal L_k(1)&=1,&
 \mathcal L_k(p^2)&\ge0 &&(\deg p\le2),\\
 \mathcal L_k((x_i^2-x_i)q)&=0
       &&&&(\deg q\le2),\\
 \mathcal L_k(x_ix_jq)&=0
       &&&&(\{i,j\}\text{ a nonedge},\ \deg q\le2),\\
 \mathcal L_k((\sum_i x_i-k)q)&=0
       &&&&(\deg q\le3).
 \end{aligned}
\]
\end{theorem}

The construction follows the corrected quartic ansatz of
\citet{HopkinsKothariPotechin2015}. The relevant comparison
with \citet{RaghavendraSchramm2015} is at the logarithmic
scale: those cited versions give a
$\widetilde\Omega(\sqrt n)$ certificate. The statement above
has a fixed positive constant multiplying $\sqrt n$; it
does not change the polynomial exponent, give an optimized
constant, or establish a lower bound for every algorithm.

\subsection{The obstruction and the mechanism}

A bound on each random summand is not by itself a PSD proof. The
deterministic mean has different positive scales on the incidence space
and its kernel. Moreover, the quartic correction is correlated with the
same graph that generates the uncorrected moments. Applying the triangle
inequality to the correction and the dangerous cross term separately
discards precisely the positive square that makes the budget work.

Let $D$ be the signless vertex-pair incidence matrix, and write
$P$ for the projection onto $\operatorname{range}(D^*)$ and $P_0=I-P$.
The comparison metric is
\[
 B=\omega^3n^2P+\omega^2n^2P_0.
\]
The random correction contains $\rho QQ^*$; the local fluctuation
contains $\alpha_0(D^*Q^*+QD)$. Their sum is an exact square minus
$\alpha_0^2D^*D/\rho$. This negative term acts only on $P$, so it is
charged against the larger $\omega^3n^2$ scale. Charging it against the
smallest eigenvalue on the whole pair space would lose this advantage.

The remaining obstruction is logarithmic, but only in particular
blocks. Seven matching-containing patterns and fifteen edge-plus-path
patterns are at the critical polynomial scale and need log-free
bounds. The proof gives elementary same-scale references for both
families. Other fluctuations may retain a fixed logarithmic factor,
because a polynomial margin remains after normalization. The norm input
is therefore selective: log-free control is needed at the critical
scale, while a coarser bound suffices away from it. The elementary
references establish the critical estimates independently of the
full classification; the PSD assembly uses both types of estimates.

The final normalized budget has the schematic form
\[
 K\bigl(\gamma^{-1}+x+x^2+\gamma x^3\bigr)+o(1),
 \qquad x=\omega/\sqrt n.
\]
Choose the fixed correction strength $\gamma$ first, then a sufficiently
small fixed upper bound on $x$, and finally $n$ large. The term $o(1)$
is uniform over the allowed internal parameter interval. This order of
choices is the reason the conclusion has a constant multiple of
$\sqrt n$, rather than a dimension-dependent loss.

A PSD pair block still does not finish the application. The lower-degree
moments must be corrected using the same quartic data, the target size
must hold exactly, and constants and linear polynomials must be included
in positivity. The last two subsections of \cref{app:sos-detail} perform these steps on one
common graph event. No independence of the tuned parameter from the
graph, and no union bound over uncountably many target sizes, is used.

The complete construction and verification are in
\cref{app:sos-detail} (page~\pageref{app:sos-detail}). In particular,
\cref{eq:sos-square} displays the exact positive square, and
\cref{eq:sos-compression} identifies the actual clique-pair compression.
The proof then bounds every residual block in its own metric, tunes
the size constraint on one simultaneous event, and recovers the
constant and linear rows by an explicit congruence. Thus positivity
is not asserted only for the pair block or for a relaxed subset of
the constraints.

For the SoS interface, the first step is exact algebra. Multiplication
of graph matrices identifies inner boundary labels and partitions the
remaining terms by their cross-collisions; repeated quotient edges cancel
in pairs in the sign model. The output is a finite list of shapes with
specified coefficients. Applying logarithmic moments and a finite union
bound gives one event on which every required block estimate holds.
The PSD argument then works pointwise on that event. In particular,
a parameter chosen from the observed graph may enforce a size constraint
without needing independence from the random blocks it multiplies.

This interface also explains why a matching norm theorem is not a
black-box PSD theorem. Triangle inequality combines upper bounds for a
finite sum, but separate lower bounds need not survive cancellation.
The correction and its cross term must first be assembled into a positive
square. Only its residual negative part is charged to the appropriate
subspace. The distinction is algebraic, not a missing logarithm in the
norm estimate. The argument in \cref{sec:sos} keeps both stages explicit.
Both reductions fix template sizes and degrees before dimension grows.
The finite norm budgets therefore do not supply uniform constants for
growing-degree relaxations. The application also chooses its own
normalization, whose effect remains explicit in the argument.

\section{Spectral thresholds for random tensor networks}\label{sec:rtn}

A tensor-network contraction gives a different use of the norm theory.
Its primitive random objects are tensors, so its natural incidence
description is already an independent-factor chaos. To study approximate
isometry, however, one must first center its Gram matrix. This creates
a sum indexed by active tensor subsets. Decoupling each term produces a
new factor template whose separator is computed in a doubled network.
The norm theorem becomes an input only after these operations.

We distinguish two questions throughout. Mean-Gram normalization asks
whether a linear map is close to an isometry on its input space.
Sample-state normalization asks about the spectrum and entropy of
a density matrix obtained from one realization. These normalizations
are not interchangeable. Even for the mean-Gram deviation, operator
and unnormalized Schatten norms can have different convergence
thresholds because a finite Schatten norm retains a dimension factor.

\subsection{The model, its normalizations, and the cut gap}
\begin{definition}[Gaussian tensor network and normalizations]\label{def:rtn-model}
Fix a finite network with $Q$ tensor vertices, internal edges joining
distinct vertices, and open edges ending at distinct degree-one
terminals. Parallel internal edges are distinct indices. Split
terminals into $a$ outputs and $b$ inputs. Every component meets
a terminal. Each edge has dimension $N$, and each tensor has
independent standard circular complex Gaussian coordinates,
independently of the other tensors. Let $e$ be the internal edge
count and $r=e+a+b$. Contracting internal indices gives
$H_N:\C^{N^b}\to\C^{N^a}$.

Set
\[
 M_N=N^{-(r-b)/2}H_N,\qquad F_N=M_N^*M_N-I_{N^b},\qquad
 \rho_A=\frac{H_NH_N^*}{\|H_N\|_F^2}.
\]
\end{definition}
Covariance identifies the two copies of every tensor coordinate:
each input label must agree, and each internal or output edge
has one free label. Thus $\E H_N^*H_N=N^{r-b}I$.
This proves the mean-Gram normalization. The state normalization
is random and is treated separately.

If $Q=0$, this incidence model has no edges or terminals.
The empty contraction is the scalar one, the deviation is zero,
and all entropies vanish. Below assume $Q\ge1$.
\begin{definition}[Signed cut gap]\label{def:rtn-cut-gap}
For nonempty tensor sets $T$, write $c(T)$ for the internal
edge boundary and $a(T),b(T)$ for incident output and input
legs. Put
\[
 \delta(T)=c(T)+a(T)-b(T),\qquad
 \Delta=\min_{\varnothing\ne T}\delta(T),\qquad d=N^b.
\]
\end{definition}
Placing precisely $T$ on the input side gives a cut of size
$b+\delta(T)$. Thus $\Delta>0$ means that the input-terminal
cut is uniquely minimum.

\subsection{Sharp deviation scales and their interpretation}

All Schatten norms below are unnormalized. In particular,
$\|I_d\|_{S_t}=d^{1/t}$ for finite $t$, whereas $\|I_d\|_{\rm op}=1$.
\begin{theorem}[Mean-Gram deviations]\label{thm:rtn}
Assume at least one tensor vertex.
If $\Delta\ge0$, then for every fixed $1\le t<\infty$,
\[
 \E\|F_N\|_{S_t}\asymp_{\mathcal G,t}
                         N^{b/t-\Delta/2}.
\]
There are positive constants $c_t,p_t$ such that the norm
exceeds $c_tN^{b/t-\Delta/2}$ with probability at least
$p_t$, uniformly in $N$.
For all signs of $\Delta$,
\[
 \|F_N\|_{S_t}\xrightarrow{\Prob}0
       \ \Longleftrightarrow\ \Delta>2b/t,\qquad
 \|F_N\|_{\rm op}\xrightarrow{\Prob}0
       \ \Longleftrightarrow\ \Delta>0.
\]
If $\Delta>0$, the uniform Gaussian norm bound additionally
gives $\E\|F_N\|_{\rm op}\asymp N^{-\Delta/2}$.
\end{theorem}

The signed gap has a direct meaning. The input-terminal cut has size
$b$. Moving a nonempty set $T$ of tensors to the input side changes its
size by $\delta(T)$. A strictly positive minimum change ensures that
the original cut is uniquely minimum, which is the operator-isometry
criterion. For a finite Schatten norm, that same gap must also pay for
the $N^b$ input directions; the required extra margin is $2b/t$.

For example, consider a single tensor, viewed as an $N^2$ by $N$
Gaussian matrix. Here $a=2$, $b=1$, and $\Delta=1$. The expected
operator deviation is of order $N^{-1/2}$ and converges to zero.
The expected Frobenius deviation is of constant order and does not
converge to zero in probability. The expected trace-norm deviation
is of order $N^{1/2}$. Thus one family of maps can be approximately
isometric in operator norm and fail the corresponding unnormalized
Frobenius or trace criterion. The independent finite-dimensional
Wishart calculation in \cref{app:wishart} verifies this dimension
dependence directly.

The proof separates three inputs. Fixed trace moments give the
finite-Schatten upper bounds and qualitative operator convergence.
An exact second moment and a fourth-moment bound yield nonrare spectral
mass, proving the lower scales and excluding convergence at equality.
The uniform Gaussian theorem in \cref{sec:gaussian} is needed for the
sharp operator expectation rate; a fixed-order moment estimate alone
would leave an arbitrarily small polynomial loss.

\subsection{The sample-state spectrum and entropy}

For $\alpha>0$, $\alpha\ne1$, write
$H_\alpha(\rho)=(1-\alpha)^{-1}\log\tr(\rho^\alpha)$.
Use the continuous extension at $\alpha=1$, and set
$H_0(\rho)=\log\rank\rho$ and
$H_\infty(\rho)=-\log\|\rho\|_{\rm op}$.
Let $s$ be the ordinary minimum input-output edge cut, not the signed
gap $\Delta$.

\begin{theorem}[Sample-normalized spectral endpoint]\label{thm:rtn-state}
For the fixed network model above, $\rho_A$ is defined almost surely.
There is a constant $C_{\mathcal G}$ such that, with probability
$1-O_{\mathcal G}(N^{-1})$,
\[
 N^{-s}\le\|\rho_A\|_{\rm op}\le C_{\mathcal G}N^{-s}.
\]
On the same event,
\[
 s\log N-\log C_{\mathcal G}
 \le H_\alpha(\rho_A)\le s\log N
 \qquad (0\le\alpha\le\infty)
\]
simultaneously for all the indicated orders.
\end{theorem}

This theorem controls the entropy endpoint at the scale level;
it does not identify a leading spectral constant.
For connected networks with both boundary sets nonempty,
\cref{thm:rtn-exact-edge} supplies the sharper conclusion by a
separate proof. We retain the present theorem for its broader
model and its simultaneous guarantee for all entropy orders.
Its lower bound on the largest eigenvalue comes from cut rank.
Its upper bound requires both the uniform operator estimate for $H_N$
and concentration of the random normalization $\|H_N\|_F^2$.
Replacing the latter by its expectation without a probability estimate
would not prove the statement.

Random tensor networks and their spectral consequences have been studied
in several settings
\citep{HaydenEtAl2016,Hastings2016,ChengEtAl2024,FitterLoulidiNechita2024}.
Our statements specify independent Gaussian coordinates, fixed network
size, and equal edge dimensions. The mean-Gram result is not asserted
for a network obtained by independently normalizing every local tensor.
Such scalar normalizations cancel in $\rho_A$, but they change the
mean-Gram question. The proof in \cref{app:rtn-detail} keeps this distinction explicit.

\subsection{Why the thresholds use different moment inputs}

For $\Delta\ge0$, the lower argument gives a nonvanishing probability
of the matching scale, excluding convergence at equality and below the
finite-Schatten threshold. For $\Delta<0$, a rank obstruction excludes
convergence in all the stated norms.
Thus the dimension factor changes the criterion for a finite Schatten
norm, rather than merely changing a constant in an operator estimate.

A single Gaussian tensor gives a useful calibration. If its output and
input dimensions are $N^a$ and $N^b$, then $\Delta=a-b$. Operator-norm
convergence requires $a>b$, while unnormalized trace-norm convergence
requires $a>3b$. Confusing these two statements changes the threshold
even in a Wishart matrix. We include the second and fourth Wishart
moments as an independent check.

Fixed trace moments establish the finite Schatten thresholds and the
qualitative operator-norm criterion. They do not alone give the sharp
operator-norm expectation rate: obtaining that rate requires a moment
estimate uniform up to logarithmic order. We keep these dependencies
separate.
For the RTN interface, centering the Gram matrix produces a sum indexed
by active tensor subsets. Quadratic Gaussian decoupling turns each term
into an independent-factor template in a doubled network. Its separator
cost gives the cut-gap penalty, while its matrix dimension supplies
the factor retained by a finite Schatten norm. Fixed moments suffice
for the finite-order conclusions. Uniform moments become essential
when obtaining sharp operator scales without a residual polynomial loss;
this is the additional output of \cref{thm:gaussian-lift} used here.

The entropy consequence has one more step. The output state is normalized
by its realized Frobenius norm, not by its expected Gram matrix. A
deterministic cut bounds rank, while a norm bound controls the largest
eigenvalue after this normalization. These two quantities bound the
R\'enyi entropies from opposite sides. Keeping this step separate from
approximate isometry explains why the state theorem is a consequence
within the RTN application, rather than a third independent application.

The full proof is in \cref{app:rtn-detail}
(page~\pageref{app:rtn-detail}). It includes the exact doubled-network
role count, the second-moment identity, a positive-probability lower
bound excluding convergence at the threshold, and the separate
normalization step needed for the entropy statement.

\section{Exact spectral edges of Gaussian tensor networks}\label{sec:rtn-edge}

The minimum cut determines the dimension scale of the largest
eigenvalue, but it does not determine its leading constant.
That constant records finer network structure. We now identify it
for every fixed connected network in the independent complex-Gaussian,
equal-dimension model. The result upgrades the sample-state scale in
\cref{thm:rtn-state} to convergence of the right spectral edge.

This requires a separate argument. The all-defect estimates for graph
matrices determine powers and logarithms up to shape-dependent constants;
they do not identify the exact exponential base of an RTN moment
sequence. Here a global flow quotient and a comparison at a different
auxiliary dimension preserve that base. The quotient is realized as a
product of overlapping Ginibre gates, to which the tensor-GUE strong
convergence theorem of \citet[Section~9.4, Theorem~9.8]{ChenGarzaVargasVanHandel2026}
applies. The reduction and dimension comparison below connect that
external theorem to the original network; they do not identify
contracted original tensor clusters with independent Gaussian gates.

\subsection{The model and the exact right edge}

\begin{definition}[Connected equal-dimension RTN model]\label{def:rtn-edge-model}
In this section the tensor-vertex graph is fixed, finite, connected,
and loopless. Parallel internal edges are allowed and retain distinct
index coordinates. The sets $A,B$ of output and input legs are both
nonempty. All internal edges and open legs have dimension $N$.
Each actual tensor vertex carries an independent array of independent
standard circular complex Gaussians with $\E|g|^2=1$.
Contractions use the unnormalized pairing
$\sum_{i=1}^N e_i\otimes e_i$.

Attach all input legs to a terminal $\mathsf b$ and all output legs to
a terminal $\mathsf a$, obtaining an augmented multigraph $G$.
Write $V_G$ for its actual tensor vertices, $v=|V_G|$, and $R=|E(G)|$
for the number of internal edges plus open legs. Let $s\ge1$ be the
minimum $\mathsf b$--$\mathsf a$ edge-cut size, allowing terminal legs
in the cut. Define
\begin{equation}\label{eq:edge-normalizations}
 \begin{gathered}
 Y_N=N^{s-R}H_NH_N^*,\qquad
 Z_N=N^{-R}\tr(H_NH_N^*),\\
 X_N=N^s\rho_A=Y_N/Z_N .
 \end{gathered}
\end{equation}
Cut factorization gives $\rank H_N\le N^s$. Whenever we use
$N^{-s}\tr$, we normalize the nonzero spectrum padded to $N^s$
eigenvalues; the original boundary matrix need not have that dimension.
\end{definition}

\begin{definition}[Permutation energy and limiting right edge]\label{def:rtn-energy}
For $p\ge1$, set $\gamma_p=(1\,2\,\cdots\,p)$ and
$|\sigma|=p-\#\operatorname{cyc}(\sigma)$ on $S_p$.
The transposition metric is $d(\sigma,\tau)=|\sigma^{-1}\tau|$.
For labels $\pi_u\in S_p$ at actual tensor vertices, fix
$\pi_{\mathsf b}=\id$, $\pi_{\mathsf a}=\gamma_p$, and put
\begin{equation}\label{eq:edge-energy}
 \begin{gathered}
 \mathcal H_G^{(p)}(\pi)
   =\sum_{\{u,w\}\in E(G)}d(\pi_u,\pi_w),\qquad
 \Delta_G(\pi)=\mathcal H_G^{(p)}(\pi)-s(p-1),\\
 A_G(p,\delta)=\#\{\pi:\Delta_G(\pi)=\delta\}.
 \end{gathered}
\end{equation}
Let $\mu_G$ be the compactly supported probability measure with
positive moments $m_p=A_G(p,0)$, and set $E_G=\sup\supp\mu_G$.
\end{definition}
This is the minimum-energy moment law of
\citet[Theorems~4 and~7]{FitterLoulidiNechita2024}.
Its existence and the identification of its right endpoint also
follow from the construction below.

\begin{theorem}[Exact RTN right edge]\label{thm:rtn-exact-edge}
Under the fixed connected model of this section,
\begin{equation}\label{eq:edge-main}
 \norm{Y_N}\xrightarrow{\Prob}E_G,\qquad
 N^s\lambda_{\max}(\rho_A)\xrightarrow{\Prob}E_G .
\end{equation}
In particular $1\le E_G<\infty$, and
\begin{equation}\label{eq:edge-min-entropy}
 H_\infty(\rho_A)=s\log N-\log E_G+o_{\Prob}(1).
\end{equation}
\end{theorem}

The theorem excludes right outliers at every fixed positive distance
from $E_G$, a conclusion not implied by convergence of fixed normalized
moments. It identifies a constant through a uniquely specified compact
law, not a finite closed-form algorithm for every graph. No local law,
edge fluctuations, or smallest-nonzero-eigenvalue assertion is included.

For a single $N$ by $N$ Gaussian gate, $s=1$ and the zero-defect
permutations are the noncrossing interval $[\id,\gamma_p]$.
Their Catalan moments give the Marchenko--Pastur edge $E_G=4$.
For a single $N^2$ by $N$ gate, the energy is
$|\pi|+2d(\pi,\gamma_p)$; its unique minimizer is $\pi=\gamma_p$.
Thus $\mu_G=\delta_1$ and $E_G=1$.
Both networks have the same minimum cut, but their min-entropy
corrections differ. These are classical single-gate calibrations;
the theorem treats arbitrary fixed connected topology.

If each internal link instead uses the normalized Bell pairing,
the raw Gram matrix gains a factor $N^{-e_{\rm int}}$.
Multiplying it by $N^{s-|A|-|B|}$ then gives exactly $Y_N$.
The sample-normalized state itself is unchanged. Likewise, normalizing
each local Gaussian tensor by its own norm multiplies the full
contraction by one nonzero scalar, which cancels from $\rho_A$.
This observation transfers the state conclusion to those local
Haar-vector tensors, but not the deterministic normalization of $Y_N$
or the mean-Gram conclusions of \cref{sec:rtn}.

\subsection{Overview of the dimension-transfer mechanism}

Preserving an exact endpoint requires greater precision than
preserving norm exponents. Let $A_G(p,\delta)$ count the
permutation labelings of the RTN Wick expansion with energy defect
$\delta$, and let
$F_G(p,N)=\sum_{\delta\ge0}A_G(p,\delta)N^{-\delta}$.
A maximum terminal flow defines a directed graph whose full
strongly connected components form a balanced acyclic quotient $Q$.
One anchor per component gives a graph-wide label comparison,
including every positive-defect layer. If $v$ is the number of
tensor vertices and $R$ counts internal edges and open legs, it yields
\begin{equation}\label{eq:intro-rtn-transfer}
 F_G(p,N)\le\frac{F_Q(p,p^2)}{1-p^{K_G}/N},
 \qquad K_G=4v+2+8R,\quad N>p^{K_G}.
\end{equation}
The quotient is realized by independent square Ginibre gates on
possibly overlapping tensor factors. Strong convergence of tensor
GUE, after expressing each gate through two GUE matrices, identifies
its limiting norm \citep{ChenGarzaVargasVanHandel2026}.
Gaussian concentration and a high-moment bound on the exceptional
event then show $F_Q(p,p^2)\le2(E_G+\eta)^p$ for each fixed
$\eta>0$ and all sufficiently large $p$. The auxiliary dimension
$p^2$ suppresses that exceptional contribution without changing
the limiting exponential base. Taking $p$ proportional to $\log N$
proves the upper edge; fixed-moment concentration proves the
matching lower edge.

\Cref{lem:interface-growing,cor:interface-log} state the analytic
bridge separately from the circuit construction. Small exceptional
probability alone does not control a growing moment: the rare-spike
example in \cref{app:positive-growing} explains why the high-moment
estimate on that contribution is also verified.

This argument never assumes that an original contracted tensor
cluster remains Gaussian. Its independent Gaussian circuit is a
comparison ensemble introduced only after a deterministic count.
It is independent of the graph-matrix classification, although both
proofs control nonleading replica states through their dimension loss.
The exact-edge statement retains its narrower model; it does not
replace the broader common-event entropy bounds above.

The comparison has an elementary form that can be used beyond this
network model. Its inputs separate the combinatorial task from the
subsequent choice of dimension: construct one global map, bound its
effect on defect, and control the total weight of each fiber.

\begin{theorem}[Positive defect-fiber transfer]\label{thm:positive-fiber}
For each integer $p\ge2$, let $\Omega_p,\Lambda_p$ be finite state
sets with nonnegative, dimension-independent weights $w_\omega,v_\lambda$
and defects $\delta(\omega),j(\lambda)\in\N_0$. Define
\[
 F_\Omega(p,N)=\sum_\omega w_\omega N^{-\delta(\omega)},\qquad
 F_\Lambda(p,M)=\sum_\lambda v_\lambda M^{-j(\lambda)}.
\]
Suppose one map $\phi_p:\Omega_p\to\Lambda_p$ satisfies, for a
fixed $\kappa\ge0$ and numbers $r_p\ge1$,
\begin{align}
 j(\phi_p(\omega))&\le\kappa\delta(\omega),
 \label{eq:interface-defect}\\
 \sum_{\substack{\omega:\phi_p(\omega)=\lambda\\
                        \delta(\omega)=\delta}}w_\omega
 &\le r_p^\delta v_\lambda\quad\text{for every $\lambda,\delta$}.
 \label{eq:interface-fiber}
\end{align}
Then, for $M\ge1$ and $N>r_pM^\kappa$,
\begin{equation}\label{eq:positive-fiber}
 F_\Omega(p,N)\le
 \frac{F_\Lambda(p,M)}{1-r_pM^\kappa/N}.
\end{equation}
Also, $a_{p,\delta}:=\sum_{\delta(\omega)=\delta}w_\omega$
satisfies $a_{p,\delta}\le r_p^\delta M^{\kappa\delta}F_\Lambda(p,M)$.
\end{theorem}
\begin{proof}
Group source states by defect and image. The hypotheses give
\[
 a_{p,\delta}\le r_p^\delta
   \sum_{\lambda:j(\lambda)\le\kappa\delta}v_\lambda
 \le r_p^\delta M^{\kappa\delta}
   \sum_\lambda v_\lambda M^{-j(\lambda)}.
\]
Multiply by $N^{-\delta}$ and sum the geometric series.
\end{proof}

The theorem does not construct the map. For the RTN, that task is
performed by the full-SCC anchor construction in
\cref{lem:edge-variation,lem:edge-fiber}. Wick weights are one,
$r_p=p^{4v}$, and $\kappa=1+4R$. For general weights, a cardinality
bound alone cannot replace \cref{eq:interface-fiber}. Nor can one
replace the full positive target sum by its zero-defect coefficient.
The path count uses a related but different input: direct evaluation
of positive falling-factorial weights, formalized in
\cref{lem:positive-calibration}. These are two uses of positivity,
not an identification of their state spaces or encodings.

The quantitative comparison can be stated independently of the limiting
law. Orient a fixed maximum family of terminal flow paths forward and
each unused edge in both directions. Let $Q$ be the quotient by the
full strongly connected components, retaining edge multiplicities.
For this quotient, $F_Q(p,M)$ denotes the positive Wick moment
polynomial defined just as for $G$. Put $C_0=1+4R$.

\begin{theorem}[Auxiliary-dimension comparison]\label{thm:edge-transfer}
For integers $p\ge2$, $M,N\ge1$ with $p^{4v}M^{C_0}<N$,
\begin{equation}\label{eq:edge-transfer}
 F_G(p,N)\le
 \frac{F_Q(p,M)}{1-p^{4v}M^{C_0}/N}.
\end{equation}
In particular, with $K_G=4v+2+8R$,
\begin{equation}\label{eq:edge-transfer-square}
 F_G(p,N)\le
 \frac{F_Q(p,p^2)}{1-p^{K_G}/N}
 \qquad(N>p^{K_G}).
\end{equation}
\end{theorem}
The complete proof is given in \cref{app:edge-transfer-proof} (page~\pageref{app:edge-transfer-proof}).

The entire proof of the edge theorem is in \cref{app:edge-detail}.
In particular it constructs the quotient, bounds the fibers of one
global anchor map, identifies the independently sampled comparison
circuit, checks the hypotheses of the external strong-limit theorem,
and derives the growing-moment and normalization estimates. The
strong-limit input itself is attributed, not claimed as a new theorem
proved here.

\subsection{Comparison and scope}

The max-flow approach of \citet{FitterLoulidiNechita2024}
already identifies the fixed-moment law, establishes weak
convergence, and derives entropy corrections. In its
series-parallel setting the law has an explicit convolution
description. Our additional conclusion is right-edge convergence
for the fixed connected Gaussian network, rather than another
identification of that law. The strong-convergence input for
overlapping tensor GUE is due to
\citet{ChenGarzaVargasVanHandel2026}, building on the weak model
of \citet{CharlesworthCollins2021}.

The bridge specific to this proof is the all-defect global comparison
with a balanced acyclic quotient, followed by auxiliary-dimension
transfer. Its proof is independent of the graph-matrix classification
in \cref{sec:counting,sec:lower}. Both arguments exploit dimension loss
to control nonleading configurations, but their combinatorial states
and precision targets differ.

All topology and support sets are fixed before $N$ grows.
The argument gives no practically uniform threshold over growing
graphs or shrinking tolerances. Nonflat link weights, shared
tensors across vertices, and non-Gaussian coordinates are outside
this exact-edge statement. Nor does the min-entropy correction
$-\log E_G$ serve as the correction for every finite R\'enyi
order. The broader coarse bound in \cref{thm:rtn-state} and
the mean-Gram criteria in \cref{thm:rtn} retain their separate
models and conclusions.

\Needspace{10\baselineskip}
\section{Concluding remarks}\label{sec:discussion}

The sharp classification separates two sources of spectral growth that
are not visible in the minimum-cut size alone. A separator determines
the polynomial cost of sharing labels, while the components exposed by
a minimizing separator determine which local fluctuations can be
selected by the operator norm. The matching upper and lower bounds
identify the resulting logarithm as a feature of the matrix, not an
artifact of concentration. The equal-coarse-data chains in
\cref{cor:coarse-separation} isolate the information detected by this
rule. The coarea argument explains why the description survives beyond
the leading trace configurations: excess
separator cost pays for the additional choices in every defect layer.

The extensions separate incidence, label-space geometry, and scalar
tails as inputs to the norm problem. The sign leaf family and the
Gaussian gated chain have different primitive variables but share
the same competition among minimizing cuts. Local weights and aspect
ratios alter the scale and the minimum-separator family, while bounded
noise and Hermite comparisons transfer conclusions between specified
laws. These results do not furnish a joint theorem for arbitrary
coefficients or dimension-dependent laws. Critical sparse occupancy
introduces a $\log\log n$ correction, and the global-coefficient
examples show that local magnitudes do not determine operator geometry.

The distinction between spectral scale and spectral precision is also
important for applications. In the degree-four construction, removing
the logarithmic loss from the critical blocks is enough to reach a
constant multiple of the square-root scale, but only after the
correlated terms are assembled into a positive square. For an exact
RTN edge, even a matching norm scale leaves the essential constant
undetermined. The separate quotient comparison preserves the
exponential base of the moment sequence and connects the original
network to an ensemble with a known strong limit. These arguments
illustrate two different ways in which sharper spectral information
becomes useful; neither reduces an application to bounding each
summand separately.

The fixed-shape classification leaves a concrete algorithmic question.
Finite enumeration evaluates the exponents, but does not efficiently
maximize the active-component count over all minimum separators. A second
question is whether the
all-defect argument can be made sufficiently uniform in the shape to
treat growing-degree semidefinite relaxations. The present encoding and
the weighted clone reduction allow constants to depend on the template;
controlling that dependence would require more than the termination
proved here.

The Lean formalisation is ongoing and can be found here.\footnote{The current development is available at
\url{https://github.com/xxxhhbb/sharp-norms-from-finite-structure}.}
\label{end:body}

\paragraph{AI Disclosure.}

We used OpenAI's ChatGPT and Codex (GPT-family models) to assist in
developing and completing portions of the proofs, including proposing
intermediate lemmas and alternative arguments, checking technical steps
and boundary cases, performing auxiliary computations, and improving the
exposition. The tools materially affected selected proof developments in
Sections~3--9 and the corresponding appendices. The research questions,
theorem statements, overarching proof architecture, and final selection
and validation of all arguments were determined by the authors. The authors
independently reviewed and verified all AI-assisted material incorporated
into the manuscript, checked the computations and references, and take full
responsibility for the correctness and originality of the work.

\clearpage
\appendix
\crefalias{section}{appendix}
\section{Proof guide and external inputs}\label{app:proof-guide}

The main text states the results, explains their scope, and retains the
separator coarea argument, coefficient flattening, and conversion from
moments to norms. The appendices below contain the technical arguments
needed to complete those proofs. All arguments needed for the stated
results are contained in the paper. The following guide locates every
numbered theorem, lemma, proposition, and corollary.
References to established external theorems are distinguished from
results proved here; their use does not constitute a new proof of those
external theorems.

\subsection{The graph-matrix classification}

\Cref{thm:intro-graph} is assembled at the end of \cref{sec:lower}.
Its upper bound uses \cref{prop:core-upper}, proved immediately after
its statement in \cref{sec:counting}; its conditional lower bound uses
\cref{app:lower-detail}. The full logarithmic-window moment claim is
completed by the detached-component argument in \cref{app:factor-detail}.
\Cref{lem:reductions,lem:color-upper,lem:color-lower,prop:terminal}
have their proofs adjacent to their statements in \cref{sec:models}.
Thus the typed/global comparison and finite evaluation are part of
the proof, rather than assumptions imported from experimental data.

\Cref{prop:chain-statistics,cor:coarse-separation} have their deductions
in \cref{sec:separating-family}. Their coefficient input,
\cref{lem:role-frontier}, is proved in \cref{app:separation-detail},
which also contains all literature substitutions and the adapted
product bound. The sharp family scales depend on the classification;
they are not an independent proof of it.

\Cref{lem:exact-trace,lem:interval,lem:coarea} are proved in
\cref{sec:counting}. \Cref{thm:count} is completed in
\cref{app:counting-detail}, including \cref{lem:path-count}, its proof,
the off-path record and decoder, and the seed-entropy budget.
\Cref{lem:positive-calibration} is proved immediately before the path
count, which instantiates it with exact falling-factorial weights.
\Cref{lem:stacking,lem:separator-flattening} are proved in
\cref{sec:lower}; \cref{lem:tilt} and the full conditional estimates
and synchronization are proved in \cref{app:lower-detail}.
The trace obstruction has its complete construction and calculation
in \cref{app:trace-detail}.

\subsection{The structured extensions}

\Cref{thm:intro-factors,thm:tail-factor,prop:profile} are proved in
\cref{app:factor-detail}. In particular, the layerwise refinement
\cref{eq:layerwise} is proved there before it is used in the weighted
trace expansion; it is not inferred from the weaker aggregate count.
\Cref{thm:bounded-comparison} has its dimension-free proof in
\cref{app:bounded-detail}; \cref{cor:bounded-factor} then has its short
deduction next to its statement in \cref{sec:bounded}.
\Cref{thm:khatri-rao} has an independent proof in
\cref{app:khatri-detail}. The sparse statement
\cref{prop:sparse-phase} is proved, regime by regime, in
\cref{app:sparse}.

\Cref{thm:aspects} is proved in \cref{app:weights-detail}, in the order
local contraction, positive-aspect minimum-face reduction, conditional
lower bound, bounded-role elimination, and unary absorption. The
Gaussian result \cref{thm:gaussian-lift} is proved in
\cref{app:gaussian-lift-proof}, and \cref{thm:hermite} in
\cref{app:hermite-proof}. Each comparison retains its own hypotheses;
these proofs do not claim their unrestricted simultaneous combination.

\subsection{The applications and the exact edge}

\Cref{thm:sos} is proved in \cref{app:sos-detail}. This includes the
exact collision algebra, the full finite pattern families, the two-scale
PSD budget, lower-degree cleanup, tuning on one common event, and the
congruence recovering the entire degree-four moment matrix.
\Cref{thm:rtn,thm:rtn-state} are proved in \cref{app:rtn-detail}, with
separate arguments for fixed Schatten order, the uniform operator
estimate, the probability lower bound, and sample normalization.

\Cref{thm:rtn-exact-edge} is proved in \cref{app:edge-detail}.
\Cref{lem:edge-wick,lem:edge-quotient,lem:edge-variation,lem:edge-fiber}
are proved there before \cref{thm:edge-transfer}, whose proof is at
\cref{app:edge-transfer-proof}. They establish a deterministic
comparison of positive moment polynomials. Next,
\cref{lem:edge-circuit,lem:edge-gate-tail,lem:edge-concentration,lem:edge-growing}
have adjacent proofs establishing the independent comparison circuit
and its growing moments. Finally \cref{lem:edge-covariance} has its
proof before the fixed-moment lower edge and the final normalization.
These are not consequences of the graph-matrix classification.

The weighted abstraction \cref{thm:positive-fiber} has its complete
proof next to its statement in \cref{sec:rtn-edge}; the SCC lemmas
verify its hypotheses for the network. The general analytic bridge
\cref{lem:interface-growing} and \cref{cor:interface-log} are proved
in \cref{app:positive-growing}, followed by the rare-spike example
and the quantitative circuit verification. The two positive-evaluation
uses are distinguished rather than presented as the same map.

\subsection{What is used from the literature}

The Boolean hypercontractive inequality used in
\cref{app:lower-conditional} is the fixed-degree estimate
$\|P\|_q\le(q-1)^{d/2}\|P\|_2$ for a degree-at-most-$d$ polynomial
in independent signs and $q\ge2$ \citep{ODonnell2014}.
Its applications here have fixed degree and fixed moment order.
Paley--Zygmund, conditional Jensen, finite-dimensional net bounds,
and integral max-flow/min-cut are used in their standard forms;
the specific estimates and graph reductions are derived in the proofs.

For the sparse iid matrix, the external Seginer comparison is stated
in \cref{eq:sparse-seginer} and attributed to
\citet[Theorem~1.1]{Seginer2000}. The complete occupancy estimates
following that comparison are included here. This is a calibration
using an established norm theorem, not a new theorem about all sparse
chaoses.

For the exact RTN edge, the substantial external input is strong
convergence for fixed polynomials in independent GUE matrices embedded
on fixed nonempty, possibly overlapping tensor supports
\citep[Section~9.4, Theorem~9.8]{ChenGarzaVargasVanHandel2026}.
The application following \cref{lem:edge-circuit} specifies the number
of variables, their supports, the coefficient dimension, and the
polynomial. The companion fixed-moment input is
\citet[Theorem~4]{CharlesworthCollins2021}. The limiting moment law
agrees with that of \citet{FitterLoulidiNechita2024}; the argument also
constructs it from the comparison circuit. We do not reprove these
external strong- and weak-limit theorems. We do prove the quotient
comparison, the passage from qualitative norm convergence to growing
moments, both bounds on the original edge, and the sample-normalization
step that use them.

Accordingly, ``complete proof'' here means a complete derivation of
the stated contributions from explicitly identified standard and
external inputs, with all manuscript-specific counting, conditioning,
and application verifications included in this PDF.

\Needspace{26\baselineskip}
\subsection{Application dependencies at a glance}

The conclusions require different inputs in addition to norm control. \Cref{tab:intro-interfaces} separates those inputs; in particular, the exact RTN edge uses the independent comparison of \cref{sec:rtn-edge}.

\begin{table}[H]
\centering
\caption{Application interfaces. The exact-edge row uses a separate
proof route for the connected RTN model with nonempty boundaries.}
\label{tab:intro-interfaces}
\medskip
\setlength{\tabcolsep}{6pt}
\renewcommand{\arraystretch}{1.12}
\begin{tabular}{@{}>{\raggedright\arraybackslash}p{0.22\linewidth}>{\raggedright\arraybackslash}p{0.33\linewidth}>{\raggedright\arraybackslash}p{0.39\linewidth}@{}}
\toprule
\textbf{Conclusion} & \textbf{Norm or moment input} & \textbf{Additional argument} \\
\midrule
Degree-four SoS feasibility & Log-free critical shapes; tails for
the remaining blocks & Subspace metric, positive square, full moment constraints \\
\addlinespace[4pt]
Finite Schatten thresholds & Fixed moments; second/fourth-moment lower bound
& Centering, decoupling, doubled-network cuts \\
\addlinespace[4pt]
Sharp operator expectation & Logarithmic-window Gaussian moments
& Centered expansion; nonrare lower scale \\
\addlinespace[4pt]
State spectrum and entropy & Uniform Gaussian norm bound
& Sample normalization; deterministic cut-rank bound \\
\midrule
Exact RTN right edge & Tensor-GUE strong convergence; auxiliary growing moments
& Flow-SCC comparison, dimension transfer, fixed-moment lower edge \\
\bottomrule
\end{tabular}
\end{table}

\section{The full-trace obstruction}\label{app:trace-detail}

\subsection{Why the full trace sequence is not enough}\label{sec:trace-obstruction}

There is a useful obstruction to any approach that treats an ordinary
trace sequence as a complete summary of the matrix. Consider two
three-role shapes with one common row and column role. In the first,
that root is isolated and the two middle roles form an edge. In the
second, all three roles form a triangle. Both matrices are diagonal.
For root label $i$, their diagonal entries are respectively
\begin{equation}\label{eq:intro-trace-collision}
 X_i=2\sum_{\substack{a<b\\a,b\ne i}}\xi_{ab},
 \qquad
 Y_i=2\sum_{\substack{a<b\\a,b\ne i}}
                         \xi_{ia}\xi_{ib}\xi_{ab}.
\end{equation}
The factor two comes from the two orders of the middle roles in
the globally injective sum.

For each fixed $i$, these two random variables have exactly the
same distribution. Indeed, condition on every root-incident sign
$\xi_{ia}$. The remaining signs $\xi_{ab}$ are independent, and
multiplying them by the fixed signs $\xi_{ia}\xi_{ib}$ preserves
their joint law. Each diagonal entry therefore has the distribution
of twice a sum of $\binom{n-1}{2}$ independent signs. Since the
trace of a power of a diagonal matrix sums coordinate powers,
\begin{equation}\label{eq:intro-identical-traces}
 \E\tr(M_0^{2p})=\E\tr(M_1^{2p})
 \qquad\text{for every integer }p\ge1.
\end{equation}
This is an exact finite-$n$ identity, not agreement of finitely
many leading moments. Odd expected trace moments vanish as well.
When $n\equiv0$ or $3\pmod4$, the number
$\binom{n-1}{2}$ is odd, so none of these sign sums can be zero.
On this infinite subsequence both matrices have deterministic
rank $n$ for every realization.

Nevertheless their expected norms differ:
$\E\|M_0\|\asymp n$, whereas
$\E\|M_1\|\asymp n\sqrt{\log n}$.
The first scale is particularly transparent. Write
$Z=2\sum_{a<b}\xi_{ab}$, so
$X_i=Z-2\sum_{a\ne i}\xi_{ia}$.
The shared scalar has $L^1$ size comparable to $n$,
while the maximum of the row-sum corrections has expected
size $O(\sqrt{n\log n})$. The upper bound follows by
triangle inequality and a scalar subgaussian union bound;
the lower bound already follows from the marginal $L^1$
norm of one $X_i$. In the rooted triangle, removing the
mandatory root exposes a connected active component, and
the lower-bound mechanism of this paper yields the
additional square-root logarithm.

The difference is in the joint dependence of the coordinates,
which an ordinary expected trace does not record.
Direct parity counting gives, for $i\ne j$,
$\operatorname{Corr}(X_i,X_j)=(n-3)/(n-1)$ but
$\operatorname{Corr}(Y_i,Y_j)=2/(n-1)$.
In the first case almost all fluctuation is common to the
coordinates. In the second it can be amplified by selecting
the root label. We do not use small pairwise correlations
as a substitute for a lower-bound proof; the fully
conditioned trial argument supplies that proof.

This example explains an important feature of the upper-bound
strategy. We do not feed the undifferentiated trace sequence
of the full shape into an optimizer and expect it to
reconstruct the expected norm. We first identify genuinely
detached components from the shape itself and take their
$L^1$ contribution separately. Only then do we use high
trace moments on the boundary core. The preprocessing
retains information lost by \cref{eq:intro-identical-traces}.
Adding more ordinary trace moments to the unprocessed
matrix cannot repair this loss.

The example also clarifies what the finite rule computes.
It is a structural classification of a supplied shape, not
a spectral reconstruction theorem. Matrix dimension, exact
rank, and even all ordinary expected trace moments can agree
while the logarithmic exponent differs. A complete
classification can therefore be simpler to evaluate on a
graph than to infer from a seemingly richer list of
spectral statistics.

\section{Coefficient calculations for the separating family}
\label{app:separation-detail}

The graph family and its sharp scales are stated in
\cref{sec:separating-family}. This appendix supplies the coefficient
calculations and comparisons used there. The sharp norm statement is a
corollary of the classification, not an independent proof of it.

\subsection{Exact coefficient flattenings}

\begin{lemma}[Role-frontier formula]\label{lem:role-frontier}
For any partition of the axes of the equal-size typed coefficient tensor
$\mathcal A$ into rows and columns, let $T$ consist of the roles having
appearances on both sides. Then
\begin{equation}\label{eq:role-frontier}
 \norm{\mathcal A_{[R\mid C]}}=n^{(v-|T|)/2}.
\end{equation}
In particular,
\begin{equation}\label{eq:chain-flat}
 \sigma(\mathcal A)=n^{(v-1)/2},\qquad
 \nu(\mathcal A)=
 \begin{cases}
 n^{(v-1)/2},&Q\ge1,\\
 n^{(v-2)/2},&Q=0.
 \end{cases}
\end{equation}
\end{lemma}
\begin{proof}
Split each edge axis into its two endpoint-coordinate axes without
moving either coordinate across the row--column partition. After row
and column permutations, the coefficient matrix is a Kronecker product
of rolewise equality tensors. A role appearing on only one side gives
an equal-label vector with $n$ nonzero entries, hence norm $\sqrt n$.
A role appearing on both sides gives
$\sum_{i=1}^n e_i^{\otimes a}(e_i^{\otimes b})^*$, where $a,b\ge1$.
Both families of vectors are orthonormal, so this factor has norm one.
Every role appears, and multiplying these norms proves
\cref{eq:role-frontier}.

For $\sigma$, the main path connects matrix axes on opposite sides.
If no role crossed the partition, consecutive appearances along the
path would all be on the same side, a contradiction. Thus $|T|\ge1$.
Putting all noise axes on the row side makes only $x_{L+1}$ cross,
so the minimum is attained. For $\nu$, the role--axis incidence graph
is connected and meets both sides, hence again $|T|\ge1$.
When $Q\ge1$, putting a single leaf edge on the row side makes its
parent the only crossing role. When $Q=0$, every nonempty segment
of row edges has two frontier roles, counting an endpoint role when
the segment reaches a matrix boundary. Thus $|T|\ge2$, attained
by a single row edge. These observations give \cref{eq:chain-flat}.
\end{proof}

For unequal positive role sizes the same factorization gives
$\prod_{x\notin T}\sqrt{m_x}$. In particular, the resulting upper
bounds are uniform for $m_x\le n$. If a role class is empty, the
matrix is zero. This is the uniformity required by
\cref{lem:color-upper}; it is not an assertion that the globally
injective coefficient tensor has the same flattenings.

\subsection{Direct outputs of the specified literature bounds}
\label{app:separation-comparison}

All shape parameters are fixed in these comparisons. In
\citet[Theorem~20]{TulsianiWu2025}, specialize the Bernoulli density
to $1/2$, so the normalized edge variables are signs. With $t\ge2$,
the stated Schatten-moment bound specializes to
\[
 \E\norm{M_\alpha}_{S_{4t}}^{4t}
 \le (48tv)^{4tv}(Cd)^d n^v n^{2t(v-1)}.
\]
The density-dependent factor is one. Taking the $4t$-th root and using
$\norm{M}\le\norm{M}_{S_{4t}}$ and Jensen gives
$\E\norm{M_\alpha}\le C_\alpha t^v n^{(v-1)/2+v/(4t)}$.
Set $t=\max(2,\lceil\log n\rceil)$ to obtain
\begin{equation}\label{eq:chain-tw}
 \E\norm{M_\alpha}\le C_\alpha
 n^{(v-1)/2}(\log n)^v.
\end{equation}
This deduction uses an expectation moment bound, not a conversion
from a fixed-failure-probability statement.

For the typed degree-$d$ chaos, every noise vector has coordinate
dimension $m=n^2$ and the matrix dimensions are $n,n$.
Theorem~2.4 of \citet{BandeiraEtAl2025} gives
$C_d(\log n)^{d/2}\sigma(\mathcal A)$ for signs.
Their Theorem~2.5 gives
$C_d[\sigma(\mathcal A)+(\log n)^{(d+2)/2}\nu(\mathcal A)]$.
When $Q\ge1$, \cref{eq:chain-flat} makes the first bound smaller.
The unequal-size calculation above and \cref{lem:color-upper}
transfer it to the globally injective model.

The graph-specific Theorem~4.8 in that reference gives the same
output directly. Its logarithmic-count parameter is
\[
 s-|U\cap V|+|W\setminus(U\cup V)|-h=1+L+Q=d.
\]
For its prescribed intermediate-flattening construction, the unique
main path uses $L+1$ edges and covering the $Q$ off-path leaves
requires all $Q$ leaf edges. The iteration count is therefore also
$d$. These are comparisons with the displayed bounds in the cited
versions, not impossibility results for the underlying methods.

When $Q=0$, the situation is different: \cref{eq:chain-flat} gives
$\nu/\sigma=n^{-1/2}$, so the strong-NCK bound is already asymptotically
log-free. No logarithmic improvement over that output is asserted
for an undecorated path.

\subsection{The adapted product bound}

Identity \cref{eq:typed-decorated-product} follows by interchanging
finite sums, first summing each leaf coordinate. For a rectangular
sign matrix with both dimensions at most $n$, zero padding and the
net estimate in \cref{lem:path-count} give expected norm $O(\sqrt n)$.
For one leaf array, put $S_i=\sum_{a=1}^{m}\xi_{i,a}$, $m\le n$.
The sign exponential-moment bound yields
\[
 \Prob\{\max_{i\le n}|S_i|>u\}
 \le \min\{1,2n e^{-u^2/(2n)}\}.
\]
For $u_0=\sqrt{2n\log(2n)}$, tail integration gives
\[
 \E\max_i|S_i|
 \le u_0+2n\int_{u_0}^\infty e^{-u^2/(2n)}\,du
 \le u_0+\frac{n}{u_0}
 \le C\sqrt{n\log(2n)}.
\]
Each gate norm is at most the product of its leaf-array maxima.
All these arrays and all mixing matrices are independent. Consequently
submultiplicativity and independence give
\[
 \E\norm{H_{\alpha_{L,\mathbf q}}}
 \le (C\sqrt n)^{L+1}
       (C\sqrt{n\log(2n)})^Q.
\]
This estimate is uniform over all role sizes at most $n$;
\cref{lem:color-upper} proves \cref{eq:chain-product-upper}.
It is sharp for a concentrated decoration by \cref{eq:chain-sharp}.
For a distributed decoration its logarithmic excess is precisely
$(Q-q_{\max})/2$. Thus the comparison includes a direct structural
argument as well as the general literature bounds.

\section{Complete all-defect counting and encoding}\label{app:counting-detail}

This appendix proves \cref{thm:count}. The definitions of legal states,
defect, and partition intervals are in \cref{sec:counting}. We keep the
original backbone in the encoding record; the repaired backbone is used
only to bound the number of possible seeds.

\subsection{A worked exact-state expansion}

A fourth-moment calculation on a single typed edge illustrates what
the exact expansion retains. Let $X$ be an $m$ by $m$ sign matrix
and take $p=2$. The row and column trace matchings are the two
different matchings on four positions. They cannot both remain
unchanged in a legal state, because their meet has singleton
cells. There are exactly three possibilities: make the row
partition full, make the column partition full, or make both
full. Their respective label weights are
$m^2(m-1)$, $m^2(m-1)$, and $m^2$. Thus
\[
 \E\tr((XX^*)^2)=2m^3-m^2.
\]
The negative coefficient in the expanded polynomial does not
come from a negative replica contribution. All three exact-state
weights are nonnegative; it comes from expanding the falling
factorials. This is one reason to keep the exact positive
representation during comparisons.

The same example shows why a pairing in the code is not an
extra Wick pairing in the moment. An even equality block may
admit several matching refinements. We choose a canonical
refinement when encoding that block; we do not multiply its
Rademacher expectation by the number of possible refinements.
For Gaussian coordinates, by contrast, repeated-coordinate
moments do carry pairing multiplicities. The private-role
comparison used later transfers those multiplicities through
an explicit moment inequality, rather than identifying the
two expansions.

The defect parameter is similarly concrete. In the single-edge
fourth moment, the two leading states have three free label
blocks and defect zero. The fully identified state has two
free blocks and defect one. At higher orders, the path
estimate counts every allowed intermediate partition with
its actual defect. On a general graph, the off-path encoding
adds only a polynomial price per lost block. The counting
theorem is therefore a stability statement about the entire
neighborhood of the leading configurations, not merely a
description of those leading configurations.

This perspective makes the role of a logarithmic moment order
transparent. With a fixed defect penalty $n^{-1}$ and an
entropy increase $p^{K_r}$, all positive defects form a
controlled geometric tail when $p=O(\log n)$. The leading
factor $p^{a_*p}$ becomes $p^{a_*/2}$ after taking the
$2p$-th root. Both scales in the final norm formula are
already visible before optimizing any analytic concentration
inequality.

\subsection{Counting a path backbone}\label{app:count-path}

Before specializing to paths, we record exactly what positive evaluation
provides. This also explains why the negative power-basis coefficient
in the preceding fourth-moment example is harmless.

\begin{lemma}[Positive evaluation and coefficient extraction]
\label{lem:positive-calibration}
For each integer $p\ge2$, let $\Omega_p$ be finite, with integer
defects $\delta(\omega)\ge0$ and weights $w_\omega\ge0$.
For a real number $D_p$, suppose at an auxiliary integer $M\ge1$ that
\[
 P_p(M)=\sum_{\omega\in\Omega_p}W_\omega(M),\qquad
 W_\omega(M)\ge c_pw_\omega M^{D_p-\delta(\omega)},\quad c_p>0.
\]
If $a_{p,\delta}=\sum_{\delta(\omega)=\delta}w_\omega$, then
\begin{equation}\label{eq:positive-calibration}
 a_{p,\delta}\le c_p^{-1}M^{\delta-D_p}P_p(M).
\end{equation}
No expansion of the $W_\omega$ into signed monomials is required.
\end{lemma}
\begin{proof}
Retain only defect-$\delta$ states in the nonnegative sum and apply
their weight lower bound. Divide by $c_pM^{D_p-\delta}$.
\end{proof}

For a path of $l\ge1$ edges, the next proof applies this lemma with
$w_\omega=1$, $D_p=pl+1$, $M=2p^2$, and $c_p=e^{-(l+1)}$.
The exact weight is $\prod_x\fall{M}{|\pi_x|}$, not the corresponding
monomial. Thus a moment estimate for independent sign matrices can
calibrate the count without enumerating all path states. The general
graph still requires the off-path encoding and seed bound below.

Fix $s$ vertex-disjoint boundary-to-boundary paths, with unit capacity
on every vertex, including endpoints. The vertex form of max-flow
min-cut provides these paths. Common boundary vertices give singleton
paths. Trim other paths so that their internal vertices are not
boundary vertices. Let $P$ be their union and $Q=W\setminus P$.
A component of $G[Q]$ meets at most one boundary side, since otherwise
it would supply another vertex-disjoint path.

For the $i$th path put
$\delta_i=\sum_{x\in P_i}d(\pi_x)-(p-1)$.
By \cref{lem:coarea} on the path, $\delta_i\ge0$. Set
$\delta_{\rm on}=\sum_i\delta_i$ and
$D=\sum_{x\in Q}d(\pi_x)$. Then
\begin{equation}\label{eq:defect-split}
 \delta=\delta_{\rm on}+D.
\end{equation}
The vector consisting of the path defects and individual off-path
defects has at most $r$ coordinates. Its number of possibilities
is at most $2^{\delta+r}\le2^{rp}$, since
$\delta\le(r-s)(p-1)$.

\begin{lemma}[Path count]\label{lem:path-count}
A path with $l$ edges has at most
$100^{2p(l+1)}p^{2t}$ legal states with
$\sum_x|\pi_x|=pl+1-t$. A singleton path has one state, at $t=0$.
\end{lemma}
\begin{proof}
For $l\ge1$, evaluate its positive moment expansion at auxiliary
dimension $m=2p^2$. For $1\le b\le p$,
$\fall mb\ge e^{-1}m^b$, by
$\log(1-x)\ge-2x$ for $0\le x\le1/2$.
The typed path matrix is a product $X_1\cdots X_l$ of independent
$m$ by $m$ sign matrices.

To bound a single factor, take two $1/4$-nets of the unit sphere,
each of size at most $9^m$. Its operator norm is at most twice the
largest absolute bilinear form on the nets. Each such form is a
sign sum with squared coefficients summing to one. The exponential
moment inequality for signs and a union bound give
\[
 \Prob\{\norm X>2(t_0+z)\}\le e^{-z^2/2},
 \qquad t_0^2=4m\log9+2\log2\le11m .
\]
Tail integration and Minkowski imply
$\Lpnorm{\norm X}{2p}\le12\sqrt m$ for $p\le m$.
Independence and submultiplicativity, using order $2p$ for every
factor, now give
\[
 \E\tr((HH^*)^p)
 \le m\,\E\prod_{j=1}^l\norm{X_j}^{2p}
 \le12^{2pl}m^{pl+1}.
\]
Every state under consideration contributes at least
$e^{-(l+1)}m^{pl+1-t}$. Positivity therefore bounds its count by
$e^{l+1}12^{2pl}2^tp^{2t}$. Since $t\le l(p-1)$, the asserted
constant suffices.
\end{proof}

Ignoring non-backbone edges enlarges the set of allowed backbone
states. For fixed defects, \cref{lem:path-count} bounds their number by
\begin{equation}\label{eq:backbone-count}
 100^{2pr}p^{2\delta_{\rm on}}.
\end{equation}

\subsection{A lossless off-path encoding}

For perfect matchings define
$d_M(\rho,\sigma)=p-|\rho\vee\sigma|$.
Their union decomposes into alternating cycles, and a cycle of $t$
pairs takes $t-1$ two-pair switches to resolve. Thus $d_M$ is the
switch distance. A radius-$t$ ball has at most $(4p^2)^t$ elements
for $t\ge1$, and one for $t=0$.

If $\rho,\sigma$ both refine an even $\pi$, then
$d_M(\rho,\sigma)\le d(\pi)$. Across a legal edge $xy$, choose a
matching refining $\pi_x\wedge\pi_y$. For any chosen refinements
$\sigma_x,\sigma_y$, the triangle inequality gives
\begin{equation}\label{eq:matching-distance}
 d_M(\sigma_x,\sigma_y)\le d(\pi_x)+d(\pi_y).
\end{equation}
A partition with at most $p$ blocks has at most $p^{2b}$
coarsening records of $b$ binary merges. Across $r$ roles, records
of at most $B$ merges number at most $(2rp^2)^B$.
Finally, if $\xi$ refines $\pi$, then
\begin{equation}\label{eq:join-rank}
 |\pi|-|\pi\vee\rho|\le d_M(\xi,\rho).
\end{equation}
Indeed, joining with $\pi$ contracts the sequence of merges generating
$\xi\vee\rho$ from $\xi$.

Fix a component $C$ of $G[Q]$ and write $D_C=\sum_{x\in C}d(\pi_x)$.
Choose a deterministic matching refinement $\sigma_x$ of each
$\pi_x$. If $C$ meets $U$, anchor it there and set $\rho_C=\tau$;
if it meets $V$, set $\rho_C=\eta$. Otherwise use
$\rho_C=\sigma_w$ for a fixed anchor $w\in C$. Only the latter
seeds are free. Along a spanning tree, \cref{eq:matching-distance}
gives
\begin{equation}\label{eq:seed-distance}
 d_M(\rho_C,\sigma_x)\le2D_C\qquad(x\in C).
\end{equation}
At a boundary anchor the initial distance from the prescribed
matching is at most its own defect, which is included in this bound.

For an attachment edge $xy$ with $x\in P,y\in C$, choose $\xi$
refining $\pi_x\wedge\pi_y$. Then
$d_M(\rho_C,\xi)\le2D_C+d(\pi_y)\le3D_C$.
By \cref{eq:join-rank}, making $\rho_C$ refine the partition at $x$
costs at most $3D_C$ merges. Normalize the state by
\[
 \widehat\pi_y=\rho_C\quad(y\in C),\qquad
 \widehat\pi_x=\pi_x\vee
       \bigvee_{C:\,x\in N_P(C)}\rho_C\quad(x\in P).
\]
If $b$ is the total number of added backbone merges, then
\begin{equation}\label{eq:merge-budget}
 b\le3rD,\qquad
 \widehat\delta=\delta_{\rm on}+b.
\end{equation}
The normalized state is legal. Coarsening endpoint partitions
joins even intersection cells; within $Q$ the endpoint matchings
agree; across an attachment the seed refines the new backbone
partition. Boundary constraints are unchanged.

The code retains the \emph{original} backbone from
\cref{eq:backbone-count}, followed by its forward coarsening record,
at cost at most $(2rp^2)^{3rD}$. Once seeds are known, reconstruct
each original off-path partition by a matching within distance
$2D_C$ of its seed, then by $d(\pi_y)$ merges. The number of these
records is at most
\begin{equation}\label{eq:reconstruction}
 \prod_C\prod_{y\in C}(4p^2)^{2D_C}p^{2d(\pi_y)}
 \le(4p^2)^{2rD}p^{2D}.
\end{equation}
All original partitions are recoverable from these records.
No inverse splitting of a repaired backbone is assumed.

More explicitly, fix orders on roles, replicas, tree edges, and
components once and for all. Within each even cell choose the matching
that pairs successive replicas in that order. Choose the first shortest
switch word whenever several words have the same length; use the first
merge word realizing a required coarsening. These conventions make
the encoder a function of the original state, rather than a relation
with an unspecified multiplicity. The defect vector fixes the word
length budgets. Shorter words can be padded by a null operation; this
is absorbed by the factors $4p^2$ and $2rp^2$ already used above.

The decoder first reads the defect vector and the original backbone
partitions. It applies the recorded forward merges to those partitions,
obtaining the repaired backbone. It then reads one seed per component
(or uses the prescribed boundary seed). For each off-path role, its
switch word recovers the chosen matching refinement, and its merge
word recovers the original partition. Block identifiers can always be
the smallest replica in the current block, so every merge is unambiguous.
The resulting partitions at all roles are exactly the original state.
In particular, two states with the same record are equal. Invalid
records may be included when counting possibilities; they only enlarge
the upper bound. No step of this decoder needs to infer an original
partition by splitting a join.

\subsection{Charging the seed entropy}

Fix the original backbone and a forward coarsening that permits a
normalized state. For a boundary-free component $C$ of $Q$, put
\[
 \theta_C=\bigwedge_{x\in N_P(C)}\widehat\pi_x,
 \qquad e_C=p-|\theta_C|.
\]
Its neighbor set is nonempty, because the core has no detached
component. A compatible seed refines $\theta_C$, so its cells must
be even. If their sizes are $2t_1,\ldots,2t_j$, the seed count is
\begin{equation}\label{eq:seed-count}
 \prod_{i=1}^j(2t_i-1)!!
 \le(2p)^{\sum_i(t_i-1)}=(2p)^{e_C}.
\end{equation}

Apply the separator layers to the normalized state. By interval
containment, every layer crossing $I(\theta_C)$ contains all
backbone neighbors of $C$. No role of $C$ lies in any layer,
because its partition is a matching. Consequently $C$ is an
active component at each of these $e_C$ layers. At a minimum
layer there are at most $a_*$ such components. There are at most
$\widehat\delta$ nonminimum layers, each with at most $r$
components. Thus
\begin{equation}\label{eq:seed-charge}
 \sum_Ce_C\le a_*(p-1)+r\widehat\delta.
\end{equation}
The crude bound $\sum_Ce_C\le r(p-1)$ absorbs the powers of two
in \cref{eq:seed-count}. All free seeds together have at most
\[
 2^{rp}p^{a_*p+r\delta_{\rm on}+3r^2D}
\]
choices. Boundary-touching components have prescribed seeds.

The coarsening can depend on the seed in the construction. This
does not invalidate the count: first count its range, then count
the fiber above each fixed coarsening. No probabilistic independence
between these combinatorial data is used.

Formally, let $\mathcal B$ denote a recorded original backbone and
forward merge word. For each fixed $\mathcal B$ in the range of the
encoder, its repaired partitions, and hence the partitions $\theta_C$,
are determined. Sum the seed bound over this range, and for each seed
sum the local reconstruction bound. The number of full records is at
most the number of possible $\mathcal B$ times the largest such fiber
bound. This elementary range--fiber inequality is valid even though
the map producing $\mathcal B$ used the seed. It is the reason the
normalization does not reverse the counting direction.

\begin{proof}[Completion of \cref{thm:count}]
Multiply the counts for the defect vector, the backbone, its
forward merges, the off-path reconstruction, and the seeds.
The exponent of $p$ is
\[
 a_*p+(r+2)\delta_{\rm on}+(3r^2+10r+2)D
 \le a_*p+K_r\delta.
\]
The remaining factors are bounded by
\[
 2^{2rp}100^{2pr}(2r)^{3rD}4^{2rD}
 \le\bigl[200^r(2r)^{2r^2}4^{r^2}\bigr]^{2p},
\]
using $D\le rp$. If $r=0$, only the empty state at defect zero
exists. This proves all cases.
\end{proof}

\section{Conditional deviations and synchronized lower witnesses}\label{app:lower-detail}

We complete the conditional estimate used after
\cref{lem:separator-flattening}. All internal randomness is exposed
before the separator trials are selected. This order is essential:
unconditional independence of those trials would be false.

\subsection{A uniform conditional moderate deviation}\label{app:lower-conditional}

Fix an active connected component $K$ with $k$ roles, and let
$B\ne\varnothing$ be its roles adjacent to $S$. Condition on all
internal arrays of $K$. Summing over $K\setminus B$ gives
coefficients $\gamma_{x_B}$. At a fixed separator label,
\[
 T_K=\sum_{x_B}\gamma_{x_B}\prod_{z\in B}\eta_z(x_z),
\]
where the $\eta_z$ are independent sign vectors: each is the product
of the crossing-edge signs incident to that role.

We use the fixed-degree Boolean hypercontractive inequality
$\|P\|_q\le(q-1)^{d/2}\|P\|_2$ for a sign polynomial of degree at
most $d$ and $q\ge2$ \citep{ODonnell2014}. Here all applications have
fixed degree and fixed $q$; their constants do not grow with $n$.

Choose $z_0\in B$ and put
\[
 Q=\sum_{x_B}\gamma_{x_B}^2,\qquad
 R_j=\sum_{x_{B\setminus\{z_0\}}}
              \gamma_{j,x_{B\setminus\{z_0\}}}^2 .
\]
Distinct summation assignments give distinct Walsh monomials,
because $K$ is connected. For a singleton component use
$\gamma=1$. Thus $\E Q=\prod_{x\in K}m_x\asymp n^k$ and
$\E R_j\asymp n^{k-1}$. Hypercontractivity and Minkowski give
\[
 \E Q^2\le C(\E Q)^2,\qquad
 \E R_j^{16}\le C n^{16(k-1)}.
\]
Paley--Zygmund, followed by Markov and a union bound, yields an
internal event $\mathcal G_K$ of probability at least $c>0$ on
which
\begin{equation}\label{eq:internal-good}
 cn^k\le Q\le Cn^k,\qquad
 \max_jR_j\le n^{k-3/4}.
\end{equation}
The last failure probability is $O(n^{-3})$: its one-coordinate
bound is $O(n^{-4})$ and there are $O(n)$ coordinates. Increasing
the constant in the upper bound for $Q$ leaves a positive
probability for the intersection.

Here is the good-set calculation with the constants separated from
the dimension. If $\E Q\ge c_1n^k$ and
$\E Q^2\le C_1(\E Q)^2$, Paley--Zygmund gives
$\Prob\{Q\ge(\E Q)/2\}\ge(4C_1)^{-1}=:q_1$.
Choose a fixed $A$ so that $\Prob\{Q>An^k\}\le q_1/4$ by
Markov. For each $j$, the sixteenth-moment bound gives
$\Prob\{R_j>n^{k-3/4}\}\le Cn^{-4}$; summing over at most
$C'n$ indices leaves $O(n^{-3})\le q_1/4$ for large $n$.
The intersection therefore has probability at least $q_1/2$.
All choices depend on the fixed component and label-size comparison
constants, not on a separator tuple.

For every full internal realization in $\mathcal G_K$, expose the
vectors $\eta_z$ for $z\ne z_0$. The remaining sum is
$Y=\sum_jz_j\eta_{z_0}(j)$, with
$\sigma^2=\sum_jz_j^2$. Uniformly over those internal realizations,
\[
 \E\sigma^2=Q,\quad \E\sigma^4\le CQ^2,\quad
 \E|z_j|^{32}\le C R_j^{16}\le Cn^{16k-12}.
\]
A second Paley--Zygmund argument, an upper Markov bound, and the
union bound at $|z_j|=n^{k/2-1/4}$ show that, with conditional
probability at least $c'>0$,
\begin{equation}\label{eq:conditional-good}
 c'n^k\le\sigma^2\le C'n^k,\qquad
 \max_j|z_j|\le n^{k/2-1/4}.
\end{equation}
The coefficient failure probability is again $O(n^{-3})$.
If $|B|=1$, there is no exposure at this step; the assertions follow
directly from \cref{eq:internal-good}.

For clarity, after a full internal realization is fixed, the same
Paley--Zygmund argument applies to the nonnegative random variable
$\sigma^2$, whose conditional mean is $Q$. Its second moment is
$\E\sigma^4\le CQ^2$. A sufficiently large fixed upper cutoff
removes less than a quarter of the resulting positive probability.
For the coefficients the displayed thirty-second moment yields
$\Prob\{|z_j|>n^{k/2-1/4}\mid\mathcal I\}\le Cn^{-4}$,
because the denominator is $n^{16k-8}$ and the numerator is
at most $Cn^{16k-12}$. The union bound again costs only $O(n)$.
These estimates hold for every internal realization in
$\mathcal G_K$, not merely on average over that event.

\begin{lemma}[Weighted sign tilt]\label{lem:tilt}
For $Y=\sum_jz_j\eps_j$ and $\sigma^2=\sum_jz_j^2>0$, suppose
$t\ge\sigma$ and $8t\max_j|z_j|/\sigma^2\le1$. Then
\[
 \Prob\{Y\ge t\}\ge\tfrac12\exp(-96t^2/\sigma^2).
\]
\end{lemma}
\begin{proof}
Tilt by $\lambda=8t/\sigma^2$, so the density is
$e^{\lambda Y}/\E e^{\lambda Y}$. Since
$x/2\le\tanh x\le x$ on $[0,1]$, the tilted mean lies between
$4t$ and $8t$, and the tilted variance is at most $\sigma^2$.
Chebyshev gives probability at least
$1-1/9-1/16>1/2$ to the band $t\le Y\le12t$.
On this band the inverse likelihood ratio is at least
$e^{-12\lambda t}$, since $\sum_j\log\cosh(\lambda z_j)\ge0$.
Integrating on the band proves the claim.
\end{proof}

Take $t=\eps n^{k/2}\sqrt{\log n}$. For each fixed $\eps>0$,
\cref{eq:conditional-good} satisfies the hypotheses for all
sufficiently large $n$, uniformly in the internal realization.
Integrating only over the event in \cref{eq:conditional-good} gives
\begin{equation}\label{eq:component-tail}
 \Prob\{|T_K(x_S)|\ge\eps n^{k/2}\sqrt{\log n}
             \mid\mathcal I\}
 \ge c n^{-C\eps^2}
 \quad\hbox{on }\mathcal G_K ,
\end{equation}
where $\mathcal I$ contains all internal arrays. The constants and
the threshold for $n$ are uniform over the full realizations in
$\mathcal G_K$.

\subsection{Synchronizing the witnesses}

Let $K_1,\ldots,K_a$ be the active components. Choose
$N=\Theta_G(n)$ separator tuples whose labels are distinct in
every separator role across the tuples. Condition on all
component-internal arrays. Crossing coordinates for distinct
tuples and distinct components are disjoint, so these trials
are independent under this full conditional law.

The internal good events use disjoint arrays and have joint
probability at least a positive constant. On this joint event,
\cref{eq:component-tail} gives success probability at least
$c n^{-C\eps^2}$ for simultaneous success in all $a$ components
at any one tuple. Choose $\eps$ so that $C\eps^2<1/2$. The
conditional probability that no tuple succeeds is then at most
$\exp(-c\sqrt n)$. Therefore
\begin{equation}\label{eq:synchronized}
 \E\max_{x_S}\prod_{j=1}^a|T_{K_j}(x_S)|
 \ge c_G n^{|D|/2}(\log n)^{a/2}.
\end{equation}
This integration uses independence given $\mathcal I$, not
independence given only the good event. Explicitly, if
$\Prob(\mathcal G)\ge c_0$ and each of $N$ conditionally independent
trials has conditional success probability at least $\rho_n$
on every full realization in $\mathcal G$, then
\[
 \Prob(\mathcal G\cap\{\hbox{some success}\})
 \ge c_0[1-(1-\rho_n)^N].
\]
If $a=0$, the product is one. If $S$ is empty, then $a=0$ and
no packing is needed.

\section{Factor moments, layerwise counting, and lower tails}\label{app:factor-detail}\label{app:tail}

This appendix proves \cref{thm:tail-factor,prop:profile} and the
factor statement \cref{thm:intro-factors}. The law and preprocessing are
those of \cref{sec:extensions}; the separate weighted sign theorem is
not an additional hypothesis here.

The proof retains the original separator layers before summing over
replica states. This is the point at which a graph-only count would lose
the factor-specific information. We first account for scalar components,
then prove the layerwise count and the conditional lower bound.
\subsection{Detached sign components}

We first prove the scalar count needed for the moment statement
\cref{eq:intro-moments}. For a connected detached graph $K$ with
$k$ roles and at least one edge, let $\delta=\sum_x(p-|\pi_x|)$.
There are constants $B,L$ depending only on $K$ such that the
number of legal scalar states is at most
\begin{equation}\label{eq:scalar-count}
 B^{2p}p^{p+L\delta}.
\end{equation}
To see this, fix a rooted spanning tree and a canonical matching
refinement at every role. The root matching has at most $(2p)^p$
choices. Across a tree edge, \cref{eq:matching-distance} bounds
the matching distance by the sum of endpoint defects. Counting
switch records gives at most $p^{3(d_x+d_y)}$ choices for a
child refinement. Recover its original partition with at most
$p^{2d_y}$ merge records, and recover the root partition using
at most $p^{2d_{\rm root}}$ additional merge records.
Multiplying over the tree gives
$(2p)^pp^{(3k+2)\delta}$. Defect allocations contribute a fixed
exponential base in $p$, proving \cref{eq:scalar-count}.

The positive scalar expansion consequently gives
\[
 \E|Z_K|^{2p}\le
 B^{2p}n^{pk}p^p\sum_{\delta\ge0}(p^L/n)^\delta .
\]
For a matching lower bound, impose the same perfect matching at
every role. These $(2p-1)!!$ distinct exact states are all legal,
each with weight $\prod_x\fall{m_x}p$. For
$p\le\min_xm_x/2$ this weight is at least
$\prod_x(m_x/2)^p$, and $(2p-1)!!\ge p!\ge(p/e)^p$.
Taking roots and using monotonicity between even moments proves
\begin{equation}\label{eq:scalar-moments}
 \|Z_K\|_q\asymp n^{k/2}\sqrt q,\qquad
 2\le q\le C_0\log(2n).
\end{equation}
At $q=1$, second and fourth moments give
$\E|Z_K|\asymp n^{k/2}$. For a higher-arity sign component,
factor legality implies two-section legality for the upper count,
and the common-matching states remain legal for the lower count.
A unary-only singleton is an ordinary sign sum and satisfies the
same assertion.

In the typed model, \cref{eq:factor-decomposition} makes these
$L^q$ norms multiply independently of the core.
\Cref{prop:core-upper} and the core $L^1$ lower bound supply its
sharp logarithmic factor. The global upper and lower transfers
have constants independent of $q$, proving
\cref{eq:intro-moments} with $d_0$ equal to the number of detached
random components. For $1\le q<2$, sandwich between $L^1$ and
$L^2$. Markov at $q=C_A\log(2n)$ gives the upper-tail exponent
$(a_*+d_0)/2$.

\subsection{Fixed moments for independent factor laws}

We need a dimension-free fixed-order replacement for Boolean
hypercontractivity. If independent symmetric variance-one
variables have bounded $2M$ moments, then for Hilbert-space vectors
$v_i$,
\begin{equation}\label{eq:hilbert-general}
 \left\|\sum_i\xi_iv_i\right\|_{L^{2M}}
 \le A_M\left(\sum_i\|v_i\|^2\right)^{1/2}.
\end{equation}
Expand the $M$ inner products in the norm power. Odd multiplicities
vanish by symmetry. Bound inner products by products of lengths.
Every surviving index word admits a pairing, and summing over
all $(2M-1)!!$ pairings only overcounts. The bounded variable
moments contribute a constant depending on $M$, not on the number
of indices. This proves the inequality.

Iterate it through the fixed number of independent occurrence
groups, using a Hilbert direct sum for coefficient lists and
Minkowski for their squared norms. Orthogonality identifies the
final coefficient norm with the $L^2$ norm. Thus every such
Hilbert-valued multilinear chaos satisfies
$\|T\|_{2M}\le A_{M,\mathcal F}\|T\|_2$, also for deterministic
coefficients produced by conditioning. Interpolation at orders
one, two, and four gives the lower stacking inequality of
\cref{lem:stacking} with a law-dependent positive constant per
occurrence.

\subsection{The layerwise refinement}

For a fixed original separator sequence $\mathbf S=(S_k)$, let
$N(\mathbf S)$ be the number of graph-legal states inducing it.
Then
\begin{equation}\label{eq:layerwise}
 N(\mathbf S)\le C_r^{2p}
 p^{\sum_k a(S_k)+K'_r\sum_k(|S_k|-s)},
 \qquad K'_r=3r^2+11r+2.
\end{equation}
Here the statistics refer to the original, not normalized, slices.

Use the encoding of \cref{app:counting-detail}. An effective backbone
merge increases interval length by one and preserves containment.
Normalizing off-path roles deletes all their $D$ memberships.
Thus, with $b\le3rD$,
\[
 \sum_k|S_k\triangle\widehat S_k|=b+D\le(3r+1)D.
\]
Since an active count is between zero and $r$, it follows that
\[
 \sum_k a(\widehat S_k)
 \le\sum_k a(S_k)+(3r^2+r)D.
\]
The seed interval argument now gives
$\sum_Ce_C\le\sum_k a(\widehat S_k)$ instead of
\cref{eq:seed-charge}. The normalized slices depend only on
the repaired backbone: all off-path intervals have width zero,
regardless of the position of their singleton interval.
Therefore this inequality is valid in a fiber with fixed
original slices. Count only repaired-backbone fibers admitting
a completion with those original slices.

Adding the path, coarsening, reconstruction, and seed exponents
gives
\[
 \sum_k a(S_k)+2\delta_{\rm on}
                +(3r^2+11r+2)D .
\]
The same fixed-base and defect-vector estimates as before prove
\cref{eq:layerwise}. This proves the required refinement without
assuming it follows from the final unweighted bound.

\subsection{Weighted trace expansion}

For occurrence $e$, put $\lambda_e=\bigwedge_{v\in e}\pi_v$.
Its cells of size $j$ contribute $\E\xi_e^j$ in the exact trace.
Symmetry removes all states with an odd factor cell; every
remaining term is nonnegative. A covered role is even as a union
of even factor cells, and an uncovered core role is even by its
boundary matching. Factor legality implies pairwise legality
in the two-section.

Write the cell sizes of $\lambda_e$ as $2t_1,\ldots,2t_j$ and
$d_e=\sum_i(t_i-1)=p-|\lambda_e|$. Tail regularity implies
\begin{equation}\label{eq:factor-weight}
 \prod_i\E|\xi_e|^{2t_i}\le \widetilde C_e^{2p}p^{\theta_ed_e}.
\end{equation}
Indeed $(2t)^{\theta t}\le2^{2\theta t}p^{\theta(t-1)}$ for
$1\le t\le p$. Under the finite profile, the corresponding bound
is $C_e^{2p}n^{\beta_ed_e}p^{\gamma_ed_e}$.

By interval containment, every layer in $I(\lambda_e)$ contains
all roles of $e$. Hence, for nonnegative weights $w_e$,
\[
 \sum_ew_ed_e\le\sum_k\sum_{e\subseteq S_k}w_e.
\]
The label contribution is at most
$C^{2p}n^{pr-\sum_k|S_k|}$. Combining these statements with
\cref{eq:layerwise} proves \cref{eq:profile}.

For fixed laws, minimum layers have charge at most $b_*$.
Nonminimum layers are at most $\delta$ in number, and their
charges are bounded by a fixed constant $B$. There are at most
$2^{r(p-1)}$ subset sequences. Thus
\[
 \E\tr((H_cH_c^*)^p)
 \le C^{2p}n^{p(r-s)+s}p^{b_*p}
       \sum_{\delta\ge0}(p^{K'_r+B}/n)^\delta .
\]
Taking $p=\lceil\log n\rceil$ proves the core upper bound of
\cref{thm:tail-factor}. Increasing the threshold constant in
Markov's inequality gives each prescribed polynomial upper tail.

For the cut-stable profile, put
$F(S)=b(S)+K'_r(|S|-s)$. On a minimizing cut,
$F(S)\le b_*^{\rm mw}$. If nonminimizers exist, their cost
gap has a positive minimum $\eta$, since the shape is fixed.
Choose a fixed $L$ bounding $(F(S)-b_*^{\rm mw})_+$. The
sequence sum is at most
\[
 n^{pr-(p-1)\kappa_*}p^{(p-1)b_*^{\rm mw}}
 \left[|\mathcal S_*|+
 (|\mathcal S|-|\mathcal S_*|)n^{-\eta}p^L\right]^{p-1}.
\]
Take $p=\lfloor c\log n\rfloor$ for fixed $0<c<C_0$. The
bracket and the root factor $n^{\kappa_*/(2p)}$ are bounded.
This proves the expectation bound. Markov at threshold $R$
times the root estimate gives $R^{-2p}$; choosing $R$ large
enough gives $n^{-A}$ without enlarging the supplied window.

\subsection{The factor lower bound}

Fix a minimum separator maximizing $a(S)+\Theta(S)$. Condition
all occurrences meeting active components and all occurrences
contained in $S$. Their weight is
\[
 w(x_S)=\prod_{K\ {\rm active}}T_K(x_S)
             \prod_{e\subseteq S}\xi^{(e)}_{x_e}.
\]
A factor cannot meet two components of the two-section minus
$S$ outside $S$. Assign the remaining occurrences to the
boundary side met by their nonseparator roles. Minimality
of $S$ observes each separator role on both sides, by the
path argument in \cref{lem:separator-flattening}. The weighted
coefficient flattening is therefore exactly the same direct
sum of weighted all-ones blocks. The general-law stacking
inequality reduces the lower bound to one simultaneous
maximum of $|w(x_S)|$.

For completeness, the needed active-component estimate follows
with full conditioning as follows. Fix an active component $K$
of $k$ roles and a crossing occurrence $e_0$. Put
$B=e_0\cap K$, $b=|B|\ge1$. Let $\mathcal I$ contain the full
arrays of occurrences wholly inside $K$, and let $O$ comprise
the other crossing slices. For a fixed separator tuple write
\[
 T_K=\sum_{y\in I_B}z_y(\mathcal I,O)
                  \xi^{(e_0)}_{x_{e_0\cap S},y},
 \quad \sigma^2=\sum_yz_y^2,\quad
 V(\mathcal I)=\E(T_K^2\mid\mathcal I).
\]
Every role of $K$ is covered by a factor; every role outside
$B$ occurs in a factor other than $e_0$. Distinct complete
assignments therefore give orthogonal monomials. The
fixed-moment inequality gives
\[
 \E V\asymp n^k,\quad \E V^2\le Cn^{2k},\quad
 \E|z_y|^{2M}\le C_Mn^{M(k-b)}.
\]
Paley--Zygmund and an upper cutoff give an internal set of
positive probability on which $cn^k\le V\le Cn^k$.
On each of its full realizations,
$\E_O\sigma^2=V$ and $\E_O\sigma^4\le CV^2$. Hence
$cn^k\le\sigma^2\le Cn^k$ has conditional probability at
least a fixed $q_0>0$.

The coefficient maximum has unconditional failure probability
\[
 \Prob\{\max_y|z_y|>n^{k/2-1/4}\}
 \le C_M n^{b-Mb+M/2}=o(1)
\]
for sufficiently large fixed $M$. Its conditional failure
probability, as a function of $\mathcal I$, has expectation
$o(1)$. Markov therefore removes an $o(1)$ set of internal
realizations and leaves a set $\mathcal G_K$ of positive
probability on which
\[
 \Prob_O\{cn^k\le\sigma^2\le Cn^k,\
              \max_y|z_y|\le n^{k/2-1/4}\mid\mathcal I\}
 \ge q_0/2 .
\]
The crossing slices have identical marginal laws for every
separator tuple, so the same full internal set works for all
tuples.

To make the conditioning step explicit, let $q_n(\mathcal I)$ be
the conditional probability of the coefficient failure event.
Its expectation is $o(1)$ by the preceding union bound. Hence
$\Prob\{q_n(\mathcal I)>q_0/2\}\le2\E q_n/q_0=o(1)$.
Remove this set from the positive-probability internal variance
set. On each remaining full realization, intersecting the variance
event with the coefficient event loses at most $q_0/2$ of conditional
probability. There is no union bound over separator tuples here:
given the same internal arrays, the as-yet unexposed crossing slices
have identical product laws for every fixed tuple. Their conditional
failure-probability functions are therefore the same function of
$\mathcal I$. Independence is used later only for tuples with
disjoint crossing coordinates.

The tail assumption with $\theta_e\le2$ implies
$\|\xi_e\|_q\le Cq$. Its moment generating function is analytic
near zero, by the absolute power series and $m!\ge(m/e)^m$.
For $k(u)=\log\E e^{u\xi_e}$, symmetry and variance one give
$k(u)=u^2/2+O(u^4)$,
$k'(u)=u+O(u^3)$, and $k''(u)=1+O(u^2)$.
For $T=\sum_i z_i\xi_i$, tilt with
$\lambda=2t/\sigma^2$. If $t\ge4\sigma$ and
$\lambda\max_i|z_i|$ is sufficiently small, the tilted mean
lies in $[1.8t,2.2t]$ and its variance is at most $2\sigma^2$.
Chebyshev gives tilted probability at least $1/2$ to
$t\le T\le3t$. Since the logarithmic moment generating
function is nonnegative, undoing the tilt gives
\[
 \Prob\{T\ge t\}\ge\tfrac12e^{-6t^2/\sigma^2}.
\]
At $t=\eps n^{k/2}\sqrt{\log n}$ the hypotheses hold uniformly
on the coefficient event above. Thus on every realization in
$\mathcal G_K$ the component success probability is at least
$c n^{-C\eps^2}$.

For a separator-only factor, Paley--Zygmund on $|\xi_e|^q$
and the two-sided moment bounds give
\[
 \Prob\{|\xi_e|\ge\tfrac12c_eq^{\theta_e/2}\}
 \ge\tfrac14
       \left(\frac{c_e}{C_e2^{\theta_e/2}}\right)^{2q}.
\]
Choose $q=\alpha_e\log n$ with fixed small $\alpha_e>0$.
This gives the required logarithmic threshold with an
arbitrarily small positive polynomial success exponent.

Intersect the independent internal good sets. On every full
realization in their intersection, use $\Theta(n)$ separator
tuples with all role labels distinct across trials. Every
crossing or separator-only occurrence has a nonempty
separator support, so its coordinates are disjoint across
these trials. Choose $\eps$ and the finitely many $\alpha_e$
so that their total success exponent is less than one.
Conditional independence then gives
\[
 \E\max_{x_S}|w(x_S)|
 \ge c n^{|D|/2}(\log n)^{[a(S)+\Theta(S)]/2}.
\]
For an empty separator there is no active or separator-only
factor and $w=1$. The flattening proves the core lower bound.

Finally, for a fixed-law detached component,
$\E Z_K^2=\prod_{v\in K}m_v$ and
$\E Z_K^4\le C(\prod_{v\in K}m_v)^2$. Interpolation gives
$\E|Z_K|\asymp n^{|K|/2}$. Restore these independent
expectation factors and unused-role sizes in
\cref{eq:factor-decomposition}. This completes
\cref{thm:tail-factor}, and hence \cref{thm:intro-factors}.

The full trace restoration is exact:
\[
 \E\tr((HH^*)^p)=
 d(n)^{2p}\prod_j\E|Z_j|^{2p}\,
                \E\tr((H_cH_c^*)^p).
\]
For example, if $H=ZI_n$ and
$Z=\sum_{i,j}g_{ij}\sim N(0,n^2)$, its trace moment is
$n^{2p+1}(2p-1)!!$. Omitting the detached moment would lose
this factorial, which cannot be absorbed into $C^{2p}$.

\subsection{Detached scalars and the moment window}

The distinction between a core matrix and detached scalar factors also
persists beyond expectation. If $d_0$ is the number of connected
detached random components, the moment statement has the form
\begin{equation}\label{eq:intro-moments}
 \Lpnorm{\norm{M_\alpha}}q
 \asymp_{\alpha,C_0}
 n^{(v+h-s)/2}(\log n)^{a_*/2}q^{d_0/2},
 \qquad 2\le q\le C_0\log(2n).
\end{equation}
The logarithmic factor is already present at fixed $q$. Replacing it by
$q^{a_*/2}$ would contradict the diagonal example above. Conversely, a
high-probability lower bound for the full matrix cannot be obtained by
ignoring the possibility that a detached scalar is small. Expectation,
growing moments, and overwhelming-probability statements require separate
quantifiers.

\section{Dimension-free bounded-noise comparison}\label{app:bounded-detail}

\begin{proof}[Proof of \cref{thm:bounded-comparison}]

First, for deterministic vectors $b_i\in B$ and $|a_i|\le1$,
\begin{equation}\label{eq:bounded-contraction}
 \Lpnorm{\sum_i a_i\eps_i b_i}p
 \le \Lpnorm{\sum_i\eps_i b_i}p.
\end{equation}
Choose independent signs $\eta_i$ with $\E\eta_i=a_i$. Conditional
Jensen for the convex function $\norm{\cdot}^p$ proves the inequality,
since the products $\eps_i\eta_i$ are independent uniform signs.
If $0<m\le a_i\le M$, applying this inequality both to $a_i/M$
and to $m/a_i$ gives two-sided factors $m$ and $M$.

For one group, take an independent copy $X'_i$ and put $D_i=X_i-X'_i$.
These differences are independent and symmetric, so jointly they have
the law $D_i=\eps_i R_i$, with $R_i=|D_i|$ and independent uniform
signs independent of all magnitudes. This representation also holds
when $R_i=0$. Since $R_i\le2K$ and $\E R_i^2=2$,
\[
 \mu_i:=\E R_i\ge K^{-1}.
\]
Conditional Jensen using $\E X'_i=0$, followed by contraction conditional
on the magnitudes, yields
\[
 \Lpnorm{\sum_iX_ib_i}p
 \le\Lpnorm{\sum_iD_ib_i}p
 \le 2K\Lpnorm{\sum_i\eps_ib_i}p.
\]
For the reverse inequality, Jensen over the magnitudes and lower
contraction give
\[
 \Lpnorm{\sum_iD_ib_i}p
 \ge\Lpnorm{\sum_i\mu_i\eps_ib_i}p
 \ge K^{-1}\Lpnorm{\sum_i\eps_ib_i}p.
\]
Minkowski bounds the left side by $2\Lpnorm{\sum_iX_ib_i}p$.
This proves the comparison for one group. Condition on all other
groups, apply these inequalities to conditional $p$th moments, and
integrate. Replacing the groups one at a time gives exactly $Q$
factors of $2K$, proving \cref{eq:bounded-comparison}.
\end{proof}
\section{An independent proof for the gated Khatri--Rao product}\label{app:khatri-detail}

\begin{proof}[Proof of \cref{thm:khatri-rao}]

Write $\circ$ for entrywise matrix multiplication. If $P,Q$ are
positive semidefinite and $P=\sum_\ell u_\ell u_\ell^*$, then
\[
 P\circ Q=\sum_\ell\diag(u_\ell)Q\diag(u_\ell)^*
 \preceq\norm Q\,\diag(P).
\]
Using $K^*K=(A_1^*A_1)\circ\cdots\circ(A_r^*A_r)$ and iterating gives
\begin{equation}\label{eq:khatri-upper}
 \norm K\le\norm{A_1}\prod_{t=2}^r\max_j\norm{A_t[:,j]}_2.
\end{equation}
For a standard Gaussian matrix $A$, $\E\norm A\le C\sqrt n$.
For completeness, a $1/4$-net of the unit sphere has at most $9^n$
points by disjoint-ball volume comparison. Approximating both singular
vectors bounds $\norm A$ by twice the largest $|x^*Ay|$ on the two
nets. Each fixed form is standard normal. The union bound
$2\cdot9^{2n}\exp(-t^2/2)$, capped at one and integrated, gives the
claimed estimate. Since each column norm is at most $\norm A$,
independence in \cref{eq:khatri-upper} implies
$\E\norm K\le C_r n^{r/2}$.

Choose $J$ as the smallest maximizer of $|D_{jj}|$. It depends only
on $D$, and is independent of every $A_t$. The $J$th column gives
\[
 \norm{KD}\ge |D_{JJ}|\prod_t\norm{A_t[:,J]}_2.
\]
For a standard Gaussian vector $g\in\R^n$, the second and fourth
moments of its Euclidean norm are $n$ and $n(n+2)$. Paley--Zygmund
therefore gives $\E\norm g_2\ge c\sqrt n$. Conditioning on $D$
and its selected column, and using the upper bound above, yields
\begin{equation}\label{eq:khatri-max}
 c_r n^{r/2}\E\max_j|D_{jj}|
 \le\E\norm{KD}
 \le C_r n^{r/2}\E\max_j|D_{jj}|.
\end{equation}
For $p=\max(2,\log n)$, maximum is bounded by the $\ell_p$ norm, so
\[
 \E\max_j\prod_h|g_{jh}|
 \le n^{1/p}\prod_h\Lpnorm{g_{1h}}p
 \le C_q(\log(2n))^{q/2}.
\]
Conversely choose $\alpha>0$ with $\alpha q/2<1$ and set
$t=\sqrt{\alpha\log n}$. Integrating the normal density on
$[t,t+1/t]$, for $t\ge1$, gives
$\Prob\{|g|\ge t\}\ge c t^{-1}e^{-t^2/2}$.
At each column, all $q$ gates exceed $t$ in absolute value with
probability at least $c_qt^{-q}n^{-\alpha q/2}$.
Its product with $n$ tends to infinity. Independence across columns
thus gives a gate product at least $t^q$ with probability tending
to one. This proves the reverse maximum estimate and, by
\cref{eq:khatri-max}, the theorem.
\end{proof}
\section{Local contraction and the complete aspect reduction}\label{app:weights-detail}\label{app:aspects}

We prove \cref{thm:aspects} by establishing the reductions in their required order.

\subsection{All-moment local contraction}

For coefficients in any normed space, write
$F_b=\sum B_{i_1,\ldots,i_Q}\prod_t b_t(i_t)\eps^{(t)}_{i_t}$.
If $0<c_t\le|b_t(i)|\le C_t$, then for every $p\ge1$,
\begin{equation}\label{eq:local-contraction}
 \left(\prod_tc_t\right)\|F_1\|_p
 \le\|F_b\|_p\le
 \left(\prod_tC_t\right)\|F_1\|_p.
\end{equation}
For $|d_i|\le1$, choose independent auxiliary signs of mean $d_i$.
Conditional convexity bounds
$\|\sum_i d_i\eps_i v_i\|^p$ by the auxiliary expectation of
$\|\sum_i\eta_i\eps_iv_i\|^p$. The products $\eta_i\eps_i$
are independent uniform signs. Average and iterate through the
groups with $d_i=b_t(i)/C_t$. For the reverse inequality,
contract the weighted chaos by $c_t/b_t(i)$.

Apply \cref{eq:local-contraction} with
$b_t=n^{-\gamma_t}a_{t,n}$. This removes all amplitudes with
a factor $n^\Gamma$ and constants independent of both $p$ and
$n$. It is done before selecting separator labels; no assertion
of identical conditional weighted laws is needed.

\subsection{Positive aspects and the minimum face}

Assume first that all remaining roles have positive exponents
and all factors have arity at least two. Remove sign factors
on a single boundary, deterministic isolated middle roles,
and detached components. Let $G$ be the resulting core and
$B_G=\sum_{v\in G}\beta_v$.

Factor-legal states are legal in its two-section. Coarea gives
the weighted defect
\[
 \Delta_\beta
 =\sum_v\beta_v(p-|\pi_v|)-\kappa(p-1)
 =\sum_k(\beta(S_k)-\kappa)\ge0,
\]
and hence
$\sum_v\beta_v|\pi_v|
 =p(B_G-\kappa)+\kappa-\Delta_\beta$.

Let $\mathcal M$ be the full family of minimum-cost separators
and fix $S_0\in\mathcal M$. The subspace
\[
 L=\{x:x(S)=x(S_0)\text{ for every }S\in\mathcal M\}
\]
is defined over the rationals. Positive coordinates and the
strict inequalities $x(T)>x(S_0)$ for $T\notin\mathcal M$
define a relatively open neighborhood of $\beta$ in $L$.
Rational points are dense in $L$, so choose one in that
neighborhood and clear denominators. This gives a positive
integer vector $w$ with precisely the same minimizing family.

Put $\kappa_w=\min_Sw(S)$. Finiteness gives a constant $c_0>0$
such that
\begin{equation}\label{eq:weighted-defect}
 \Delta_\beta\ge c_0t_w,\qquad
 t_w=\sum_k(w(S_k)-\kappa_w)\in\N.
\end{equation}
Explicitly take the minimum of
$(\beta(T)-\kappa)/(w(T)-\kappa_w)$ over nonminimizers;
if there are none take $c_0=1$.

Replace role $v$ by a clique of $w_v$ clones, retaining all
its boundary memberships, and replace each graph edge by a
complete join between the two fibers. For a deleted clone
set $T$, let $S_T$ be the original roles whose whole fibers
are deleted. Every surviving fiber is connected. Quotient
connectivity, boundary touching, and attachment are exactly
those of $G-S_T$. Thus $T$ is a separator if and only if
$S_T$ is. Undoing a partial fiber deletion preserves
separation and decreases size. Every minimum clone separator
therefore consists of full fibers, and corresponds precisely
to a member of $\mathcal M$. Its minimum size is $\kappa_w$
and its maximum active count is $a_*$.

Copy each original partition to every clone in its fiber.
Within a fiber its meet is the original even partition;
between fibers legality follows from the original edge.
This is an injection into legal clone states of defect
$t_w$. By \cref{thm:count}, their number is at most
$C_w^{2p}p^{a_*p+K_wt_w}$. The positive label weights and
\cref{eq:weighted-defect} now give
\begin{equation}\label{eq:aspect-trace}
 \E\tr((H_GH_G^*)^p)
 \le C^{2p}n^{p(B_G-\kappa)+\kappa}p^{a_*p}
       \sum_{t\ge0}(p^{K_w}/n^{c_0})^t.
\end{equation}
For any fixed logarithmic window the ratio tends to zero.
Choose
$p=\max(2,\lceil\log(2n)\rceil,\lceil q/2\rceil)$.
Taking roots proves the core upper bound at scale
$n^{(B_G-\kappa)/2}(\log(2n))^{a_*/2}$, including bounded $q$.

\subsection{The aspect-sensitive lower bound}

Positive costs make each minimum-cost separator
inclusion-minimal. Its conditional coefficient flattening,
proved exactly as in \cref{lem:separator-flattening}, has norm
\[
 \left(\prod_{v\in G\setminus(D\cup S)}m_v\right)^{1/2}
      \max_{x_S}\prod_{K\ {\rm active}}|T_K(x_S)|.
\]
The minimality path observes every separator role on both
sides. A factor cannot bridge two different components
outside $S$. These facts prove the flattening for the actual
unequal dimensions.

We spell out the change in the conditional estimate, since
a comparable-size theorem does not imply it. Fix a component
$K$ and a crossing occurrence with nonempty inside support
$B$. Put $B_K=\sum_{v\in K}\beta_v$. With the notation of
the factor lower proof,
\[
 \E T_K^2\asymp n^{B_K},\qquad
 \E z_y^2\asymp n^{B_K-\beta(B)},\qquad
 \E V^2\le Cn^{2B_K}.
\]
Fixed Hilbert moments give the same conditional
Paley--Zygmund estimates for $\sigma^2$ on a positive-measure
set of full internal realizations. For any fixed $M>2$,
\[
 \Prob\{\max_y|z_y|>
              n^{B_K/2-\beta(B)/4}\}
 \le C_M n^{\beta(B)(1-M/2)}=o(1).
\]
Conditional Markov removes an $o(1)$ set of internal
realizations. On each remaining realization, the simultaneous
variance and coefficient event has probability bounded below,
uniformly over separator tuples.

For $t=\eps n^{B_K/2}\sqrt{\log(2n)}$, the sign-tilt
hypothesis follows from
\[
 \frac{t\max_y|z_y|}{\sigma^2}
 =O\!\left(n^{-\beta(B)/4}\sqrt{\log(2n)}\right)\longrightarrow0.
\]
The one-trial conditional success probability is at least
$c n^{-C\eps^2}$. If $S\ne\varnothing$, there are at least
$c n^{b_S}$ tuples distinct in every separator role,
where $b_S=\min_{v\in S}\beta_v>0$. Choose $\eps$ so that
the sum of component success exponents is less than $b_S$.
Full conditional independence yields
\[
 \E\max_{x_S}\prod_K|T_K(x_S)|
 \ge c n^{\beta(D)/2}(\log(2n))^{a(S)/2}.
\]
When there are no active components the product is one.
Combining this with the flattening and a maximizing
minimum-cost cut proves the core lower bound.

For a detached higher-arity component, orthogonality and
fixed fourth moments give
$\E|Z_K|\asymp n^{\beta(K)/2}$. For growing upper moments
use \cref{eq:scalar-count}: a unit ordinary defect loses
at least $\min_{v\in K}\beta_v>0$ in the exponent of $n$.
Thus
\[
 \E|Z_K|^{2p}\le C^{2p}n^{p\beta(K)}p^p
       \sum_{\delta\ge0}
       (p^L/n^{\min_{v\in K}\beta_v})^\delta .
\]
Its logarithmic-window root is at most
$Cn^{\beta(K)/2}\sqrt q$ at $p=\lceil q/2\rceil$,
with the exact second moment treating $p=1$.
Independent restoration gives one factor $\sqrt q$ per
detached component and none for the expectation.

\subsection{Bounded roles}

Return to nonnegative aspects and write
$Z=\{v:\beta_v=0\}$. Their nonempty label sets have uniformly
bounded sizes. Fix consistent zero-boundary labels and let
$H_b$ be the corresponding rectangular compression. There
are boundedly many such blocks, and
$\|H_b\|\le\|H\|\le\sum_b\|H_b\|$.

For each occurrence $e$ and tuple $u$ on its zero coordinates,
introduce an auxiliary sign multiplying that entire array
slice. A fixed zero-middle assignment $z$ acquires the
character
\[
 \chi_z(\eta)=
 \prod_{e:e\cap Z\ne\varnothing}
                  \eta_{e,(z,b)|_{e\cap Z}}.
\]
Two assignments differing on a factor-covered zero-middle
role have different characters: the occurrence label is part
of the slice identifier. Unused zero-middle roles instead
give a deterministic multiplicity $L_0$ with $1\le L_0\le C$.
If $F_z$ is the fixed-assignment contribution with these
unused sums omitted, Fourier projection gives exactly
\[
 \E_\eta[\chi_z(\eta)H_b(\eps^\eta)]=L_0F_z(\eps).
\]
Every fixed slice flip preserves the independent sign law.
Minkowski therefore bounds $L_0\|\|F_z\|\|_p$ by
$\|\|H_b\|\|_p$. Conversely, each block is a sum of only
boundedly many such terms. All consistent assignments have
the same reduced incidence law, though they need not be
independent. Summing block norms proves two-sided comparison
between $\|\|H\|\|_p$ and $\|\|F_z\|\|_p$, with constants
independent of $p,n$.

The resulting positive scopes are $e\cap P$. Empty scopes
give scalar signs of modulus one. Scopes contained in one
positive boundary give diagonal sign isometries. Neither
changes the norm.

\subsection{Unary absorption and completion}

Multiply all unary occurrences at a positive middle role
$v$ into one independent sign vector $\sigma_v(x_v)$.
If $v$ belongs to a higher-arity occurrence, choose one
such occurrence $e(v)$. Conditional on all unary arrays,
replace its coordinates by
\[
 \widetilde\eps^{(e)}_{x_e}
 =\eps^{(e)}_{x_e}
       \prod_{v:e(v)=e}\sigma_v(x_v).
\]
This deterministic sign flip preserves the entire
conditional product law, independently of the conditioned
values, while algebraically absorbing the unary factors.
It therefore proves an equality in distribution of the
whole matrix, not merely a parity upper bound.

A graph-isolated middle role without factors contributes
$m_v$. A graph-isolated middle role with unary factors
contributes $\sum_i\sigma_v(i)$, whose expectation is
comparable to $\sqrt{m_v}$ and whose $L^q$ norm is at most
$C\sqrt{m_vq}$. It is one detached random component regardless
of how many unary occurrences were multiplied.

The positive higher-arity argument now applies to the
remaining model. Restoring deterministic roles, detached
scalars, bounded-role comparisons, and local amplitudes
gives the exponent
$\Gamma+(B_++h_\beta-\kappa)/2$ and the claimed logarithmic
and moment factors. If $P=\varnothing$, a fixed-assignment
contribution is a scalar product of signs of norm one;
the same bounded comparison gives $f=\Gamma$ and
$g=g_{\rm tail}=0$. Markov at logarithmic order proves
the tail claim.

For rational exponents, determine zero roles, perform these
finite structural operations, enumerate core separators,
compare rational costs exactly, retain all ties, and
maximize the active count. This terminates and proves the
computability assertion in \cref{thm:aspects}. It does not
verify the analytic amplitude bounds of an arbitrary input
program. For real exponents, the mathematical formula
remains valid but exact evaluation requires a representation
that can decide the finite cost comparisons.

\section{Finite Gaussian lifts and Hermite polarization}\label{app:gaussian-detail}\label{app:gaussian}

\subsection{The uniform Gaussian lift: full proof}\label{app:gaussian-lift-proof}
We prove \cref{thm:gaussian-lift}.

For each of $Q$ Gaussian occurrences on a nonempty scope $e$,
adjoin one private middle role of size $n$ and replace it by
a sign factor on the enlarged scope. Denote this sign
template by $\widehat{\mathcal F}$, and its sign certificate
by $(\widehat f,\widehat g,\widehat d_0)$.

Let $S_m=m^{-1/2}\sum_{j=1}^m\eps_j$ and $g$ be standard
Gaussian. For $1\le k\le m$,
\[
 (2k-1)!!\frac{\fall mk}{m^k}
       \le \E S_m^{2k}\le(2k-1)!!=\E g^{2k}.
\]
The lower bound retains exact pair partitions in the
expansion. For the upper bound, pair within every even
fiber; counting all pairings overcounts each surviving
assignment. For $k\le m/2$,
$\fall mk/m^k\ge e^{-k(k-1)/m}$.

Let $T_X(p)=\E\tr((H_XH_X^*)^p)$. In the coefficient-one
real model every trace coefficient is nonnegative.
Odd primitive multiplicities vanish for both laws.
For each occurrence, the half-multiplicities sum to $p$.
Termwise comparison therefore gives
\begin{equation}\label{eq:gaussian-trace-comparison}
 T_S(p)\le T_G(p)\le e^{Qp(p-1)/m}T_S(p),
 \qquad p\le m/2.
\end{equation}
For $m=n$, the private-role sums give the exact identity
in law $H_S=n^{-Q/2}H_{\widehat{\mathcal F}}$.
No central limit approximation is involved.

For the reverse norm comparison, set
$\mu=\E|g|=\sqrt{2/\pi}$. Independence of Gaussian signs
and magnitudes and conditional Jensen show, in any real
Banach space and for $q\ge1$,
\[
 \left\|\sum_j\eps_jb_j\right\|_{L^q}
 \le\mu^{-1}\left\|\sum_jg_jb_j\right\|_{L^q}.
\]
Apply this to each entire independent factor group of
the lifted polynomial. Replacing all lifted signs by
Gaussians makes each normalized private sum standard
Gaussian, so
\[
 \|\|H_G\|\|_q\ge
 \mu^Q n^{-Q/2}\|\|H_{\widehat{\mathcal F}}\|\|_q.
\]

For the upper bound first split detached components.
On a core with $Q_c$ occurrences, choose $p$ proportional
to $\log(2n)$ with $2p\ge q$. Its row dimension $D$
is a fixed power of $n$ up to constants. By
\cref{eq:gaussian-trace-comparison},
\[
 \|\|H_{c,G}\|\|_q
 \le e^{Q_c(p-1)/(2n)}D^{1/(2p)}
                    \|\|H_{c,S}\|\|_{2p}
 \le C n^{\widehat f_c-Q_c/2}
                         (\log(2n))^{\widehat g}.
\]
The exponential and dimension factors are bounded.
For a detached scalar use $p=\lceil q/2\rceil$ in
the scalar version of the comparison and
\cref{eq:scalar-moments}; at expectation use $p=1$.
Its scale is $n^{(\widehat r_j-Q_j)/2}\sqrt q$,
with no $\sqrt{\log n}$ in expectation.
Independent restoration proves
\begin{equation}\label{eq:gaussian-uniform}
 \E\|H_G\|\asymp n^{\widehat f-Q/2}
                         (\log(2n))^{\widehat g},\qquad
 \|\|H_G\|\|_q\asymp n^{\widehat f-Q/2}
            (\log(2n))^{\widehat g}q^{\widehat d_0/2}
\end{equation}
for every fixed $C_0>0$, uniformly over
$2\le q\le C_0\log(2n)$ and all sufficiently large $n$.
The constants and cutoff may depend on the fixed template,
role-size comparison constants, and $C_0$.

For circular complex arrays write $Z=(A+iB)/\sqrt2$.
Expanding $Q$ factors and applying Minkowski gives the
real-Gaussian upper comparison with factor $2^{Q/2}$.
Conditioning on all $A$ arrays gives
$\E(H_Z\mid A)=2^{-Q/2}H_A$, and Jensen gives the
reverse comparison. No positivity of complex trace
coefficients is asserted in this step.

If all original scopes have arity at least two, the
lift has the same minimum separator size as the
original two-section. A path through a private role
can be shortened through its scope clique, and every
lifted separator must already separate the old paths.
A minimum separator therefore contains no private
role. After its deletion, a private role is an active
singleton precisely when its whole scope lies in the
separator; otherwise it joins the component of the
surviving scope. Thus
\[
 \widehat f-Q/2=(|W|+h-s)/2,\qquad
 \widehat g=\frac12\max_{|S|=s}
                   \bigl(a(S)+|\{e:e\subseteq S\}|\bigr).
\]
Unary factors are handled by the lift itself, which
distinguishes random singleton roles from unused roles.

\subsection{Hermite polarization: full proof}\label{app:hermite-proof}
We prove \cref{thm:hermite}.
We prove the dimension-free comparison. For Banach
coefficients $b_a$, let
\[
 Y=\sum_ab_a\frac{\operatorname{He}_d(g_a)}{\sqrt{d!}},
 \qquad
 D=\sum_ab_a\prod_{j=1}^dg_a^{(j)},\qquad
 A_d=\frac{d^{d/2}}{\sqrt{d!}}.
\]
Then $A_d^{-1}\|D\|_q\le\|Y\|_q\le A_d\|D\|_q$
for all $q\ge1$.
Put $h_a=d^{-1/2}\sum_jg_a^{(j)}$. Conditional Gaussian
means and covariances, or their generating function,
give
$\E(\prod_jg_a^{(j)}\mid h_a)=d^{-d/2}
 \operatorname{He}_d(h_a)$.
Jensen proves the upper bound for $Y$.

For the other direction let
$h_a(\eps)=d^{-1/2}\sum_j\eps_jg_a^{(j)}$.
The full sign Fourier coefficient is
\[
 \E_\eps\left[\left(\prod_j\eps_j\right)
               \operatorname{He}_d(h_a(\eps))\right]
 =d!d^{-d/2}\prod_jg_a^{(j)}.
\]
All lower-degree terms vanish, and the leading term
survives only when every sign is used once.
Minkowski and the Gaussian law of each $h_a(\eps)$
prove the reverse bound. Apply the comparison
conditionally through the fixed independent
occurrence groups, then use
\cref{eq:gaussian-uniform}.

\section{Complete degree-four construction and positivity verification}\label{app:sos-detail}\label{app:sos}

We prove \cref{thm:sos}, including normalization, Boolean and
nonedge constraints, the exact size constraint, and positivity on all
polynomials of degree at most two. The graph event is chosen first;
the size parameter is then tuned pointwise on that same event.

\subsection{Exact graph-matrix multiplication}\label{app:sos-algebra}

For shapes $\alpha,\beta$ with compatible inner boundary lengths,
matrix multiplication forces equality between their inner boundary
labels. Every additional cross-equality is a partial matching of
roles, since each original embedding is injective. Let $\mu$ range
over partial matchings extending these forced identifications.
Identify matched roles and keep all other quotient roles distinct.
If an edge occurs from both factors, cancel the two copies using
$\eps_{ij}^2=1$. Denote the resulting shape by
$\alpha\star_\mu\beta$. Then, pointwise,
\begin{equation}\label{eq:graph-product}
 M_\alpha M_\beta=\sum_\mu M_{\alpha\star_\mu\beta},
 \qquad M_\alpha^*=M_{\alpha^*},
\end{equation}
where $\alpha^*$ exchanges boundaries.

To prove the product identity, expand pairs of embeddings and group
them by their exact cross-collision relation. For a fixed relation,
pairs of embeddings are in bijection with injective quotient
embeddings. Their sign product is exactly the quotient monomial
after cancellation, and the external coordinates agree. The
collision classes are disjoint and exhaustive. Thus every fixed
finite graph-matrix expression has an exact finite expansion.
This algebraic closure is specific to the sign identity used
above; Gaussian repeated coordinates cannot be canceled in this way.

We now implement this program using the corrected quartic ansatz of
\citet{HopkinsKothariPotechin2015}. Let $W$ be a symmetric zero-diagonal sign matrix and
$A_{ij}=(1+W_{ij})/2$ for $i\ne j$.

\subsection{The common graph event}

All expansions below involve a fixed finite list of shapes on
at most five roles and at most ten edges. The growing-moment
theorem gives a simultaneous bound $Kn^fL$ with
$L=(\log(2n))^{10}$, except on an event of probability at most
$n^{-12}$ for sufficiently large $n$. This deliberately loose
logarithm suffices for every polynomially subcritical term.

Two families need log-free bounds. For completeness, the
following reference argument supplies them independently of
an enumeration of logarithmic exponents. Deleting edges from
a typed sign shape only removes parity constraints in its
positive trace expansion. If the reference is a disjoint
union of boundary edges and two-edge boundary paths, its
matrix is a tensor product of independent sign matrices
and products of two such matrices. The net proof in
\cref{lem:path-count}, or its tail integral at arbitrary
logarithmic order, bounds these factors at scales
$n^{1/2}$ and $n$, respectively. Independence or H\"older
with a fixed enlarged moment window handles their tensor
product. Trace domination, taking a moment at least
$\log n$ to absorb the coordinate dimension, proves the
same scale for a supergraph with the same polynomial
exponent. \Cref{lem:color-upper} transfers the estimate to
global injectivity. We use this only for the explicit
references described below.

Unordered pair matrices are compressions of ordered
boundary-coordinate matrices, with fixed multiplicative
constants for a shared-root orbit. A symmetric sum over
two unordered middle roles is replaced by half the
\emph{whole} ordered sum before its Fourier expansion.
No nonsymmetric individual pattern is restricted by an
ordering condition and then treated as a full graph matrix.

Let $C_4$ be the number of four-cliques and
$Z_4=\sum_{T\text{ four-clique}}\sum_s\prod_{a\in T}W_{sa}$.
The graph event also includes
\begin{equation}\label{eq:sos-counts}
 |C_4-\tbinom n4/64|\le Kn^3L,\qquad
 |Z_4|\le Kn^{5/2}L,\qquad \|W\|\le K\sqrt n.
\end{equation}
For the first bound expand the six clique edges: the
$63$ nonempty scalar patterns have polynomial exponent
at most $3$. For the second, the four center-star edges
are mandatory and the six clique edges are optional:
the $64$ patterns are connected with exponent $5/2$.
The conversion from unordered quartets has factor $1/24$.
We enlarge one fixed $K$ to cover all finite coefficients
and all graph bounds. It is independent of the parameters
chosen below.

\subsection{Quartic moments and the filled matrix}

For an internal parameter $\omega\ge8$ and fixed $\gamma>0$,
put $\tau=\gamma(\omega/n)^5$ and
\[
 \mu_4(T)=\one_{\{T\text{ clique}\}}
       \left[\frac{\binom\omega4}{C_4}
                    +\tau\sum_s\prod_{a\in T}W_{sa}\right].
\]
Let $\omega'$ solve
$\binom{\omega'}4=\binom\omega4+\tau Z_4$ near $\omega$.
Extend downwards by
$\mu_j(S)=(\omega'-j)^{-1}
 \sum_{x\notin S}\mu_{j+1}(S\cup\{x\})$ for $j=3,2,1$,
and set $\mu_0=1$. On \cref{eq:sos-counts},
\begin{equation}\label{eq:sos-shift}
 2^{-14}n^4\le C_4\le n^4,\qquad
 d:=|\omega'-\omega|
 \le K\gamma\omega^2n^{-5/2}L\le1
\end{equation}
uniformly for $8\le\omega\le c\sqrt n$, eventually.
This follows from the derivative of $\binom t4$, which
is bounded below by a constant times $\omega^3$ on
$[\omega-1,\omega+1]$.

Index matrices by unordered pairs. Let $D$ be the
signless vertex-pair incidence matrix, put $J_D=D^*D$,
and let $P$ project onto $\operatorname{range}(D^*)$,
with $P_0=I-P$. Then
$DD^*=(n-2)I+\one\one^*$ and $\|J_D\|=2n-2$.
Define $Q_{\{a,b\},s}=W_{sa}W_{sb}$ and
$S_W=\diag(W_{ab})$.

The filled uncorrected pair matrix has entries
\[
 M'(I,J)=\beta_i
 \sum_{\substack{T\supseteq I\cup J\\|T|=4}}
 \prod_{e\in E(T)\setminus(E(I)\cup E(J))} A_e,
 \qquad i=|I\cap J|,
\]
where
$\beta_0=\binom\omega4$,
$\beta_1=\binom\omega3/4$, and
$\beta_2=\binom\omega2/6$.
Its expectation is
\begin{equation}\label{eq:sos-mean}
 E=\alpha_0\one\one^*+(\alpha_1-\alpha_0)J_D
                +(\alpha_2-2\alpha_1+\alpha_0)I,
\end{equation}
with
\[
 \alpha_0=\frac{\binom\omega4}{16},\quad
 \alpha_1=\frac{\binom\omega3(n-3)}{64},\quad
 \alpha_2=\frac{\binom\omega2\binom{n-2}2}{192}.
\]
Let $b=\alpha_2-2\alpha_1+\alpha_0$. The three eigenvalues
on constants, their orthogonal complement in
$\operatorname{range}(D^*)$, and $\ker D$ are
$b+(2n-2)(\alpha_1-\alpha_0)+\binom n2\alpha_0$,
$b+(n-2)(\alpha_1-\alpha_0)$, and $b$.
For $n\ge128$ and $8\le\omega\le n/128$,
elementary binomial bounds give respectively the lower
bounds
$\omega^4n^2/24576$, $\omega^3n^2/16384$,
and $\omega^2n^2/8192$. In particular
\begin{equation}\label{eq:sos-metric}
 E\succeq aB,\qquad
 a=2^{-15},\qquad B=\omega^3n^2P+\omega^2n^2P_0.
\end{equation}
For example,
$\alpha_2\ge\omega^2n^2/6144$,
$2\alpha_1\le\omega^2n^2/24576$,
$\alpha_1-\alpha_0\ge\omega^3n/8192$, and
$\binom n2\alpha_0\ge\omega^4n^2/24576$ suffice for these
bounds. The two-scale metric is essential.

\subsection{Local fluctuations and the exact square}

On disjoint pairs $I=\{a,b\},J=\{c,d\}$, let
$\mathcal X=\{ac,ad,bc,bd\}$. The local fluctuation is
$\alpha_0\sum_{\varnothing\ne F\subseteq\mathcal X}
 \prod_{e\in F}W_e$.
Four of these patterns are single edges, four are
two-edge stars, and seven contain a cross perfect
matching. Completing the full single-edge and star
orbits gives the exact identity
\[
 L_0=\alpha_0(D^*WD+D^*Q^*+QD+Z),\qquad \|Z\|\le Kn.
\]
The seven remaining patterns have a two-edge matching
reference of the same exponent $1$, so are log-free.
Indeed, among the nonempty subsets of the four edges of
$\mathcal X$, the four singleton subsets and four two-edge stars
are the exceptional orbits already completed. The other subsets
are the two perfect matchings, the four three-edge subsets, and
the full four-edge set. Each contains one of the two perfect
matchings. This accounts for all $15$ local patterns without
requiring a computer enumeration.
The unwanted overlap entries in the three completed
matrices have bounded magnitudes and $O(n)$ entries
per row and column; their norm is $O(n)$.

On distinct overlapping pairs $\{a,b\},\{a,c\}$ the
local fluctuation is $\alpha_1W_{bc}$.
If $\operatorname{Lift}(T)$ denotes compression of
$T\otimes I+I\otimes T$ to symmetric off-diagonal pair
coordinates, then $L_1=\alpha_1\operatorname{Lift}(W)$
and $\|L_1\|\le K\omega^3n^{3/2}$.

Set $\rho=\tau C_4/16$ and define
\[
 R(I,J)=16\rho\,\one_{\{I\cap J=\varnothing\}}
    \prod_{e\in\mathcal X}A_e
          \sum_s\prod_{u\in I\cup J}W_{su},\qquad
 \widetilde R(I,J)=\rho\sum_s
                 \prod_{u\in I\triangle J}W_{su}.
\]
The exact all-collision identity is
\begin{equation}\label{eq:sos-gram}
 \widetilde R=\rho(QQ^*+S_WJ_DS_W).
\end{equation}
For disjoint pairs it is immediate. On
$\{a,b\},\{a,c\}$, $QQ^*$ misses exactly the term
$W_{ab}W_{ac}$ supplied by $S_WJ_DS_W$.
On the diagonal the two terms are $n-2$ and $2$.

On disjoint pairs, $R-\widetilde R$ has $15$ patterns:
the center is joined to all four boundary roles and
a nonempty subset of $\mathcal X$ is present.
Choose any edge $ac$ in that subset. The disjoint
edge $a-c$ and path $b-s-d$ form a reference with
separator size two and exponent $3/2$, equal to
that of the parent. Thus these patterns are log-free.
On overlapping pairs the retained matrix
$\widetilde R/\rho$ is exactly
$\operatorname{Lift}(W^2)+(2-n)I$, of norm $O(n)$.
Using \cref{eq:sos-shift},
\begin{equation}\label{eq:sos-r-error}
 \|R-\widetilde R\|\le K\rho(n^{3/2}+n)
                 \le K\gamma\omega^5\sqrt n,\quad
 \frac{\gamma\omega^5}{2^{18}n}\le\rho\le
                  \frac{\gamma\omega^5}n .
\end{equation}

Complete the dangerous square exactly:
\begin{equation}\label{eq:sos-square}
 \begin{split}
 \rho QQ^*+\alpha_0(D^*Q^*+QD)
 &=
 \rho\left(Q+\frac{\alpha_0}{\rho}D^*\right)
       \left(Q+\frac{\alpha_0}{\rho}D^*\right)^*
       -\frac{\alpha_0^2}{\rho}J_D .
 \end{split}
\end{equation}
Both the displayed square and $\rho S_WJ_DS_W$ are
positive semidefinite. The negative term is supported
on $P$. Therefore
\[
 \left\|B^{-1/2}
 \left(\alpha_0D^*WD-\frac{\alpha_0^2}{\rho}J_D\right)
 B^{-1/2}\right\|
 \le K\left(\frac{\omega}{\sqrt n}+\frac1\gamma\right).
\]
The remaining normalized costs of $\alpha_0Z$, $L_1$,
and $R-\widetilde R$ are at most
$K\omega^2/n$, $K\omega/\sqrt n$, and
$K\gamma\omega^3/n^{3/2}$, respectively.
No independence of the random scalar $\rho$ and $W$
has been used.

\subsection{Extension-count fluctuations}

The shared-root extension fluctuation is exactly
\[
 \Delta_1(\{a,b\},\{a,c\})=
 \frac{\beta_1}{16}(1+W_{bc})
 \sum_{s\notin\{a,b,c\}}
       [(1+W_{as})(1+W_{bs})(1+W_{cs})-1].
\]
Its $14$ patterns give
$\|\Delta_1\|\le K\omega^3n^{3/2}L$.
For the kernel block there is the stronger bound
$\|P_0\Delta_1P_0\|\le K\omega^3nL$.
Indeed, nine patterns have both exclusive boundary
roles incident and a $b$--$c$ path avoiding $a$;
their separator size is at least two. The other
five have an isolated exclusive boundary role.
If $c$ is isolated, write the completed pattern
as $UD$, where
\[
 U_{I,a}=\one_{\{a\in I\}}
       \sum_{s\notin I}\prod_{e\in F(s,a,I\setminus\{a\})}W_e .
\]
On distinct overlapping pairs the difference is
the single excluded term $s=c$; the diagonal has
magnitude at most $2(n-2)$. The collision matrix
has row and column absolute sums at most $4n$.
Since $DP_0=0$, only this collision matrix survives
the compression. Transpose when $b$ is isolated.
This covers all five exceptions.

The nine--five split can be checked directly. Write the chosen
nonempty star subset as $F\subseteq\{as,bs,cs\}$ and let the
optional edge be $bc$. If $bc$ is present, all seven choices of
$F$ have the required $b$--$c$ path. If $bc$ is absent, exactly
the two choices containing both $bs$ and $cs$ have such a path.
The five remaining choices are
$\{as\},\{bs\},\{cs\},\{as,bs\},\{as,cs\}$, each with an
isolated exclusive boundary role. Thus every one of the $14$
patterns belongs to the stated estimate or collision completion.

The diagonal extension fluctuation is
\[
 \Delta_2(I,I)=\beta_2\left[
 \sum_{\substack{s<t\\s,t\notin I}}
 A_{as}A_{at}A_{bs}A_{bt}A_{st}
                         -\frac{\binom{n-2}2}{32}\right].
\]
Replace the symmetric middle sum by half the
ordered sum. Its $31$ nonempty patterns have
equal two-role boundaries and at most one
isolated middle role, so
$\|\Delta_2\|\le K\omega^2n^{3/2}L$.
Bounding the four $P,P_0$ blocks separately yields
\begin{equation}\label{eq:sos-delta}
 \|B^{-1/2}(\Delta_1+\Delta_2)B^{-1/2}\|
 \le KL\left[\frac{1+\sqrt\omega}{\sqrt n}
                       +\frac{\omega}{n}\right].
\end{equation}

\subsection{The exact lower-degree cleanup}

For $j=2,3$ and a $j$-set $S$, let $c_4(S)$ be the number of
four-cliques containing $S$, and define the uncorrected moments
\[
 \mu_j^0(S)=\frac{\binom\omega j\,c_4(S)}{C_4\binom4j}.
\]
These vanish on noncliques. Equivalently, they are the downward
recursion of the uncorrected fourth moments with the internal
parameter $\omega$, not the corrected parameter $\omega'$.

Let $G_{\rm pair}$ project onto graph-edge pair
coordinates and let
$N'_{I,J}=\mu_{|I\cup J|}(I\cup J)$.
For a clique $S$ of size two or three, put
$h_S=\sum_{T\supseteq S,\ T\text{ four-clique}}
                         \sum_s\prod_{v\in T}W_{sv}$.
The downward recursion gives exactly
\[
 \mu_3=r_3\mu_3^0+\frac{\tau h_S}{\omega'-3},
 \qquad
 \mu_2=r_2\mu_2^0+
       \frac{2\tau h_S}{(\omega'-2)(\omega'-3)},
\]
where
$r_3=(\omega-3)/(\omega'-3)$ and
$r_2=(\omega-2)(\omega-3)/[(\omega'-2)(\omega'-3)]$.
The factor two counts the two orders of the added
vertices. Also $|r_i-1|\le Kd/\omega$.

Define filled matrices supported on shared-root
off-diagonal and diagonal pairs, respectively:
\[
 \begin{split}
 K_3(\{a,b\},\{a,c\})
 &=A_{bc}\sum_{d\notin\{a,b,c\}}A_{ad}A_{bd}A_{cd}
                \sum_sW_{sa}W_{sb}W_{sc}W_{sd},\\
 K_2(\{a,b\},\{a,b\})
 &=\sum_{\substack{c<d\\c,d\notin\{a,b\}}}
 A_{ac}A_{ad}A_{bc}A_{bd}A_{cd}
                \sum_sW_{sa}W_{sb}W_{sc}W_{sd}.
 \end{split}
\]
The $16$ and $32$ Fourier patterns have five roles,
no isolated middle, and separator size two.
Thus $\|K_3\|+\|K_2\|\le Kn^{3/2}L$.
Let $M_1^{\rm cl},M_2^{\rm cl}$ denote $C_4\mu_3^0$
on overlapping clique pairs and $C_4\mu_2^0$ on
clique-pair diagonals. Nonnegative count bounds give
norms at most $K\omega^3n^2$ and $K\omega^2n^2$.
Set
\[
 \begin{split}
 \operatorname{Err}
 &=(r_3-1)M_1^{\rm cl}+(r_2-1)M_2^{\rm cl}\\
 &\quad+\frac{C_4\tau}{\omega'-3}G_{\rm pair}K_3G_{\rm pair}
 +\frac{2C_4\tau}{(\omega'-2)(\omega'-3)}
                         G_{\rm pair}K_2G_{\rm pair}.
 \end{split}
\]
For $N=M'+R$, direct comparison in each overlap
case gives the exact identity
\begin{equation}\label{eq:sos-compression}
 G_{\rm pair}(N+\operatorname{Err})G_{\rm pair}=C_4N'.
\end{equation}
Moreover
\[
 \frac{\|\operatorname{Err}\|}{\omega^2n^2}
 \le K\gamma L\left[
 \frac{\omega^2+\omega}{n^{3/2}}
                  +\frac{\omega^2}{n^{5/2}}\right].
\]
This is an orthogonal coordinate compression,
not a general entrywise masking inequality.

\subsection{Budget, simultaneous sizes, and full positivity}

Every error above is included in
\[
 \begin{split}
 \varepsilon_n=K\bigg\{&
 \frac1\gamma+\frac{\omega}{\sqrt n}
 +\frac{\omega^2}{n}
 +\frac{\gamma\omega^3}{n^{3/2}}\\
 &+L\left[\frac{1+\sqrt\omega}{\sqrt n}
 +\frac{\omega}{n}
 +\frac{\gamma(\omega^2+\omega)}{n^{3/2}}
 +\frac{\gamma\omega^2}{n^{5/2}}\right]\bigg\}.
 \end{split}
\]
After discarding only positive squares,
$N+\operatorname{Err}\succeq(a-\varepsilon_n)B$.
Choose $\gamma=32K/a$ and
$c=a/(128K\gamma)$. The four nonvanishing terms are
a fixed fraction of $a$ for $\omega\le c\sqrt n$;
all $L$ terms tend to zero uniformly on this interval.
For sufficiently large $n$,
$\varepsilon_n<a/4$ throughout
$8\le\omega\le c\sqrt n$. By
\cref{eq:sos-compression}, $N'\succeq0$ with
zero nonedge rows.

The event was defined without any parameter choice.
For every real $9\le k\le(c/2)\sqrt n$, solve
\begin{equation}\label{eq:sos-tuning}
 \binom t4+\gamma(t/n)^5Z_4=\binom k4
\end{equation}
on $[k-1,k+1]$. The derivative of the first term
is bounded below by a constant times $k^3$.
Both the perturbation divided by $k^3$ and its
derivative divided by $k^3$ tend to zero uniformly
over this interval of $k$, by \cref{eq:sos-counts}.
Thus the endpoints bracket the target and the
derivative is positive. The unique root satisfies
$|t-k|\le C\gamma Kk^2n^{-5/2}L=o(1)$.
Use $t$ as the internal parameter and $k$ in the
downward recursion. All previous identities are
pointwise and all budgets hold at this $t$.
No union over real $k$ or independence from the
chosen parameter is required.

Define the linear functional $\mathcal L_k$ by
these squarefree moments and Boolean reduction.
Downward counting yields
$\sum_{|S|=j}\mu_j(S)=\binom kj$ for $1\le j\le4$.
The recursion gives
$\mathcal L_k((\sum_i x_i-k)x_S)=0$ for $|S|\le3$.
Boolean constraints hold by reduction. A nonclique
set has zero moment since every superseding quartet
is a nonclique, proving the nonedge constraints.

To include the constant and linear rows in
positivity, let $T$ have pair rows and columns
indexed by subsets of size at most two:
\[
 T_{\{i,j\},\varnothing}=\frac{2}{k(k-1)},\quad
 T_{\{i,j\},\{a\}}=\frac{\one_{\{a\in\{i,j\}\}}}{k-1},
 \quad
 T_{\{i,j\},\{a,b\}}=\one_{\{\{i,j\}=\{a,b\}\}}.
\]
The full moment matrix is exactly $T^*N'T$.
Indeed these columns replace constants and linear
monomials by homogeneous quadratics. The required
identities are
\[
 \begin{split}
 \sum_{j\ne i}x_ix_j-(k-1)x_i
 &=x_i(\sum_jx_j-k)-(x_i^2-x_i),\\
 2\sum_{i<j}x_ix_j-k(k-1)
 &=(\sum_ix_i-k)(\sum_ix_i+k-1)
                         -\sum_i(x_i^2-x_i).
 \end{split}
\]
For a degree-two polynomial $p$ and its quadratic
replacement $h$, the product $(h-p)(h+p)$ uses
these constraints only through degree four.
Thus $\mathcal L_k(p^2)=\mathcal L_k(h^2)\ge0$.
This proves all the asserted constraints and
full degree-four positivity simultaneously.

The finite template families used above have
sizes $7,15,14,31,16,32,64,63$, totaling $242$.
Their explicit descriptions, collision completions,
and norm references were given in the proof;
a finite computational check is not used as
a substitute for any of these arguments.

\section{RTN deviations, probability lower bounds, and entropy}\label{app:rtn-detail}\label{app:rtn}

We prove \cref{thm:rtn,thm:rtn-state}. Finite Schatten thresholds,
the expected operator endpoint, and the sample-state conclusion require
different moment inputs; the proof does not infer growing-order control
from a fixed-order expansion.

\subsection{A fixed-order trace lemma}\label{app:rtn-fixed-proof}

For a coefficient-one incidence matrix $J$ with $R$ roles,
$K$ independent real or circular complex Gaussian occurrences
on nonempty scopes, and role sizes $N$, let $h$ be the number
of unused middle roles and $s$ the two-section separator size.
For every fixed integer $p\ge1$,
\begin{equation}\label{eq:fixed-gaussian-trace}
 \E\tr((JJ^*)^p)
 \le \operatorname{Bell}(2p)^{R-h}((2p)!)^K
                       N^{p(R+h-s)+s}.
\end{equation}
Remove unused middle roles, giving the matrix multiplier $N^h$.
In an exact equality state a real Gaussian factor requires
even multiplicities; a circular complex factor requires
balanced conjugate and unconjugate multiplicities, which
also implies evenness. Each nonzero occurrence moment is
nonnegative and at most $(2p)!$. Factor legality implies
the interval compatibility of \cref{lem:interval}. The
separator layers yield
$\sum_v|\pi_v|\le p(R-h-s)+s$; for $p=1$ the layer sum is
empty and the same bound holds. There are at most
$\operatorname{Bell}(2p)^{R-h}$ partition assignments.
Bound falling factorials by powers of $N$ and restore
$N^{2ph}$ to prove \cref{eq:fixed-gaussian-trace}.
Detached components are allowed. The constant need not
be controlled as $p$ grows.

\subsection{Centering, decoupling, and doubled cuts}

At an occurrence write
$\bar g_i g_j=\delta_{ij}+W_{ij}$. Expand the product
over occurrences and remove the all-covariance term.
The remaining terms $K_T$ are indexed by nonempty active
tensor sets $T$.

For arbitrary Banach coefficients $B_{ij}$ and $q\ge1$,
introduce an independent copy $g'$ and rotate to
$u=(g+g')/\sqrt2$, $v=(g-g')/\sqrt2$.
Conditional Jensen, followed by the triangle inequality,
gives
\[
 \left\|\sum B_{ij}(\bar g_i g_j-\delta_{ij})\right\|_{L^q}
 \le\left\|\sum B_{ij}(\bar u_iv_j+\bar v_iu_j)\right\|_{L^q}
 \le2\left\|\sum B_{ij}\bar u_iv_j\right\|_{L^q}.
\]
The last two summands have the same law by swapping
$u,v$. Iterating across active groups and using the
Schatten-$2p$ norm yields
\begin{equation}\label{eq:rtn-decouple}
 (\E\|F_N\|_{S_{2p}}^{2p})^{1/(2p)}
 \le\sum_{T\ne\varnothing}2^{|T|}N^{-(r-b)}
                    (\E\|D_T\|_{S_{2p}}^{2p})^{1/(2p)}.
\end{equation}
Here $D_T$ has two independent copies at active tensors
and covariance identifications at inactive tensors.
This is only an upper decoupling inequality.

Write $e_T,e_C$ for internal edges within $T$ and its
complement, and $a_T,a_C,b_T,b_C$ for their terminal
counts. The roles of $D_T$ are as follows:
\begin{center}
\begin{tabular}{ll}
\toprule
Original edge type & Roles after covariance contraction\\
\midrule
Internal within $T$ & $2e_T$ copied internal roles\\
Internal outside $T$ & $e_C$ unused middle roles\\
Internal crossing $T$ & $c(T)$ bridge roles\\
Output on $T$ & $a_T$ bridge roles\\
Output outside $T$ & $a_C$ unused middle roles\\
Input on $T$ & $2b_T$ separate boundary roles\\
Input outside $T$ & $b_C$ common boundary roles\\
\bottomrule
\end{tabular}
\end{center}
In particular $R_T=r+e_T+b_T$ and $h_T=e_C+a_C$.

Assume $\Delta\ge0$. In one copy of the induced
$T$-network, regard crossing and output bridges as
output terminals. The excess of a cut indexed by
$\varnothing\ne S\subseteq T$ over its $b_T$ input
cut is $\delta(S)\ge0$. Integral max-flow/min-cut
gives $b_T$ edge-disjoint paths from its inputs
to distinct bridges. Reflect and concatenate them
in the other copy. They give $b_T$ disjoint
boundary paths in the edge-role two-section.
The input edges give a matching upper cut.
Adding the mandatory $b_C$ identity roles shows
that the separator size of $D_T$ is exactly $b$.
This also covers $b_T=0$.

Apply \cref{eq:fixed-gaussian-trace}. The exponent
after normalization is determined by
\[
 R_T+h_T-b-2(r-b)=-c(T)-a_T+b_T=-\delta(T).
\]
The finite sum in \cref{eq:rtn-decouple} therefore proves
\begin{equation}\label{eq:rtn-fixed}
 \E[d^{-1}\tr|F_N|^{2p}]\le C_pN^{-p\Delta},
 \qquad \Delta\ge0,
\end{equation}
for each fixed $p$.

\subsection{Second moments and nonrare spectral mass}

At two Gram replicas a circular complex tensor has
two Wick pairings, identity or swap, including when
coordinates coincide. Let $T$ be the swapped set.
Each crossing internal edge, output leg on $T$,
or input leg outside $T$ loses one free label.
After normalization its contribution is
$N^{-\delta(T)}$. The empty set contributes one
and is exactly removed by centering. Hence
\begin{equation}\label{eq:rtn-second}
 \E[d^{-1}\tr F_N^2]
       =\sum_{T\ne\varnothing}N^{-\delta(T)}.
\end{equation}

Put $\epsilon=N^{-\Delta/2}$ and $L=2^Q-1$.
On the Gaussian sample space enlarged by an independent
uniform eigenvalue index $j$, let
$Y=|\lambda_j(F_N)|$.
\cref{eq:rtn-fixed,eq:rtn-second} give
$\epsilon^2\le\E Y^2\le L\epsilon^2$ and
$\E Y^4\le C_2\epsilon^4$.
Interpolation gives $\E Y\ge C_2^{-1/2}\epsilon$.
For $Z=d^{-1}\|F_N\|_1$, $\E Z=\E Y$ and
$Z^2\le d^{-1}\tr F_N^2$. Paley--Zygmund yields
\begin{equation}\label{eq:rtn-mass}
 \E Z\ge C_2^{-1/2}\epsilon,\qquad
 \Prob\{Z\ge\epsilon/(2\sqrt{C_2})\}
                   \ge(4C_2L)^{-1}.
\end{equation}
Thus the lower scale is not inferred merely from
a nonzero second moment.

For fixed finite $t\ge1$, choose a fixed integer
$p$ with $2p\ge t$. Power means and
\cref{eq:rtn-fixed} give
\[
 \E\|F_N\|_{S_t}
 \le d^{1/t-1/(2p)}
                  (\E\tr|F_N|^{2p})^{1/(2p)}
 \le C_t d^{1/t}\epsilon .
\]
Conversely $\|F_N\|_{S_t}\ge d^{1/t-1}\|F_N\|_1$.
\Cref{eq:rtn-mass} gives matching expectation
and positive-probability lower bounds. This proves
the finite Schatten assertion in \cref{thm:rtn}.

If $\Delta>2b/t$, Markov proves convergence to zero.
For $0\le\Delta\le2b/t$, the probability lower
bound excludes convergence, including equality.
If $\Delta<0$, a cut of size $b+\Delta$ factors
the map through at most $N^{b+\Delta}$ coordinates.
At least $d-N^{b+\Delta}$ eigenvalues of $F_N$
are then $-1$. This excludes convergence in any
of the stated norms.

For the operator norm, if $\Delta>0$, choose
one fixed integer $p>b/\Delta$. Then for each
fixed $\eta>0$,
\[
 \Prob\{\|F_N\|_{\rm op}>\eta\}
 \le C_p\eta^{-2p}N^{b-p\Delta}\longrightarrow0.
\]
For $\Delta=0$ use \cref{eq:rtn-mass}; for
$\Delta<0$ use rank. This proves the qualitative
criterion using fixed moments alone. Fixed
moments give an expected upper rate
$N^{-\Delta/2+\zeta}$ for each fixed $\zeta>0$,
not the exact endpoint $\zeta=0$.

\subsection{The uniform operator endpoint}

We now use the additional uniform Gaussian
reduction in \cref{app:gaussian}. For a standard
network the role two-section is the line graph,
and role separators correspond to terminal-separating
edge cuts. Adding private roles does not change
the minimum separator size: every detour through
a private role can be shortened in its incident
clique.

For an inclusion-minimal terminal-separating cut,
adding back any one cut edge creates an input-output
path. Its endpoints therefore lie in components
meeting opposite terminal sets. A terminal-free
component of the network minus the cut cannot
touch a cut edge; it would consequently have
been terminal-free already, contrary to the model.
Every tensor retains a path to a terminal and
an uncut incident edge. Its private role attaches
to a boundary-touching component. Thus every
minimum lifted cut has active count zero.

For $D_T$, remove its unused middle roles and
common identity roles. The other components
are standard doubled networks. Boundary-touching
ones have zero active count by the preceding
argument. A component of the induced $T$-network
with no input produces a detached scalar after
doubling through bridges. When $\Delta>0$ such
a component has at least one bridge. Let $k_T$
count these components. The uniform Gaussian
bound and the same decoupling give
\begin{equation}\label{eq:rtn-uniform}
 \|\|F_N\|_{\rm op}\|_{L^q}
 \le C\sum_{T\ne\varnothing}
                  N^{-\delta(T)/2}q^{k_T/2}
\end{equation}
in every fixed logarithmic window. For expectation,
the detached $q$ powers are absent.
Together with \cref{eq:rtn-mass}, this proves
$\E\|F_N\|_{\rm op}\asymp N^{-\Delta/2}$ for
$\Delta>0$.

The cut function $\delta$ is additive on
connected components of an induced tensor set.
Since each nonempty component has gap at least
$\Delta>0$, a minimizing set is connected.
Its $k_T$ is zero or one according as it has
an input or not. Every nonminimizer has integer
gap at least $\Delta+1$, whose additional
$N^{-1/2}$ absorbs any fixed logarithmic power.
Consequently, if $\kappa_0=1$ when a minimizing
set has no input and $\kappa_0=0$ otherwise,
\[
 \Prob\{\|F_N\|_{\rm op}>
 C_A N^{-\Delta/2}(\log(2N))^{\kappa_0/2}\}
 \le N^{-A}
\]
for every fixed $A>0$. This is only an upper tail.

\subsection{The sample-state endpoint}

Let $s$ be the ordinary minimum input-output
edge cut. The original network has no unused
middle role or detached component, and its
minimum lifted cuts have zero active count.
Thus the Gaussian uniform bound gives
$\|H_N\|_{\rm op}\le C_AN^{(r-s)/2}$ with
failure probability at most $N^{-A}$.

Put $X=N^{-r}\|H_N\|_F^2$. Covariance gives
$\E X=1$. In its second moment a swapped
tensor set $T$ loses one free label at every
crossing internal edge and every incident
open leg. Hence
\[
 \E X^2=\sum_{T\subseteq[Q]}N^{-\partial(T)},
 \qquad \partial(T)=c(T)+a(T)+b(T).
\]
Every nonempty $T$ has $\partial(T)\ge1$,
since every network component meets a terminal.
Therefore $\operatorname{Var}X\le(2^Q-1)/N$
and $\Prob\{|X-1|>1/2\}\le4(2^Q-1)/N$.

The map is nonzero almost surely: some entry
is a nonzero polynomial in Gaussian coordinates,
as seen by evaluating all coordinates at one.
Cut factorization gives $\rank H_N\le N^s$.
On the preceding event and the operator upper
event with $A=2$,
\[
 N^{-s}\le\|\rho_A\|_{\rm op}
       =\frac{\|H_N\|_{\rm op}^2}{\|H_N\|_F^2}
       \le C N^{-s}.
\]
Since $H_\infty\le H_\alpha\le H_0$ for all
nonnegative R\'enyi orders, including endpoints,
this proves
$s\log N-\log C\le H_\alpha(\rho_A)\le s\log N$
simultaneously for $0\le\alpha\le\infty$ with
probability $1-O(N^{-1})$.
Normalizing each Gaussian tensor by its own
nonzero norm leaves $\rho_A$ unchanged, but
does not preserve mean-Gram normalization.

\section{Complete proof of the exact RTN spectral edge}\label{app:edge-detail}

This appendix proves \cref{thm:rtn-exact-edge,thm:edge-transfer}.
The quotient and comparison are established before the external
strong-convergence theorem is used. Concentration and exceptional-event
moment control are then proved separately, so that the moment order
may grow with the original dimension.

\subsection{Exact Wick moments and the flow quotient}\label{app:edge-wick-proof}

\begin{lemma}[Positive moment polynomial]\label{lem:edge-wick}
For every integer $p,N\ge1$,
\begin{equation}\label{eq:edge-wick}
 F_G(p,N):=\E[N^{-s}\tr Y_N^p]
       =\sum_{\delta\ge0}A_G(p,\delta)N^{-\delta}.
\end{equation}
All coefficients are nonnegative integers.
\end{lemma}
\begin{proof}
Wick's rule at vertex $u$ pairs its $p$ unconjugated occurrences
with its $p$ conjugated occurrences by $\pi_u$.
An edge $\{u,w\}$ then has
$p-d(\pi_u,\pi_w)$ free index loops. Hence
\[
 \E\tr(H_NH_N^*)^p
 =\sum_\pi N^{pR-\mathcal H_G^{(p)}(\pi)}.
\]
Coincident coordinate values retain their Wick multiplicities;
there is no injective-label restriction here.
Choose $s$ edge-disjoint simple terminal paths by integral max flow.
Triangle inequality on each path gives
$\mathcal H_G^{(p)}\ge s(p-1)$.
Assigning $\id$ and $\gamma_p$ on the two sides of a minimum cut
attains equality. The normalization in
\cref{eq:edge-normalizations} proves the identity.
The opposite cyclic-trace convention gives $\gamma_p^{-1}$;
inverting all labels preserves the distances and counts.
\end{proof}

Fix such a family of $s$ paths and orient their union $F$ forward.
Form a directed graph $D$ with one forward arc for each used edge
and both arcs for each unused edge. Contract its complete strongly
connected components (SCCs), discard internal edges, and retain
cross-component edges with multiplicity. Denote this quotient by $Q$
and its terminal components by $B_0,A_0$.

\begin{lemma}[Balanced acyclic quotient]\label{lem:edge-quotient}
$Q$ is an acyclic directed multigraph, $B_0\ne A_0$, and every
quotient vertex lies on a projected terminal flow path.
All cross-component edges are used edges. The projected paths remain
edge-disjoint, and
\[
 \deg^-(B_0)=0,\quad\deg^+(B_0)=s,\qquad
 \deg^-(A_0)=s,\quad\deg^+(A_0)=0.
\]
Each other quotient vertex $C$ has
$\deg^-(C)=\deg^+(C)=k_C\ge1$. The minimum terminal cut of $Q$ is $s$.
\end{lemma}
\begin{proof}
An unused edge is bidirectional, so its endpoints belong to the
same SCC. A condensation is acyclic.
Across a cut of size $s$, each of the $s$ edge-disjoint paths
crosses exactly once: all cut edges are used and directed toward
the sink side. No arc crosses back. Thus the terminal components
are distinct.

A vertex off the flow can be joined to a first flow vertex
by a path consisting only of unused edges, using connectedness.
Its SCC therefore meets the flow. A directed path cannot leave
an SCC and return, since this would create a cycle in the
condensation. Thus projected paths are simple and cover every
quotient vertex. Summing flow conservation inside each component
cancels internal used edges. This gives balanced external degrees
away from the terminal components. Since every component lies on
a source--sink path in a DAG, the source has no incoming edge and
the sink no outgoing edge. Their degrees are consequently $s$.
The projected flow and the cut immediately after $B_0$ give
the matching lower and upper cut bounds.
\end{proof}

Contracting only the unused-edge connected components would not
suffice: used edges can create directed cycles between those
components. The complete SCC construction is used throughout.
Choose one anchor $r(C)$ per SCC, using the appropriate terminal
in $B_0,A_0$, and an actual vertex otherwise. Define the single
global map
\begin{equation}\label{eq:edge-anchor}
 \omega_C=\pi_{r(C)},\qquad \bar\pi_u=\omega_{C(u)} .
\end{equation}

\subsection{All-defect comparison at an auxiliary dimension}

\begin{lemma}[Variation inside an SCC]\label{lem:edge-variation}
If $\Delta_G(\pi)=\delta$ and $C(u)=C(w)$, then
$d(\pi_u,\pi_w)\le2\delta$.
\end{lemma}
\begin{proof}
Write $r_u=|\pi_u|$ and assign to each directed arc $u\to w$
the nonnegative reduced cost
$c_{uw}=d(\pi_u,\pi_w)-(r_w-r_u)$.
Flow conservation yields
\[
 \delta=\sum_{e\in F}c_e+\sum_{e\notin F}d_e,\qquad
 \sum_{\text{arcs of }D}c_e
   =\delta+\sum_{e\notin F}d_e\le2\delta.
\]
For unused edges the two arc costs sum to $2d_e$.
Choose simple directed paths $P:u\to w$ and $P':w\to u$
inside the SCC. Triangle inequality and telescoping give
\[
 d(\pi_u,\pi_w)\le r_w-r_u+c(P),\qquad
 d(\pi_u,\pi_w)\le r_u-r_w+c(P').
\]
Each path cost is at most $2\delta$; their edges need not be
disjoint. Adding proves the assertion.
\end{proof}

Define $A_Q,\mathcal H_Q,\Delta_Q$ using the same terminal labels
and minimum cut $s$, and let $C_0=1+4R$.
\begin{lemma}[Defect and fiber bounds]\label{lem:edge-fiber}
The map \cref{eq:edge-anchor} satisfies
$0\le\Delta_Q(\omega)\le C_0\delta$.
For $p\ge2$, each fiber at fixed $\omega,\delta$ has at most
$p^{4v\delta}$ original labelings. Consequently
\begin{equation}\label{eq:edge-count-comparison}
 A_G(p,\delta)\le p^{4v\delta}
                 \sum_{j=0}^{C_0\delta}A_Q(p,j),\qquad
 A_G(p,0)=A_Q(p,0).
\end{equation}
\end{lemma}
\begin{proof}
Deleted edges have equal labels after applying the map, so
$\mathcal H_Q(\omega)=\mathcal H_G(\bar\pi)$.
The terminal labels are unchanged, and
\[
 \mathcal H_G(\bar\pi)
 \le\mathcal H_G(\pi)+
       \sum_{u\in V_G}\deg_G(u)d(\pi_u,\bar\pi_u)
 \le\mathcal H_G(\pi)+4R\delta.
\]
Here we used \cref{lem:edge-variation} and
$\sum_{u\in V_G}\deg_G(u)\le2R$.
The quotient cut bound gives nonnegative defect.

A radius-$2\delta$ ball in $S_p$ has size at most
\[
 \sum_{h=0}^{2\delta}\binom p2^{h}
 \le\left(1+\binom p2\right)^{2\delta}
 \le p^{4\delta}.
\]
For $\delta=0$ the size is one. Every original label lies
in this ball about its anchor. Multiplying over $v$ vertices
is a bound on possible choices, not an assertion of independence.
At defect zero, all labels inside an SCC agree. Conversely,
every zero-defect quotient labeling lifts uniquely to a
component-constant labeling with the same energy.
\end{proof}

We now prove the auxiliary-dimension theorem stated in
\cref{sec:rtn-edge}. Its notation is the positive moment polynomial of
\cref{lem:edge-wick} and the quotient just constructed.
\label{app:edge-transfer-proof}

\begin{proof}
By positivity of the quotient coefficients,
\[
 \sum_{j\le C_0\delta}A_Q(p,j)
 \le M^{C_0\delta}\sum_{j\ge0}A_Q(p,j)M^{-j}
 =M^{C_0\delta}F_Q(p,M).
\]
Insert this in \cref{eq:edge-count-comparison}, multiply by
$N^{-\delta}$, and extend the finite sum to an infinite
geometric series with ratio $p^{4v}M^{C_0}/N$.
Taking $M=p^2$ gives the second bound.
\end{proof}

The quotient comparison retains its positive-defect states instead
of projecting onto zero-defect states alone. It evaluates their
entire positive moment polynomial at a new dimension. This is the
point at which a separately sampled Gaussian quotient can be used
without changing the distributional assumptions on the original RTN.

\subsection{The overlapping Ginibre circuit and its strong limit}

Label the flow paths $1,\ldots,s$. For each internal quotient
vertex $C$, let $K_C\subseteq[s]$ be the paths through it;
$|K_C|=k_C$. Order these vertices topologically as
$C_1,\ldots,C_t$. Independently sample normalized square Ginibre
matrices $G_C^{(M)}$ of size $d_C=M^{k_C}$, with entry variance
$d_C^{-1}$. Embed each gate on tensor factors $K_C$ of
$(\C^M)^{\otimes s}$, using identity on the other factors, and set
\begin{equation}\label{eq:edge-circuit}
 B_M=\widehat G_{C_t}^{(M)}\cdots\widehat G_{C_1}^{(M)} .
\end{equation}
\begin{lemma}[Exact circuit realization]\label{lem:edge-circuit}
Let $H_{Q,M}$ be the raw independent-Gaussian contraction of $Q$,
with direct terminal edges interpreted as identity wires. In the
wire-labeled coordinates,
\[
 B_M=M^{-\frac12\sum_C k_C}H_{Q,M},\qquad
 Y_{Q,M}=B_MB_M^*,\qquad
 F_Q(p,M)=\E[M^{-s}\tr(B_MB_M^*)^p].
\]
\end{lemma}
\begin{proof}
An internal quotient tensor has $k_C$ inputs and $k_C$ outputs.
Its matricization is a square iid Gaussian matrix. Path labels
give consistent input and output tensor factors, and topological
contraction is exactly multiplication of the embedded gates.
Different boundary ordering contributes only permutation unitaries.
Counting edges at their tails gives $R_Q=s+\sum_C k_C$,
so $M^{s-R_Q}$ is the product of the squared gate normalizations.
The Wick formula also holds for direct identity wires, whose
terminal permutations give their usual index-loop factors.
\end{proof}

If $t=0$, $B_M=I$, $F_Q(p,M)=1$, and $E_G=1$.
Suppose henceforth that $t>0$. A normalized complex Ginibre gate
has the distribution
\[
 G_C^{(M)}=(S_{C,1}^{(M)}+iS_{C,2}^{(M)})/\sqrt2,
\]
where the $2t$ matrices $S$ are independent normalized GUE
matrices of the indicated base dimensions. The covariance of
this array is
$\E G_{ij}\overline{G_{kl}}=d_C^{-1}\one_{i=k,j=l}$,
and its pseudocovariance is zero, proving the asserted iid
circular Gaussian law.

In \citet[Section~9.4, Theorem~9.8]{ChenGarzaVargasVanHandel2026},
take the number of tensor factors to be $s$, the number of GUE
variables to be $2t$, their supports to be
$K_{C_1},K_{C_1},\ldots,K_{C_t},K_{C_t}$, and the coefficient
dimension to be one. These supports are nonempty; the theorem
does not require them to be distinct or disjoint.
The polynomial is the ordered product in \cref{eq:edge-circuit},
of fixed degree $t$, independent of $p$ and $M$.
We obtain a limit $b$ in a $C^*$-probability space with faithful
trace $\tau$ such that
\begin{equation}\label{eq:edge-strong}
 \norm{B_M}\longrightarrow\norm b=:L
 \quad\text{almost surely}.
\end{equation}
The weak moment limit of \citet[Theorem~4]{CharlesworthCollins2021},
also recalled directly before that strong-convergence theorem,
gives, for fixed $p$,
\begin{equation}\label{eq:edge-limit-moments}
 \lim_{M\to\infty}F_Q(p,M)
   =\tau((bb^*)^p)=A_Q(p,0)=A_G(p,0).
\end{equation}
Expected moments are justified directly by that weak-limit
interface, or by uniform integrability: the gate moment bound
proved below implies
\[
 \sup_{M\ge1}\E\!\left[
    \bigl(M^{-s}\tr(B_MB_M^*)^p\bigr)^2\right]
 \le [C(1+\sqrt{4p})]^{4pt}<\infty
\]
for fixed $p$. This bound does not require the growing-moment
conclusion.

The scalar spectral measure of $bb^*$ has compact support and
therefore is uniquely determined by these moments. It is $\mu_G$.
Faithfulness implies that its support is the spectrum of $bb^*$:
a nonzero nonnegative continuous function of $bb^*$ has positive
trace. It follows that
\begin{equation}\label{eq:edge-base}
 E_G=\norm{bb^*}=L^2 .
\end{equation}
Also $m_1=1$, so $E_G\ge1$.

\subsection{From strong convergence to growing moments}

Qualitative norm convergence does not by itself control growing
moments. We supply both concentration and control of the size
of the exceptional contribution.

\begin{lemma}[Gate tails and moments]\label{lem:edge-gate-tail}
For a normalized $d$ by $d$ complex Ginibre matrix $G$, universal
constants $c,C>0$ satisfy
\[
 \Prob\{\norm G>16\}\le e^{-cd},\qquad
 (\E\norm G^r)^{1/r}\le C(1+\sqrt{r/d})\quad(r\ge1).
\]
\end{lemma}
\begin{proof}
A $1/4$-net of the complex unit sphere has at most $9^{2d}$
points. Two such nets give $\norm G\le2\max_{u,w}|w^*Gu|$.
Each scalar in this maximum is circular Gaussian of variance
$1/d$. Thus
\[
 \Prob\{\norm G>x\}
 \le\min\{1,\exp(4d\log9-dx^2/4)\}.
\]
For $x\ge16$ the second term is at most $\exp(-dx^2/8)$.
Integrating the tail of $(\norm G-16)_+$ gives
$\E[(\norm G-16)_+^r]\le(8/d)^{r/2}\Gamma(r/2+1)$.
The bound $\Gamma(r/2+1)^{1/r}\le C\sqrt r$ and Minkowski's
inequality prove the moment estimate.
\end{proof}

\begin{lemma}[Concentration about the limiting norm]\label{lem:edge-concentration}
For each $\varepsilon>0$ there are positive constants
$c_\varepsilon,C_\varepsilon$ and $M_0$, depending on $Q,\varepsilon$,
such that
\[
 \Prob\{\norm{B_M}>L+\varepsilon\}
 \le C_\varepsilon e^{-c_\varepsilon M}\qquad(M\ge M_0).
\]
\end{lemma}
\begin{proof}
Write gate entries as $(\xi+i\eta)/\sqrt{2d_C}$ in independent
standard real Gaussian coordinates. Let $\Omega_M$ be the set
on which every gate norm is at most 16.
Since $d_C\ge M$, \cref{lem:edge-gate-tail} gives
$\Prob(\Omega_M^c)\le t e^{-c_0M}$.
In particular $\Prob(\Omega_M)\ge3/4$ once
$M\ge\max\{1,\lceil c_0^{-1}\log(4t)\rceil\}$.

For any two coordinate vectors in $\Omega_M$, telescope their
gate products. A gate difference is bounded in operator norm
by its base-matrix Frobenius norm; embedding with identity
factors preserves operator norm. Cauchy--Schwarz over gates gives
the pairwise Lipschitz constant
\[
 \ell_M=16^{t-1}\left(\sum_C(2d_C)^{-1}\right)^{1/2}
       \le16^{t-1}\sqrt{t/(2M)}
\]
for $f_M=\norm{B_M}$ on $\Omega_M$. No convexity of
$\Omega_M$ is required, and no ambient $M^s$ Frobenius factor
is introduced.

The restricted Gaussian concentration inequality
\citep[Lemma~5.6]{ChenGarzaVargasVanHandel2026} yields
\[
 \Prob\{|f_M-\operatorname{med}f_M|>u\}
 \le\Prob(\Omega_M^c)+C e^{-cu^2/\ell_M^2}.
\]
By \cref{eq:edge-strong}, every choice of median tends to $L$.
For all sufficiently large $M$, take $u=\varepsilon/2$.
Combining the two exponential bounds proves the assertion.
\end{proof}

\subsection{The analytic bridge at an auxiliary dimension}
\label{app:positive-growing}

The preceding concentration estimate is one input to the following
general criterion. The other is high-moment control of the exceptional
contribution. We state the criterion before verifying both inputs for
the comparison circuit.

\begin{lemma}[From reference concentration to growing moments]
\label{lem:interface-growing}
Let $Z_M\ge0$ and $L\ge0$. Suppose that for every $\eps>0$ there
are $C_\eps,c_\eps>0$ and $M_\eps$ such that, for a fixed $\beta>0$,
\[
 \Prob\{Z_M>L+\eps\}\le C_\eps e^{-c_\eps M^\beta}
 \qquad(M\ge M_\eps).
\]
Choose integers $M_p$ with $M_p^\beta/p\to\infty$ and suppose
$\sup_{p\ge p_0}\Lpnorm{Z_{M_p}}{4p}\le D<\infty$.
Then for every $\eta>0$ and all sufficiently large $p$,
\begin{equation}\label{eq:interface-growing}
 \E Z_{M_p}^{2p}\le2(L^2+\eta)^p.
\end{equation}
\end{lemma}
\begin{proof}
Choose $\eps>0$ with $(L+\eps)^2<L^2+\eta$. Below the threshold
$L+\eps$, the contribution is at most $(L+\eps)^{2p}$.
Cauchy--Schwarz bounds the remaining contribution by
$D^{2p}C_\eps^{1/2}\exp(-c_\eps M_p^\beta/2)$.
Its $p$-th root tends to zero, so each contribution is eventually
at most $(L^2+\eta)^p$.
\end{proof}

\begin{corollary}[Logarithmic-order transfer with the right base]
\label{cor:interface-log}
Under \cref{thm:positive-fiber}, suppose $r_p\le p^b$ for fixed
$b\ge0$. Choose $M_p=\lceil p^a\rceil$ with fixed $a>0$ and
$a\beta>1$. If $F_\Lambda(p,M_p)\le\E Z_{M_p}^{2p}$ for a family
satisfying \cref{lem:interface-growing}, then for every $\eta>0$
and all sufficiently large $p$,
\[
 F_\Omega(p,N)\le
 \frac{2(L^2+\eta)^p}{1-p^b\lceil p^a\rceil^\kappa/N}
 \quad\text{whenever the denominator is positive}.
\]
The denominator tends to one when $p\asymp\log N$. The coefficient
$a_{p,\delta}$ is bounded by the same numerator times
$p^{b\delta}\lceil p^a\rceil^{\kappa\delta}$.
\end{corollary}
\begin{proof}
Insert \cref{eq:interface-growing} into \cref{eq:positive-fiber}
and its coefficient bound. A fixed power of $\log N$ is $o(N)$.
\end{proof}

The high-moment assumption is not a consequence of qualitative strong
convergence. Let $Z_M=e^{\sqrt M}$ with probability $e^{-M}$ and
$Z_M=1$ otherwise. The event probabilities are summable, so in any
common coupling $Z_M\to1$ almost surely. For fixed $k$,
\[
 \E Z_M^k=1-e^{-M}+e^{-M+k\sqrt M}\longrightarrow1,
 \qquad \E Z_{p^2}^{2p}\ge e^{p^2}.
\]
This family even has exponentially rare upper deviations. It shows
that fixed moments and small exceptional probability alone do not
control growing moments. The uniform $4p$-moment bound is a sufficient
way to control that contribution here, not a claim that it is the
only possible condition for such a transfer.

For the RTN, take $Z_M=\norm{B_M}$, $\beta=1$, $a=2$,
$b=4v$, and $\kappa=1+4R$. The normalized trace is bounded by the
operator norm moment, and $L^2=E_G$. The following proof supplies
the uniform $4p$ bound and gives the original quantitative statement
with all model constants retained.

\begin{lemma}[Auxiliary growing moments]\label{lem:edge-growing}
For every $\eta>0$, there is $p_0(Q,\eta)$ such that
\begin{equation}\label{eq:edge-growing}
 F_Q(p,p^2)\le2(E_G+\eta)^p\qquad(p\ge p_0).
\end{equation}
Consequently, for every $\delta\ge0$ and the same threshold,
\begin{equation}\label{eq:edge-all-defects}
 A_G(p,\delta)\le2(E_G+\eta)^p p^{K_G\delta}.
\end{equation}
\end{lemma}
\begin{proof}
For $p\le M$, independence and \cref{lem:edge-gate-tail} give
\[
 \E\norm{B_M}^{4p}
 \le\prod_C\E\norm{G_C^{(M)}}^{4p}\le D_Q^{4p}
\]
with $D_Q$ independent of $M,p$.
Fix $\varepsilon>0$. On the exceptional event of
\cref{lem:edge-concentration}, apply Cauchy--Schwarz to obtain
\begin{equation}\label{eq:edge-exceptional}
 F_Q(p,M)\le\E\norm{B_M}^{2p}
 \le(L+\varepsilon)^{2p}
       +D_Q^{2p}C_\varepsilon^{1/2}e^{-c_\varepsilon M/2}.
\end{equation}
Fix $\eta$ first and choose $\varepsilon$ so that
$(L+\varepsilon)^2<E_G+\eta$. At $M=p^2$, the logarithm
of the second term divided by $p$ is
$2\log D_Q+(\log C_\varepsilon)/(2p)-c_\varepsilon p/2$,
which tends to $-\infty$. This proves
\cref{eq:edge-growing}. Combining it with
\cref{eq:edge-count-comparison} and positivity at $M=p^2$
gives \cref{eq:edge-all-defects}, uniformly in $\delta$.
When $t=0$ the same assertions follow directly from $F_Q=1$.
\end{proof}

This estimate preserves the exact exponential base up to any
fixed $\eta>0$. It does not assert the stronger relative bound
$A_G(p,\delta)\le A_G(p,0)p^{C_G\delta}$.

\subsection{The upper edge, fixed moments, and normalization}

Fix $\varepsilon>0$. By \cref{thm:edge-transfer,lem:edge-growing},
\[
 \E[N^{-s}\tr Y_N^p]\le
 \frac{2(E_G+\varepsilon/2)^p}{1-p^{K_G}/N}
\]
for all sufficiently large $p$ satisfying $p^{K_G}<N$.
Choose
\[
 p_N=\lceil c_\varepsilon\log N\rceil,\qquad
 c_\varepsilon>
 \frac{s}{\log((E_G+\varepsilon)/(E_G+\varepsilon/2))}.
\]
Then $p_N^{K_G}/N\to0$ and all fixed large-parameter thresholds
hold eventually. Positivity and Markov's inequality give
\begin{equation}\label{eq:edge-upper}
 \Prob\{\norm{Y_N}>E_G+\varepsilon\}
 \le\frac{2N^s}{1-p_N^{K_G}/N}
 \left(\frac{E_G+\varepsilon/2}{E_G+\varepsilon}\right)^{p_N}
 \longrightarrow0 .
\end{equation}
The trace dimension $N^s$ is essential in this choice of $p_N$.

For the lower bound it suffices to use fixed moments. We give
a direct Wick proof, including a useful covariance identity.
Write $T_{N,p}=N^{-s}\tr Y_N^p$.
\begin{lemma}[Nonnegative covariance polynomial]\label{lem:edge-covariance}
For fixed positive integers $p,q$,
\[
 \operatorname{Cov}(T_{N,p},T_{N,q})
   =\sum_{\delta\ge1}B_G(p,q;\delta)N^{-\delta},
 \qquad B_G(p,q;\delta)\in\mathbb Z_{\ge0},
\]
and
\begin{equation}\label{eq:edge-fixed-L2}
 \E[(T_{N,p}-m_p)^2]
 \le\frac{((2p)!)^v}{N}+\frac{(p!)^{2v}}{N^2}.
\end{equation}
\end{lemma}
\begin{proof}
Use labels in $S_{p+q}$, terminal identity at $\mathsf b$,
and
$\Gamma=(1\,\cdots\,p)(p+1\,\cdots\,p+q)$ at $\mathsf a$.
The mixed Wick sum is
\[
 \E[T_{N,p}T_{N,q}]
 =\sum_{\pi\in S_{p+q}^{V_G}}
       N^{s(p+q-2)-\mathcal H_G^\Gamma(\pi)}.
\]
All labelings preserving the two trace blocks split into
an $S_p$ labeling and an $S_q$ labeling. Their energies add,
so their entire contribution, including positive defects,
is $F_G(p,N)F_G(q,N)$.

No block-mixing labeling has zero defect. Indeed, at minimum
energy every unused edge has equal labels and each flow
path is geodesic from $\id$ to $\Gamma$. Every vertex connects
to a flow vertex through unused edges. Thus each label lies
in $[\id,\Gamma]$ in absolute order. If $\sigma$ lies in
that interval, concatenate reduced transposition words for
$\sigma$ and $\sigma^{-1}\Gamma$. The resulting word has
length $|\Gamma|$. Starting at identity, every step must
merge two cycles; a splitting step would prevent the total
cycle-count decrease from equaling the word length. A merge
across the two final cycle supports could never be undone.
Hence $\sigma$ preserves both supports.

Subtracting the block-preserving contribution leaves positive
terms with integer defects at least one. There are at most
$((p+q)!)^v$ terms, giving the asserted covariance polynomial
and the variance bound $((2p)!)^v/N$.
The ordinary Wick sum also gives
$0\le\E T_{N,p}-m_p\le(p!)^v/N$.
Adding squared bias proves \cref{eq:edge-fixed-L2}.
\end{proof}

Because $\rank Y_N\le N^s$, $T_{N,p}\le\norm{Y_N}^p$.
Compactness and positivity of $\mu_G$ imply $m_p^{1/p}\to E_G$.
For $0<\varepsilon<E_G$, choose a fixed $p$ with
$m_p>(E_G-\varepsilon)^p$. Then \cref{eq:edge-fixed-L2} gives
\[
 \Prob\{\norm{Y_N}\le E_G-\varepsilon\}
 \le\Prob\{T_{N,p}\le(E_G-\varepsilon)^p\}\longrightarrow0.
\]
For $\varepsilon\ge E_G$ the lower estimate is trivial.
Together with \cref{eq:edge-upper} this proves the first
convergence in \cref{thm:rtn-exact-edge}.

Finally, $\E Z_N=1$. In the second moment of $Z_N$, each
vertex has either the identity or the transposition label
in $S_2$, and both terminal labels are identity. If $T$
is the transposed vertex set, its term is $N^{-|\partial T|}$.
Thus
\begin{equation}\label{eq:edge-trace-normalization}
 \Var Z_N=\sum_{\varnothing\ne T\subseteq V_G}
                  N^{-|\partial T|}
          \le(2^v-1)/N.
\end{equation}
Every nonempty $T$ has a boundary by connectedness, including
$T=V_G$. Some entry of $H_N$ is a nonzero polynomial in the
Gaussian coordinates, so $Z_N>0$ almost surely.
Slutsky's theorem applied to $X_N=Y_N/Z_N$ proves the second
convergence. Since $E_G\ge1$, continuity of the logarithm gives
\cref{eq:edge-min-entropy}, completing the proof.

\section{An independent Wishart calibration}\label{app:wishart}

Let $X$ be $P$ by $d$ with independent standard
circular complex coordinates, $P\ge d$, and
$F=X^*X/P-I_d$. For $\gamma=(1\,\cdots\,k)$,
direct Wick expansion gives
\[
 \E\tr(X^*X/P)^k
       =P^{-k}\sum_{\pi\in\mathfrak S_k}
                      P^{\#\pi}d^{\#(\gamma\pi)} .
\]
Pairing conjugate position $s$ with unconjugated
position $\pi(s)$ identifies row labels along
cycles of $\pi$ and column labels along cycles
of $\gamma\pi$. Coinciding labels do not erase
pairing multiplicities.

For $k=1,2,3,4$, this yields
\[
 \begin{aligned}
 \E\tr W&=d,\\
 \E\tr W^2&=d+d^2/P,\\
 \E\tr W^3&=d+3d^2/P+(d^3+d)/P^2,\\
 \E\tr W^4&=d+6d^2/P+(6d^3+5d)/P^2
                                      +(d^4+5d^2)/P^3,
 \end{aligned}
\]
where $W=X^*X/P$. Subtracting the identity gives
\[
 \E\tr F^2=d^2/P,\qquad
 \E\tr F^4=(2d^3+d)/P^2+(d^4+5d^2)/P^3.
\]
For a uniform eigenvalue index let
$Y=|\lambda_j(F)|$ and $u=d/P$. Then
$\E Y^2=u$ and $\E Y^4\le9u^2$.
Interpolation and Cauchy--Schwarz give
\begin{equation}\label{eq:wishart-check}
 \frac13d\sqrt{d/P}\le\E\|F\|_1\le d\sqrt{d/P},
 \qquad
 \Prob\{\|F\|_1\ge d\sqrt{d/P}/6\}\ge1/36 .
\end{equation}
For the probability bound use
$(d^{-1}\|F\|_1)^2\le d^{-1}\tr F^2$ and
Paley--Zygmund. These finite-dimensional bounds
hold uniformly over every $P\ge d$, without
a fixed-aspect limit. At $P=N^a,d=N^b$, they
confirm the trace threshold $a>3b$.
For $a<b$, rank already excludes convergence.

\section{Proof of the sparse phase diagram}\label{app:sparse}

This appendix proves \cref{prop:sparse-phase} for its specified iid
family. Let $D_i=\sum_j B_{ij}$, $D_{\max}=\max_iD_i$, and
$\lambda=n\rho$. The $h=1$ case of \citet[Theorem~1.1]{Seginer2000}
and the deterministic maximum-row-norm lower bound imply
\begin{equation}\label{eq:sparse-seginer}
 \E\norm X\asymp \rho^{-1/2}\E\sqrt{D_{\max}}.
\end{equation}
Indeed, each row norm is $\rho^{-1/2}\sqrt{D_i}$, and the maximum
column norm has the same distribution as the maximum row norm.
The comparison constant is independent of dimension and distribution.
This is the only external norm theorem used in the proof. We now
estimate the expectation on the right, rather than only a typical value.

For every integer $k\ge1$, a binomial union bound gives
\begin{equation}\label{eq:sparse-binomial}
 \Prob\{D_{\max}\ge k\}\le n(e\lambda/k)^k.
\end{equation}
Also, the exponential-moment estimate
$\E e^{D_i}=(1+\rho(e-1))^n\le\exp(\lambda(e-1))$ gives
\[
 \Prob\{D_{\max}\ge k\}\le n\exp(\lambda(e-1)-k).
\]

\paragraph{Polynomially growing row occupancy: $0\le\beta<1$.}
Here $\lambda=n^{1-\beta}$ grows polynomially. For
$k_0=\lceil C\lambda\rceil$ with a sufficiently large fixed $C$,
the last bound has a negligible geometric tail after $k_0$.
Consequently $\E D_{\max}=O(\lambda)$ and Jensen gives
$\E\sqrt{D_{\max}}=O(\sqrt\lambda)$.
One row has mean $\lambda$ and variance at most $\lambda$, so
Chebyshev's inequality gives the matching lower bound. At $\beta=0$
each row count is deterministically $n$, with the same conclusion.

\paragraph{Critical occupancy: $\beta=1$.}
Now $\lambda=1$. Put $k_0=\lceil3\log n/\log\log n\rceil$.
The bound $n(e/k)^k$ is $o(1)$ at $k_0$, and its ratio at successive
integers is at most $e/(k+1)$. Summing the tail yields
$\E D_{\max}=O(\log n/\log\log n)$.
For the lower bound, fix $0<\delta<1$ and set
$k=\lfloor(1-\delta)\log n/\log\log n\rfloor$. Since $k=o(\sqrt n)$,
\[
 \Prob\{D_i=k\}
 =\binom nk n^{-k}(1-1/n)^{n-k}
 =\exp(-k\log k+O(k)+o(1)).
\]
Thus $n\Prob\{D_i=k\}=n^{\delta+o(1)}\to\infty$.
The rows are independent, so $D_{\max}\ge k$ with probability
tending to one. Jensen supplies the matching upper bound for
$\E\sqrt{D_{\max}}$.

\paragraph{Sparse but nonempty: $1<\beta<2$.}
Choose a fixed integer $k_0>2/(\beta-1)$. The first term of
\cref{eq:sparse-binomial} at $k_0$ tends to zero, and successive
terms have ratio at most $e\lambda/(k+1)=o(1)$. Thus
$\E D_{\max}=O(1)$. There is at least one occupied entry with
probability $1-(1-\rho)^{n^2}\to1$, proving
$\E\sqrt{D_{\max}}\asymp1$.

\paragraph{Bounded total occupancy: $\beta=2$.}
Let $N=\sum_{i,j}B_{ij}$. Then $\E N=1$ and
$\E\sqrt{D_{\max}}\le\E\sqrt N\le1$.
On the other hand $\Prob\{N=1\}\to e^{-1}$, giving a constant
lower bound.

\paragraph{Rare occupancy: $\beta>2$.}
Here $\mu=\E N=n^2\rho\to0$. Pointwise
$\sqrt{D_{\max}}\le N$, so its expectation is at most $\mu$.
Exactly one entry is occupied with probability
$\mu(1-\rho)^{n^2-1}\sim\mu$, yielding the reverse bound.
Multiplication by $\rho^{-1/2}$ in \cref{eq:sparse-seginer} gives
all regimes of \cref{eq:sparse-phase}.

\clearpage
\bibliographystyle{plainnat}
\bibliography{references}
\end{document}